\documentclass[11pt]{amsart}
\usepackage{amssymb,upgreek}
\usepackage{enumerate,titletoc,mathrsfs,graphicx}

\usepackage[usenames,dvipsnames]{xcolor}
\definecolor{darkblue}{rgb}{0.0, 0.0, 0.55}
\usepackage[pagebackref,colorlinks,linkcolor=BrickRed,citecolor=OliveGreen,urlcolor=darkblue,hypertexnames=true]{hyperref}

\DeclareMathOperator{\RE}{Re}

\DeclareMathOperator\id{id}

\DeclareMathOperator{\im}{im} 
\DeclareMathOperator{\area}{area}
\DeclareMathOperator{\sign}{sign} 

\DeclareMathOperator{\IM}{Im}

\DeclareMathOperator\sgn{sgn}

\renewcommand{\setminus}{\smallsetminus}
\renewcommand{\subset}{\subseteq}

\renewcommand{\qedsymbol}{\rule[.12ex]{1.2ex}{1.2ex}}

\def\eps{\varepsilon}

\newcommand\ph\varphi
\newcommand\ps\psi
\newcommand\ep\varepsilon
\newcommand\rh\varrho
\newcommand\al\alpha
\newcommand\be\beta
\newcommand\ga\gamma
\newcommand\om\omega
\newcommand\ta\tau
\renewcommand\th\vartheta
\newcommand\de\delta
\newcommand\ze\zeta
\newcommand\ch\chi
\newcommand\et\eta
\newcommand\io\iota
\newcommand\la\lambda
\newcommand\si\sigma
\newcommand\ka\kappa

\newcommand\Ga\Gamma
\newcommand\De\Delta
\newcommand\Th\Theta
\newcommand\La\Lambda
\newcommand\Si\Sigma
\newcommand\Ph\Phi
\newcommand\Ps\Psi
\newcommand\Om\Omega

\def\cB{\mathcal B}
\def\cC{\mathcal C}

\def\bes{\begin{equation*}}
\def\ees{\end{equation*}}

\def\ben{\begin{enumerate} }
\def\een{\end{enumerate} }
\def\benum{\begin{enumerate} }
\def\eenum{\end{enumerate} }

\def\bmat{\left[\begin{array}{ccccccccccccccc} }
\def\emat{\end{array}\right]}

\def\bmat{\begin{bmatrix}}
\def\emat{\end{bmatrix}}
\def\beq{\begin{equation}}
\def\eeq{\end{equation}}

\def\barr{\begin{array}}
\def\earr{\end{array}}

\def\cB{\mathcal{B}}

\def\La{\Lambda}

\def\TM{M}

\def\NN{\mathbb N}
\def\N{\mathbb N}

\def\isp{\rotatebox[origin=c]{180}{$\psi$}}

\def\smatn{\mathbb S_n}

\def\cB{ {\mathcal B} }
\def\cC{ {\mathcal C} }
\def\cD{ {\mathcal D} }

\def\cH{ {\mathcal H} }
\def\cI{ {\mathcal I} }

\def\cK{ {\mathcal K} }

\def\cM{{\mathcal M}}

\def\cT{{\mathcal T}}

\def\la{\lambda}

\def\tL{{\tilde L}}

\newcommand\smat{\mathbb S}
\def\smatng{\mathbb S_n^g} 

\def\smatd{\mathbb S_d}
\def\smatdg{\mathbb S_d^g}

  \def\cDA{{\cD_{L_A}}}
  \def\cDB{{\cD_{L_B}}}

\def\j1tog{ j= 1, \ldots, g }

\def\cI{{\mathcal I}}

\def\tA{\tilde{A}}

\def\L0t{\ (L_0 \otimes I_n ) \ }

\def\PLINE1{\beta}

\DeclareMathOperator{\trace}{tr}

\def\bE{\mathbf{e}}

\def\sssec{\subsubsection}
\def\ssec{\subsection}

\def\bem{\begin{pmatrix} }
\def\eem{\end{pmatrix}}

\def\cube{\mathfrak{C}}
\def\cubeg{\cube^{(g)}}

\def\Lopt{\hat L}
\def\Jopt{\hat J}
\def\sopt{{\hat s}}
\def\topt{{\hat t}}
\def\aopt{{\hat a}}
\def\bopt{{\hat b}}
\def\Jopt{{\hat J}}

\def\mfD{\mathfrak{D}} 
\def\aofst{a(s,t)}
\def\bofst{b(s,t)}

 \def\CDF{Cumulative Distribution Function}
 
 \def\PDF{probability density function}

\def\sdd{ {\frac {\head}  {\dd} } }

\def\eih{e_{\head,\tail}} 

\def\sis{{\sigma_{s,t}}}
\def\sih{{\eih}}

\def\ddhalf{{\frac \dd 2}}

\def\NNhalf{{\frac12 \NN }}

\def\Pbe{{P^{b(\head, \tail +1)} }}

\def\BRV{{\mathfrak B}}
\def\SRV{{\mathfrak S}}

\def\hPhi{{\hat \Phi}}

\def\head{\mathfrak{s}}
\def\tail{\mathfrak{t}}
\def\fC{\mathscr{C}}
\def\dd{\mathfrak{d}}

\def\salt{\mathfrak{s}}
\def\talt{\mathfrak{t}}
\def\eiha{e_{\salt,\talt}}

\def\ds{\displaystyle}

\def\mfB{\mathfrak{B}}

\def\mfBm{\mfB^{\mathrm{min}}}
\def\mfBc{\mfB^{\mathrm{spin}}}
\def\mfBmax{\mfB^{\mathrm{max}}}

\def\mfBoh{\mfB^{\mathrm{oh}}}

\def\cD{\mathcal D}

\newcommand{\C}{{\mathbb C}}

\newcommand{\RR}{{\mathbb R}}
\newcommand{\R}{{\mathbb R}}

\newtheorem{thm}{Theorem}[section]

\newtheorem{cor}[thm]{Corollary}
\newtheorem{lem}[thm]{Lemma}
\newtheorem{lemma}[thm]{Lemma}
\newtheorem{prop}[thm]{Proposition}

\newtheorem{problem}[thm]{Problem}

\theoremstyle{definition}

\newtheorem{conj}[thm]{Conjecture}

\newtheorem{remark}[thm]{Remark}
\newenvironment{rem}%
         {\begin{remark}}
             {{\hfill $\Box ~~$}\end{remark}}

\newtheorem{sublem}[thm]{Sublemma}
\newtheorem{sublemma}[thm]{Sublemma}

\numberwithin{equation}{section}

\newtheorem{exa}[thm]{Example}
\newenvironment{example}%
         {\begin{exa}}
             {{\hfill $\Box ~~$}\end{exa}}

\newcounter{Inc}

\usepackage[dvipsnames]{xcolor}

\newcommand{\df}[1]{{\bf{#1}}{\index{#1}}}

\title[Commuting Dilations and Linear Matrix Inequalities]{Dilations, Linear Matrix Inequalities,\\ the Matrix Cube Problem and Beta Distributions}

\author[J.W. Helton]{J. William Helton${}^1$}
\address{J. William Helton, Department of Mathematics\\
  University of California \\
  San Diego}
\email{helton@math.ucsd.edu}
\thanks{${}^1$Research supported by the NSF grant
DMS 1201498, and the Ford Motor Co.}

\author[I. Klep]{Igor Klep${}^{2}$}
\address{Igor Klep, Department of Mathematics, 
The University of Auckland, New Zealand}
\email{igor.klep@auckland.ac.nz}
\thanks{${}^2$Supported by the Marsden Fund Council of the Royal Society of New Zealand. Partially supported by the Slovenian Research Agency grants P1-0222 and L1-6722.}

\author[S. McCullough]{Scott McCullough${}^3$}
\address{Scott McCullough, Department of Mathematics\\
  University of Florida\\ Gainesville 
   }
   \email{sam@math.ufl.edu}
\thanks{${}^3$Research supported by the NSF grant DMS-1361501}

\author[M.  Schweighofer]{Markus  Schweighofer}
\address{Markus  Schweighofer, Fachbereich Mathematik und Statistik\\
Universit\"at Konstanz}
\email{markus.schweighofer@uni-konstanz.de}


\date{\today}

\makeindex

\begin{document}

\setcounter{tocdepth}{3}
\contentsmargin{2.55em} 
\dottedcontents{section}[3.8em]{}{2.3em}{.4pc} 
\dottedcontents{subsection}[6.1em]{}{3.2em}{.4pc}
\dottedcontents{subsubsection}[8.4em]{}{4.1em}{.4pc}

\makeatletter
\newcommand{\mycontentsbox}{%
{
\addtolength{\parskip}{-2.3pt}
\tableofcontents}}
\def\enddoc@text{\ifx\@empty\@translators \else\@settranslators\fi
\ifx\@empty\addresses \else\@setaddresses\fi
\newpage\mycontentsbox\newpage\printindex}
\makeatother

\setcounter{page}{1}

\begin{abstract}
An operator $C$ on a Hilbert space $\mathcal H$ dilates to an operator $T$ on a Hilbert space $\mathcal K$ if there is an isometry
 $V:\mathcal H\to \mathcal K$ such that $C= V^* TV$. 
A main result of this paper is, for a positive integer $d$, 
the
simultaneous dilation, up to a sharp factor $\vartheta(d)$, expressed as a ratio of $\Gamma$ functions for $d$ even, of all $d\times d$ 
symmetric matrices of operator norm at most one to a collection of \emph{commuting} self-adjoint contraction operators on a Hilbert space. \looseness=-1

 Dilating to commuting operators has consequences for the
 theory of linear matrix inequalities (LMIs).
 Given  a tuple  $A=(A_1,\dots,A_g)$ of $\nu\times \nu$ symmetric matrices,
 $L(x):= I-\sum A_j x_j$ is a {\it monic linear pencil} of size $\nu$.
The solution set   $\mathscr S_L$ of the corresponding  
linear matrix inequality,
 consisting of those $x\in\R^g$ for which $L(x)\succeq 0$, is a {\it spectrahedron.}  
 The set $\cD_L$ of tuples $X=(X_1,\dots,X_g)$
 of symmetric matrices (of the same size) for which $L(X):=I -\sum A_j\otimes X_j$ is positive
 semidefinite, is a {\it free spectrahedron.}  
 It is shown that  any tuple $X$ of $d\times d$  symmetric matrices
 in a bounded free spectrahedron $\cD_L$ dilates, up to a scale factor depending only on $d$,
to a tuple $T$  of \emph{commuting} self-adjoint operators with joint spectrum in
 the corresponding spectrahedron $\mathscr S_L$.
From another viewpoint, the scale factor measures
the extent that a positive map can fail to be completely positive.

 Given another monic linear pencil $\tilde{L}$, the inclusion $\cD_L\subset \cD_{\tilde{L}}$
 obviously implies the inclusion $ \mathscr S_L\subset\mathscr S_{\tilde{L}}$ 
 and thus can be  thought of  as its free relaxation. 
Determining if one free spectrahedron
contains another can be done by solving an explicit LMI and is thus computationally
tractable.
 The scale factor for commutative dilation of  $\cD_L$ gives a
 precise measure of  the worst case error inherent  in the free relaxation,
 over all monic linear pencils $\tilde{L}$ of size $d$.

 The set $\cubeg$ of $g$-tuples  of symmetric matrices of norm at most
 one is an example of a free spectrahedron  known as the free cube and its associated spectrahedron is the cube $[-1,1]^g$. The free relaxation of the 
 the NP-hard inclusion problem $[-1,1]^g\subset\mathscr S_L$ was introduced by Ben-Tal and Nemirovski. They obtained the lower bound $\vartheta(d),$ expressed as the solution of an optimization problem over diagonal matrices of trace norm $1,$ for the divergence between the original and relaxed problem. The result here on simultaneous dilations of contractions proves this bound is sharp.
Determining an analytic formula for $\vartheta(d)$ produces, as a by-product, 
new probabilistic results for the binomial and beta distributions.
\end{abstract}

\maketitle

\newpage

\section{Introduction}
\label{sec:intro}
 Free analysis\index{free analysis} \cite{KVV14} and the theory of free functions and free sets 
 traces its roots back to the work of  Taylor \cite{Taylor1, Taylor2}.  Free functions generalize
  the notion of polynomials in $g$-tuples of freely noncommuting variables and free sets
  are appropriately structured subsets of the union, over $d$,  of $g$-tuples of $d\times d$ matrices. 
 The current interest in the subject arises in large
 part due to its applications in free probability \cite{Voi04,Voi10},
  systems engineering 
 and connections to optimization \cite{Ball11,BGR,BGFB94,dOHMP09,Skelton}
 and operator algebras and systems
 \cite{Arv69,Arv72,Pau,BLM,Pisier-book, Davidson, Davidson-Kennedy}. 
 The main branch of convex optimization to emerge in the last 20 years,
 semidefinite programming\index{semidefinite programming} \cite{Nem06}, is based on linear pencils,\index{linear pencil}
linear matrix inequalities\index{linear matrix inequality}\index{LMI} (LMIs) 
and spectrahedra\index{spectrahedron}
 \cite{BGFB94,WSV}.   The book
  \cite{BPR13} gives an overview of  the 
  substantial theory of LMIs and spectrahedra and the connection to real algebraic geometry. 
  A  linear pencil $L$ is  a simple special case of 
a free function  
  and is of special 
  interest because the free spectrahedron  $\cD_L=\{X : L(X)\succeq 0\}$\index{$\cD_L$} is evidently  convex and  conversely an 
  algebraically defined free convex set  is a  free spectrahedron \cite{EW,HM12}. 
  In this article the relation between inclusions of spectrahedra and inclusions of the corresponding 
  free spectrahedra is explored using operator theoretic ideas. The analysis 
 leads to new dilation theoretic results and to new probabilistic results
   and   conjectures  which 
 can be read independently of the rest of this article by skipping now to Section \ref{sec:introprob}.  It also furnishes a complete solution to the 
 matrix cube problem of Ben-Tal and Nemirovski \cite{BtN}, which contains,
 as a special case, the $\frac{\pi}{2}$-Theorem of Nesterov \cite{Nes97} and 
 which in turn is related to
  the symmetric Grothendieck Inequality.

A central  topic  of this paper 
is  dilation,\index{dilation} up to a scale factor,
of a tuple $X$ of $d\times d$  symmetric matrices 
in a free spectrahedron $\cD_L$
to  tuples $T$
 of \emph{commuting} self-adjoint operators with
 joint spectrum in  the corresponding spectrahedron $\mathscr S_L$.
We shall prove that  these scaled commutative  dilations 
exist and  the scale factor
describes the error in the free relaxation 
$\cD_L\subset\cD_{\tilde L}$ of the spectrahedral
inclusion problem 
$\mathscr S_L \subset \mathscr S_{\tilde{L}}$.\index{$\mathscr S_L$} 
The precise results are stated in  Section \ref{sec:op-dil}. 
As a cultural note these scale factors
 can be interpreted as the amount of modification required
 to make a positive map
  completely positive; see Section \ref{sec:cpINTRO}. 

In this paper we completely analyze the free cubes\index{free cube}, $\cubeg$,\index{$\cubeg$}
 the free spectrahedra consisting of $g$-tuples of symmetric
 matrices of norm at most one;  the corresponding
 spectrahedron is the cube $[-1,1]^g$.
We show that,  for each $d$,
there exists a
collection $\mathscr C_d$ of commuting self-adjoint
 contraction operators on a Hilbert space,
such that, up to the scale factor $\th(d)$,\index{$\th(d)$}
any $d \times  d$  symmetric contraction  matrix
dilates to $T$ in   $\mathscr C_d$; see
Section \ref{sec:dilate and cube}.
Moreover, we give a formula for the optimal scale factor $\vartheta(d)$;
see Section \ref{sec:precise}.
As a consequence,
 we recover the error bound  given by Ben-Tal
 and Nemirovski for the computationally tractable free relaxation of the
NP-hard cube inclusion problem.
  Further, we show that this bound is
  best possible, see Section \ref{sec:mcINTRO}.

Proof of sharpness of the error bound $\vartheta(d)$
  and giving a formula for $\vartheta(d)$ requires concatenating all of the areas
  we have discussed and it requires all but a few sections
  of this paper.
For example, finding a formula  for $\vartheta(d)$
 required new  results for
the binomial and beta distributions
and necessitated a generalization of
Simmons' Theorem \cite{Sim1894}  (cf.~\cite{PR07})
to Beta distributions.
Our results and conjectures in probability-statistics appear
 in Section \ref{sec:introprob}.

 The rest of the introduction gives
detailed statements of the results just described
and a guide to their proofs.

\subsection{Simultaneous dilations}
 \label{sec:dilate and cube}
Denote by $\N:=\{1,2,3,\dots\}$ the set of positive integers and by $\R$ the set of real numbers.
For $n\in\N$, denote by $\smat_n$ the set of symmetric $n\times n$ matrices with entries from $\R$. 
 A matrix $X\in\smatn$ {\bf dilates} to  an operator $T$ 
 on a Hilbert space $\mathcal H$ if there is an isometry $V:\mathbb R^n\to \mathcal H$ such that
 $X=V^*TV$.   Alternately, one says
 that $X$ is a \df{compression} of $T$.
   A tuple $X\in\smatng$ {\bf dilates} to a tuple $T=(T_1,\dots,T_g)$ of bounded operators
 on a Hilbert space $\mathcal H$ if there is an isometry $V:\mathbb R^n\to \mathcal H$ such that
 $X=V^*TV$ (in the sense that $X_j = V^* T_jV$). 
In other words, $T$ has the form
 \[ T_i= \begin{pmatrix} X_i & *_i \\ *_i & *_i \end{pmatrix}\] 
where the $*_i$ are  bounded operators between appropriate Hilbert spaces.

 One of the oldest dilation\index{dilation} theorems is due to
 Naimark \cite{N43}. In its simplest form it dilates a tuple of (symmetric) positive semidefinite matrices (of the same size) which sum to the identity to a tuple of commuting (symmetric) projections which sum to the identity.\index{dilation!Naimark} 
 It has  modern applications in the theory of frames from signal processing. In this direction, perhaps the most general version of the Naimark Dilation Theorem 
 dilates a (possibly nonselfadjoint) operator valued measure to a (commuting) projection valued measure on a Banach space. 
 The most general and complete reference for this result  and its antecedents is \cite{HLLL14} with \cite{HLLL14b}
being an exposition of the theory. See also \cite{LS13}. 
A highly celebrated dilation result  in operator theory is 
 the power dilation theorem of Sz.-Nagy\index{theorem!Sz.-Nagy} \cite{SzN53} which, 
  given a contraction $X$, constructs a unitary $U$
   such that $X^n$ dilates to $U^n$ for natural numbers $n$.
 That von Neumann's inequality\index{von Neumann's inequality}\index{theorem!von Neumann} is
 an immediate consequence gives some measure of the power of this result.
 The two variable generalization of the Sz.-Nagy dilation theorem, the power dilation of a commuting
  pair of contractions to a commuting pair of unitaries, is known as the
 commutant lifting theorem  (there are counterexamples to commutant lifting
 for more than two  contractions) and is 
due to Ando, Foias-Sz.-Nagy, Sarason.  
 It  has  major applications to linear systems engineering;
 see \cite{Ball11,FFGK,BGR} as samples of the  large literature.
  The (latest revision of the) classic book \cite[Chapter 1.12]{SzNFBK} 
  contains further remarks on the history of dilations.
Power dilations up to a scale factor $K$
are often called $K$-spectral dilations and these are a highly active are
of research. An excellent survey article is  \cite{BB13}.

 The connections
 between dilations and completely positive\index{completely positive map} maps were exposed
 most famously in the work of
  Arveson \cite{Arv69,Arv72}.
 Presently, dilations and completely positive maps appear in many contexts.
 For examples, they are fundamental objects
 in the theory of operator algebras, systems and spaces
 \cite{Pau} as well as in quantum computing
 and quantum information theory \cite{NC}.
 In the articles \cite{HKM12,HKM13}, the theory of completely
 positive maps was used to systematically
 study free relaxations of spectrahedral inclusion problems
 which arise in semidefinite programming \cite{Nem06,WSV}
 and systems engineering \cite{BtN} for instance.
 In this article, dilation theory is used to measure the degree
  to which a positive map can fail to
 be completely positive,  equivalently the error inherent in
  free relaxations of
 spectrahedral inclusion.

The dilation constant $\vartheta(d)$ for dilating $d\times d$  contractions to commuting contractions operators
is given by the following optimization problem \cite{BtN} which is potentially of independent interest.
\beq
\label{eq:BTbd}
\frac{1}{\th(d)}  := 
\min_{\substack{a\in\R^d\\[.5mm]|a_1|+\dots+|a_d|=d}}\int_{S^{d-1}}\left|\sum_{i=1}^d a_i\xi_i^2\right|d \xi 
 = \min_{\substack{B\in \smat_d\\[.5mm]\trace|B|=d}}
 \int_{S^{d-1}} | \xi^*B\xi | \,  \, d\xi 
\eeq
where the unit sphere $S^{d-1}\subseteq \R^ {d}$ (having dimension $d-1$) is equipped with the uniform probability measure (i.e., the unique rotation invariant measure of total mass $1$). 
That $\vartheta(d)$ can be expressed using incomplete beta functions will be seen in 
Section \ref{sec:precise}.
Evidently $\vartheta(d)\ge 1.$ 

   A  self-adjoint operator $Y$
 on $\mathcal H$ is a \df{contraction} if $I\pm Y \succeq 0$ or equivalently $\|Y\|\le 1$.

\begin{thm}[Simultaneous Dilation]
 \label{thm:dilate}
 Let $d\in\N$.  There is a Hilbert space $\cH$, a family $\fC_d$ of commuting self-adjoint  contractions
  on $\cH$, and an isometry $V:\mathbb R^d\to \cH$ such that for each symmetric $d\times d$ contraction matrix $X$
  there exists a $T\in \fC_d$ such that 
\[
  \frac{1}{\th(d)} X = V^* T V. 
\]

   Moreover, $\th(d)$ is the smallest such constant in the sense that if
   $\th'\in\R$ satisfies $1\le\th'<\th(d)$, then there is $g\in\N$ and
   a $g$-tuple of $d\times d$ symmetric contractions $X$   
   such that $\ds\frac{1}{\th'}X$
   does not dilate to a $g$-tuple of commuting self-adjoint  contractions on a Hilbert space. 
\end{thm}

\begin{proof}
  The first part of Theorem \ref{thm:dilate} is stated and proved as Theorem \ref{thm:rhodoesit}.
  The optimality of $\vartheta(d)$ is proved as part of Theorem \ref{thm:nextbest}.  The Hilbert space
  $\cH$, isometry $V$ and collection $\fC_d$ are all explicitly constructed. See equations (\ref{eq:H}), (\ref{eq:V}) and
  \eqref{eq:defcCd}.
\end{proof}

\subsection{Solution of the minimization problem \protect{\eqref{eq:BTbd}}}
\label{sec:precise}
\index{beta function} \index{incomplete beta function}\index{regularized beta function}\index{beta function, incomplete}\index{beta function, regularized}
In this section matrices $B$ which produce the optimum in Equation
 \eqref{eq:BTbd} are described and a formula for $\th(d)$ is given in terms of  Beta functions.
 Recall, the incomplete beta function is, for real arguments $\al,\be>0,$ 
 and an additional argument $p\in[0,1]$, defined by 
\[B_p(\al,\be)=\int_0^px^{\al-1}(1-x)^{\be-1}dx.\] 
 The Euler beta function is $B(\al,\be)=B_1(\al,\be)$ 
and the regularized (incomplete) beta function is\looseness=-1
\[I_p(\al,\be)=\frac{B_p(\al,\be)}{B(\al,\be)}\in[0,1].\]

The minimizing matrices $B$ to \eqref{eq:BTbd} 
will have only two different eigenvalues. 
 For nonnegative numbers $a,b$ 
and $s,t\in\N,$  let $J(s,t;a,b)=a I_s \oplus (-b) I_t$ denote 
the $d\times d$ diagonal matrix $J(s,t;a,b)$
 with first $s$ diagonal entries $a$ and last $t$ diagonal entries $-b$.

The description of the solution to \eqref{eq:BTbd} depends on the parity of $d$.

\begin{thm}
\label{thm:thetaExplicit}
 If  $d$ is an even positive integer, then
\begin{equation}
 \label{eq:mc}
 \frac{1}{\vartheta(d)} 
  = \int_{S^{d-1}} \left|\xi^* J\left(\frac d 2 ,\frac d 2 ;1,1\right) \xi\right| \, d\xi 
= 2I_{\frac12}\left(\frac d4,\frac d4+1\right)-1                        
= \frac{\Gamma\left(\frac12+\frac d4\right)}{\sqrt\pi\,\Ga\left(1+\frac d4\right)}
\end{equation}
where $\Gamma$ denotes the Euler gamma function. 
 In particular, the minimum in \eqref{eq:BTbd} occurs at a $B=J(s,t;a,b)$ with $s=t=\frac d2$ and $a=b=1.$

In the case that $d$ is odd, 
 there exist  $a,b\ge 0$ such that 
\begin{align}
 \label{eq:oddmc0} \frac{1}{\vartheta(d)} 
 & = \int_{S^{d-1}} \Big|\xi^* J\left(\frac {d+1}2 ,\frac {d-1}2 ;a,b\right) \xi \Big| \, d\xi \\
\label{eq:oddmc} & = 2 I_{\frac a{a+b}}
\left(\frac {d-1}4, \frac {d+1}4 +1\right) -1 =2  I_{\frac{b}{a+b}}\left(\frac {d+1}4,\frac {d-1}4 +1\right) -1,
\end{align}
and
$a\frac{d+1}{2}+b\frac{d-1}{2}=d$. This last equation together with
\eqref{eq:oddmc} uniquely determines $a,b$. 
Furthermore, the minimum in \eqref{eq:BTbd} occurs at a $B=J(s,t;a,b)$ with $s=\frac {d+1}2$ and $t=\frac {d-1}2$.
\end{thm}

\subsubsection{Proof of Theorem {\rm\ref{thm:thetaExplicit}}}
The proof is involved but we now describe some of the ideas.
A key observation is that the minimum defining $\vartheta(d)$ 
in \eqref{eq:BTbd} can be taken over matrices of the form
 $J(s,t;a,b)$, instead of over all symmetric matrices
 $B$ with $\trace(|B|)=d;$
 see Proposition \ref{prop:starastarb}.
In addition, $s+t=d$ and  we may take $as+bt=d= \trace | J(s,t,a,b) |$.
The key identity  connecting Beta functions to $J$ is
\begin{equation}
 \label{eq:theta-beta}
\begin{split}
 \int_{S^{d-1}} \big| \xi^* J(s,t;a,b)\xi \big| \, d\xi 
  = &  \int_{S^{d-1}} \left | a \sum_{j=1}^s \xi_j^2 - b \sum_{j=s+1}^d \xi_j^2\ \right | \, d\xi  \\
    = & \frac{2}{d} \left( as I_{\frac{a}{a+b}}\left(\frac{t}{2},\frac{s}{2}+1\right) + bt I_{\frac{b}{a+b}}\left(\frac{s}{2},\frac{t}{2}+1\right)  \right) -1,
\end{split}
\end{equation}
 which is verified in
 Section \ref{sec:prequest}.

 The optimality conditions for the  optimization problem 
 \eqref{eq:BTbd} (with $J(s,t;a,b)$ replacing $B\in\mathbb S_d$) 
 are presented in Section \ref{sec:opt}, 
and the proof of Theorem \ref{thm:thetaExplicit} 
concludes in Section \ref{sec:nov}. 
\hfill\qedsymbol

Bounds on the integral \eqref{eq:oddmc} representing $\vartheta(d)$
when $d$ is odd
 can be found below in 
Theorem \ref{thm:thetaEvenAndOdd}.

\subsubsection{Coin flipping and Simmons' Theorem}
 
 Theorem \ref{thm:thetaExplicit} is closely related to coin flipping. 
 For example, when $d$ is divisible by $4$,
the
 right hand side of \eqref{eq:mc} just becomes the probability of getting exactly $\frac d4$  heads when tossing a fair coin
$\frac d2$ times, i.e.,
 \[\binom{\frac d2}{\frac d4}\left(\frac12\right)^{\frac d2}.\]
Furthermore, a core ingredient in analyzing the extrema of  
\eqref{eq:theta-beta} or \eqref{eq:BTbd} as needed for
$\vartheta(d)$ is the following inequality.

\begin{thm}
 \label{simmons+INTRO}
For $d\in\mathbb N$ and $s,t\in\mathbb N$  with $s+t=d$, if $\ds s\geq\frac d2$,  then  
\begin{equation}
 \label{eq:someconjINTRO}
   I_{\frac sd}\left(\frac s2+1,\frac t2\right) 
\geq  1- I_{\frac sd}\left(\frac s2,\frac t2+1\right).
\end{equation}
 \end{thm}

\begin{proof}
See  Section \ref{sec:simmhalf}.
\end{proof}
 
If $s,d$ in Theorem \ref{simmons+INTRO} are both even, equivalently
 $\frac s2$ and $\frac t2$ are natural numbers, 
then \eqref{eq:someconjINTRO}
reduces to the following:
toss a coin
whose probability for head is $\frac sd\geq\frac12$, 
 $d$ times.
Then
the probability of getting fewer than $s$ heads
\emph{is no more than}
the probability of getting more than $s$  heads.
This result is known in classical probability as Simmon's Theorem \cite{Sim1894}.

Further probabilistic connections
 are described in Section \ref{sec:introprob}.

\subsection{Linear matrix inequalities (LMIs), spectrahedra and general dilations}
\label{sec:op-dil}

In this section we discuss simultaneous dilation of tuples
of symmetric matrices satisfying linear matrix inequalities to commuting self-adjoint operators.

 For $A,B \in\smat_n$, write
 $A\preceq B$ (or $B\succeq A$) to express that $B-A$ is positive semidefinite (i.e., has only nonnegative eigenvalues).
 Given a $g$-tuple $A=(A_1,\dots,A_g)\in\smat_{\nu}^g$, 
the expression
\beq\label{eq:definePencil}
 L_A(x) = I_\nu -\sum_{j=1}^g A_j x_j
\eeq
 is a {\bf (monic)  linear pencil}\index{linear pencil}\index{monic linear pencil} and $L_A(x)\succeq 0$ is a {\bf linear matrix inequality (LMI)}.\index{LMI}\index{linear matrix inequality} Its 
 solution set $\mathscr S_{L_A}=\{x\in\mathbb R^g: L_A(x)\succeq 0\}$ is a
\df{spectrahedron} (or an 
  \df{LMI domain})
  containing $0$ in its interior \cite{BPR13,BGFB94}. Conversely, each
 spectrahedron with non-empty interior can be written in this form
 after a translation \cite{HV07}. 
Every polyhedron is a spectrahedron.\index{polyhedron} 
For example, that the  cube\index{cube} $[-1,1]^g$ in $\mathbb R^g$ is an example of a spectrahedron, is seen as follows. 
 Let  $E_j$ denote the $g\times g$ diagonal matrix with a $1$ in the $(j,j)$ entry and zeros elsewhere and  define $C\in\smat_{2g}^g$ by setting
\beq\label{eq:Ccube}
  C_j:=\begin{pmatrix}1&0\\0&-1\end{pmatrix}\otimes E_j=\begin{pmatrix}E_j&0\\0&-E_j\end{pmatrix}
\eeq
for $j\in\{1,\dots,g\}$.
 The resulting spectrahedron $\mathscr S_{L_C}$ is the cube {$[-1,1]^g$}. 

For $n\in\N$ and tuples $X\in \smatng$, let
\[
\begin{split}
 L_A(X) & = I_{\nu n}-\sum_{j=1}^g A_j\otimes X_j, \quad\text{and} \\
 \cD_{L_A}(n) &=\big\{X\in\smatng : L_A(X)\succeq0
 \big\},
 \end{split}
\]
 where $\otimes$ is the Kronecker tensor product.\index{tensor product}\index{tensor product!Kronecker}  The sequence $\cD_{L_A}=(\cD_{L_A}(n))_n$ is a \df{free spectrahedron}.
 In particular, $\cD_{L_A}(1) =\mathscr S_{L_A}$ and $\cD_{L_C}(n)$ is the collection of $g$-tuples of 
  $n\times n$ symmetric contraction matrices.  We call  $\cubeg:=\cD_{L_C}$  the \df{free cube} (in $g$-variables).
Free spectrahedra are closely connected with operator systems for which
 \cite{FP,KPTT,Arv08} are a few recent references.  

 \def\fs{\mathscr S}
 
\subsubsection{Dilations to commuting operators}

The general dilation\index{dilation} problem is as follows: given a linear pencil $L$ and a tuple $X\in\cD_L$, does $X$ dilate to a commuting tuple $T$ of
self-adjoint operators with joint spectrum in $\fs_L$?

 Suppose $L$ is a monic linear pencil in $g$-variables and the corresponding
 spectrahedron $\fs_L$ is bounded.  
  Because there exist  constants $c$ and $C$ such that
\[
  c\, \cubeg \subset \cD_L \subset C\, \cubeg,
\]
 from Theorem \ref{thm:dilate} it follows that for 
each  $n  \in \NN$,  and
$X \in  \cD_L(n)$ there exists a $t\in\R_{>0}$, a Hilbert space $\mathcal H$,
 a commuting tuple $T$ of self-adjoint operators on $\mathcal H$ with joint spectrum 
  in $\fs_L$  and an isometry
  $V:\mathbb R^n\to \mathcal H$ such that 
\[
  X = V^* \frac{1}{t} T  V.
\]
  The largest $t$ such that for each $X \in \cD_L(n)$ the tuple $tX$  dilates to a commuting tuple 
  of self-adjoint operators with joint spectrum in $\fs_L$ is the
{\bf commutability index of $L$}\index{commutability index},  denoted  by $\tau(L)(n)$. (If $\fs_L$ is not bounded, then there need not be
 an upper bound on $t$.)
 The constants $c,C$ and Theorem \ref{thm:BtN} below  produce
  bounds on $t$ depending only upon the monic pencil $L$ and  $n$.

\subsubsection{Spectrahedral inclusion problem}\label{subsub:spIncl}

 Given two monic linear pencils $L,\tL$ 
 and corresponding spectrahedra, determine if the inclusion $\fs_L\subset \fs_{\tL}$ holds.
 The article \cite{HKM13} considered the free variable relaxation of this inclusion problem, dubbed
  the  \df{free spectrahedral inclusion problem}:
 when does the  containment  $\cD_L \subset \cD_\tL$ hold?
In \cite{HKM13,HKM12} it is shown, via an algorithm and 
  using complete positivity, that  any such free 
 inclusion problem  can be converted to  an
SDP\index{semidefinite programming}\index{SDP}  feasibility problem
(whose complexity status is unknown but
 is believed to be efficient to solve in theory and practice; cf.~\cite[Ch. 8.4.4]{WSV}). 
 (See also Section \ref{sec:cpINTRO} below.)

\subsubsection{Accuracy of the free relaxation}
Now that we have seen free spectrahedral inclusion 
problems are in principle solvable, what do they say about 
the original inclusion problem?
Inclusion of free sets $\cD_L \subset \cD_\tL$  implies trivially 
the inclusion of the corresponding classical 
spectrahedra  $\fs_L \subset \fs_\tL.$ Conversely, in the case that $\fs_{\tL}$ is bounded 
 there exists $c,C>0$ such that 
\begin{equation*}
  c\cubeg \subset \cD_{\tL} \subset C \cubeg,
\end{equation*}
 and hence  there exists an $r\in\R_{>0}$ such that 
\[
\fs_L \subset \fs_\tL
\qquad 
 \text{implies}
\qquad 
r \cD_L  \subset  \ \cD_\tL.
\]
We call such an $ r $ an \df{$\cD_L$-$\cD_\tL$-inclusion constant}.
Theorem \ref{thm:BtN} and the constants $c,C$ produce
 a lower bound on $r$. Let $r(L,\tilde{L})$ denote the largest such $r$
 (if $\fs_{\tilde{L}}$ is not assumed bounded, then a largest $r$ need not exist) and let 
\[
  r(L)(d):=  \min\big\{ r(L,\tL): \tL \mbox{ is of size } d \mbox{ and } \fs_L\subset \fs_{\tL}\big\}.
\]
 We call  the sequence $r(L):= (r(L)(d))_d$
   the  \df{$\cD_L$-inclusion  scale}.

The connection between spectrahedral inclusions and general dilations arises as follows:

\begin{thm} 
 \label{thm:scottDilated-intro}
Suppose $L$ is a  monic linear pencil 
and $\fs_L$ is bounded. 
\ben[\rm(1)]
\item
\label{it:BTNmuis1} 
The commutability index for $L$ equals its inclusion scale, $\tau(L)= r(L)$.
That is $\tau(L)(d)$ is the largest constant such that 
\[
\tau(L)(d) \  \cD_L \subset    \cD_\tL
\]
  for each $d$  and monic linear pencil 
 $\tL$ of size $d$ satisfying $\fs_L\subset\fs_\tL.$
\item
\label{it:gend}
If  $\cD_L$ is   balanced $($i.e., for each  $X\in\cD_L$ we have $-X\in\cD_L)$, then for each $n\in\N$,
 $$  {\tau(L)(n)} \geq\frac1n.
 $$
\een
\end{thm}

\begin{proof}
 The proof appears in Section \ref{sec:generally free}.
\end{proof}

\subsection{Interpretation in terms of completely positive maps}\label{sec:cpINTRO}

The intimate connection between dilations and completely positive maps was first exposited by Stinespring \cite{Sti55} to give abstract necessary and sufficient conditions for the existence of dilations. 
The theme was explored further by Arveson, see e.g.~\cite{Arv69, Arv72};
we refer the reader to \cite{Pau} for a beautiful exposition.
We next explain how our dilation theoretic results pertain to
(completely) positive maps.

The equality between the commutability index and the inclusion scale (Theorem \ref{thm:scottDilated-intro}) can be interpreted via
positive and completely positive maps. Loosely speaking,
each unital positive map can be scaled to a unital
completely positive map in a uniform way.

\def\mBS{\mathbb S}
Suppose $A\in\mBS_\nu^g$ is such that the associated spectrahedron
$\fs_{L_A}$ is bounded. If $\tA\in\mBS_\eta^g$ is another $g$-tuple,
consider the unital linear map
\beq\label{eq:Phi}
\begin{split}
\Phi:{\rm span}\{I,A_1,\ldots, A_g\} & \to 
{\rm span}\{I,\tA_1,\ldots, \tA_g\} \\
A_j & \mapsto \tA_j.
\end{split}
\eeq
(It is easy to see that $\Phi$ is well-defined by the boundedness of $\fs_{L_A}$.)
For $c\in\R$ we define the following scaled distortion of $\Phi$:
\[
\begin{split}
\Phi_c:{\rm span}\{I,A_1,\ldots, A_g\} & \to 
{\rm span}\{I,\tA_1,\ldots, \tA_g\} \\
I & \mapsto I\\
A_j & \mapsto c \tA_j.
\end{split}
\]

\begin{cor}\label{cor:lameCP}
With the setup as above, 
$c:=\tau(L_A)(\eta)$ is the largest scaling factor 
with the following property:
if $\tilde{A}\in\mBS_\eta^g$
 and  if $\Phi$ is positive,  then 
  $ \Phi_{c}$  completely positive.
\end{cor}

\begin{proof}
A map $\Phi$ as in \eqref{eq:Phi}
is $k$-positive iff $\cD_{L_A}(k)\subseteq\cD_{L_{\tA}}(k)$ by
\cite[Theorem 3.5]{HKM13}.
Now apply Theorem \ref{thm:scottDilated-intro}.
\end{proof}

\subsection{Matrix cube problem}\label{sec:mcINTRO}
Given $A\in\smatdg$, 
the matrix cube problem\index{matrix cube problem}
of  Ben-Tal and Nemirovski \cite{BtN} is to determine,  whether 
$\mathscr S_{L_C} =[-1,1]^g \subseteq\mathscr S_{L_A}$.
While their primary interest in this problem came from robust analysis (semidefinite programming with interval uncertainty and quadratic Lyapunov stability 
 analysis and synthesis), this problem is in fact far-reaching.
 For instance, \cite{BtN}
 shows that determining the 
maximum of a positive definite quadratic form over the unit cube 
  is a special case of the matrix cube problem (cf.~Nesterov's $\frac\pi2$-Theorem \cite{Nes97},\index{theorem!Nesterov's $\frac\pi2$}
or the Goemans-Williamson \cite{GW95} SDP relaxation of the Max-Cut problem).\index{theorem!Goemans-Williamson}  
Furthermore, it implies the symmetric Grothendieck inequality. \index{theorem!Grothendieck}
A very recent occurrence of the matrix cube problem
  is described in \cite{BGKP}, see their Equation (1.3).
There a  problem  in statistics is shown equivalent to
whether a LMI, whose coefficients are Hadamard products
of a  given matrix, contains a cube.

  Of course, one could test the inclusion $\mathscr S_{L_C} \subseteq \mathscr S_{L_A}$ by checking if all vertices of the cube $\mathscr S_{L_C}$ are contained in $\mathscr S_{L_A}$. 
 However, the number of vertices grows exponentially with the dimension $g$. Indeed the matrix cube problem 
 is NP-hard \cite{BtN,Nem06}; see also \cite{KTT13}.
 A principal result in \cite{BtN} is the identification of a computable\index{theorem!Ben-Tal, Nemirovski}
 error bound  for a natural relaxation of the matrix cube problem.
In \cite{HKM13} we observed this relaxation is in fact equivalent
to the free  relaxation\index{free relaxation}\index{matricial relaxation} $\cubeg\subseteq\cD_{L_A}$.

With the notations introduced above, 
we can now present the theorem 
of Ben-Tal and Nemirovski  bounding the 
error of the free relaxation.

\begin{thm}[\cite{BtN}]
 \label{thm:BtN} Given $g,\nu\in\N$ and  $B\in\smat_{\nu}^g,$ if  $[-1,1]^g\subset \fs_{L_B},$ then
 \begin{enumerate}[\rm(a)]
 \item $\cubeg \subseteq\th(\nu)\  \cD_{L_B}$;
 \item if $d\in\N$ is an upper bound for the ranks of the $B_j$, then
  $\cubeg \subseteq\th(d)\ \cD_{L_B}$.
\end{enumerate}
\end{thm}

\begin{proof}
Part (a) of Theorem \ref{thm:BtN} is shown,  in Theorem \ref{thm:nextbest}, to be a consequence of our Dilation Theorem \ref{thm:dilate}.
A further argument, carried out in Section \ref{sec:rankvsize},  establishes part (b).
\end{proof} 
 
 In this article we show that the bound $\th(d)$ in  Theorem \ref{thm:BtN}(a) (and hence in (b)) is sharp.

\begin{thm}
 \label{thm:sharp}
 Suppose $d\in\N$ and  $\th'\in\R.$  If $1\le\th'<\vartheta(d),$ then there is $g\in\N$ and $A\in\smat_d^g$ such that $[-1,1]^g\subset \fs_{L_A}$, but
   $\cubeg(d) \not\subset\th'\ \cD_{L_A}(d)$. 
\end{thm}

\begin{proof}
See Section \ref{sec:optkstar}.
\end{proof}

\begin{rem}\rm
 \label{rem:cube-dil}
   Theorem \ref{thm:scottDilated-intro} applied to the free cube(s) implies,
    for a given $g$, that 
    $\tau(\cubeg)$ equals $r(\cubeg)$ and, for fixed $d$,  
   the sequences $(\tau(\cubeg)(d))_g$ and $(r(\cubeg)(d))_g$ 
    termwise decrease with $g$ to a common limit, 
   which, in view of Theorems \ref{thm:BtN} and \ref{thm:sharp},
   turns out to be  $\vartheta(d)$. In particular, for any $g$ and any $g$-tuple $C$  of symmetric contractive matrices there exists
    a $g$-tuple of commuting self-adjoint contractions $T$ on Hilbert space such that $\frac1{\vartheta(d)} C$ dilates to $T$, a statement considerably weaker
    than the conclusion of Theorem \ref{thm:dilate}.
\end{rem}

\subsection{Matrix balls}\index{matrix ball}
 Paulsen \cite{Pau} studied the family of operator space structures associated to a normed vector space. Among these, he identified canonical maximal and minimal structures and measured the divergence between them with a parameter he denoted by $\alpha(V)$ \cite{Pau}. In the case that $V$ is infinite dimensional, $\alpha(V)=\infty$ and if $V$ has dimension $g$, then $\frac{\sqrt{g}}{2}\le \alpha(V)\le g$. Let $\ell^g_1$ denote the vector space $\mathbb C^g$ with the $\ell_1$ norm. Its unit ball is the cube $[-1,1]^g$ and in this case $\sqrt{\frac{g}{2}} \le \alpha(\ell^g_1) \le \sqrt{g-1}$. 

We pause here to point out differences between his work and the results in this paper, leaving it to the interested reader to consult \cite{Pau} for precise definitions.  Loosely, the maximal and minimal operator space structures involve quantifying over  matrices, not necessarily symmetric, of all sizes $d$. By contrast, in this article the matrices are symmetric,  the coefficient are real,  we study operator systems (as opposed to spaces) determined by linear matrix inequalities and, most significantly, to this point  the size $d$ is fixed.   In particular, for the matrix cube,  the symmetric matrix version of $\ell^g_1$ with the minimal operator space structure, the parameter $\vartheta(d)$ obtained by fixing $d$ and quantifying over $g$  remains finite.

 Let $\ell^g_2$ denote the $\mathbb C^g$ with the $\ell_2$ (Euclidean) norm and let $\mathbb B_g$ denote the (Euclidean) unit ball in $\mathbb R^g$. 
 In Section \ref{sec:balls}, we consider the free relaxation of the problem of including $\mathbb B_g$ into a spectrahedron with $g$, but not $d$, fixed.  Thus we study the symmetric variable analog of $\alpha(\ell^g_2)$.  Among our findings is that the worst case inclusion scale is exactly $g$.  By contrast, $\alpha(\ell^g_2)$ is only known to be bounded above by $\frac{g}{2^{\frac14}}$ \cite{Pisier-book} and below by roughly $\frac{g+1}{2}$ \cite{Pau}.

\subsection{Adapting the Theory to Free Nonsymmetric Variables}

In this brief subsection we explain how our dilation theoretic 
results extend to nonsymmetric variables. That is, we present the 
Simultaneous Dilation Theorem (Corollary \ref{cor:dilate})
dilating arbitrary contractive complex matrices to 
commuting normal contractions up to a scaling factor.

\def\fN{\mathscr N}

\begin{cor}[Simultaneous Dilation]
 \label{cor:dilate}\index{theorem!simultaneous dilation}
 Let $d\in\N$.  There is a Hilbert space $\cH$, a family $\fN_d$ of commuting normal  contractions
  on $\cH$, and an isometry $V:\mathbb \C^d\to \cH$ such that for each complex $d\times d$ contraction matrix $X$
  there exists a $T\in \fN_d$ such that 
\[
  \frac{1}{\sqrt2\ \th(2d)} X = V^* T V. 
\]
\end{cor}

\begin{proof}
Given any $d\times d$ complex contraction $X$, the matrices
\[
S= \frac12(X+X^*),\quad
K=\frac1{2i}(X-X^*)
\]
are self-adjoint contractions with
$S+iK=X$. 
Consider the $\R$-algebra $*$-homomorphism
\[
\begin{split}
\iota:M_d(\C) & \to M_{2d}(\R) \\
(a_{ij})_{i,j=1}^d &\mapsto 
\begin{pmatrix}
(\RE a_{ij})_{i,j=1}^d & (\IM a_{ij})_{i,j=1}^d \\
-(\IM a_{ij})_{i,j=1}^d & (\RE a_{ij})_{i,j=1}^d
\end{pmatrix}
\end{split}
\]
A straightforward calculation shows
\[
Y=\frac1{2}\begin{pmatrix} I_d \\ i I_d \end{pmatrix} ^*
\iota(Y) \begin{pmatrix} I_d \\ i I_d\end{pmatrix}  
\]
for any $Y\in M_d(\C)$. 

Let $\fC_{2d}$, $\cH$ and $V$ be as in Theorem \ref{thm:dilate}.  Then $\iota(S),\iota(K)$ dilate up to a factor $\frac1{\vartheta(2d)}$ to elements $\widehat S,\widehat {K}$ of $\fC_{2d}$.  Thus
\[
\frac1{\vartheta(2d)} (\iota(S)+i\, \iota(K)) = V^* (\widehat S+i \widehat {K})V ,
\]
whence
\[
\frac1{\vartheta(2d)} X = W^* (\widehat S+i \widehat {K})W ,
\]
where $W$ is the isometry
\[
W=\frac1{\sqrt2} V  \begin{pmatrix} I_d \\ i I_d\end{pmatrix} .
\]

Hence letting 
$\fN_{2d}=\frac1{\sqrt2}(\fC_{2d}+i\fC_{2d})$
we have
\[
\frac1{\sqrt2\vartheta(2d)} X = V^* N V
\]
for 
\[
N=\frac1{\sqrt2}(\widehat S+ i\, \hat K)\in\fC_d.
\]
It is clear that elements of $\fN_{2d}$ are pairwise commuting normal contractions.
\end{proof}
 
\subsection{Probabilistic theorems and interpretations}
 \label{sec:introprob}
  This section assumes only basic knowledge about the Binomial and Beta distributions\index{distribution}\index{distribution!Binomial}\index{distribution!Beta}
  and does not depend upon the rest of this introduction.  The proof of 
  Theorem \ref{thm:thetaExplicit} produced, as byproducts, several theorems
  on the Binomial and Beta distribution which are discussed here and
  in more detail in Section \ref{sec:shortprob}.

 We thank Ian Abramson for  describing a Bayesian perspective.

\def\mfT{\mathfrak T} 
 
 \subsubsection{Binomial distributions}
 With $\dd$ fixed,
   perform
 $\dd$ independent flips\index{coin flipping} of a
 biased  coin whose probability of coming up heads is $p.$ 
 Let $\SRV$ denote the random variable representing the number of heads which occur,
and let $P_p(\SRV=\head )$  denote the probability of getting exactly $\head $ heads.
 On the same probability space is the random variable 
$\mfT$ which represents the number of tails which occur. Of
course $\mfT= \dd-\SRV$ and the probability of getting
exactly $\tail$ tails is denoted $P_p(\mfT=  \tail )$.
The distribution of $\SRV,$
$${\rm Bin}(\dd,p; \head ):= {\dd \choose  \head } p ^ \head   (1- p)^\tail,
$$
 at $\head$ is {\bf Binomial} with parameters $p$ and $\dd$.
Our main interest will be behavior of functions of the 
form  $P_{p(\head)} ( \SRV \geq \head)$ for a function $p(\head)$  close to
$\sdd$.

The \CDF\  (CDF) of  a {\bf Beta Distributed}\index{beta distribution} random variable 
$\BRV$  
with shape parameters $\head, \tail$
is  the function of $x$ denoted
$P^{b(\head, \tail)}(\BRV\leq x)=I_x(\head,\tail)$.
Its mean is $\frac \head {\head + \tail}$ and
its \PDF\ (PDF)  is
\[
\varrho_{\head,\tail}(x)=
\frac1{B(\head,\tail)} x^{\head-1}(1-x)^{\tail-1}.
\]
When   $\head, \tail  $ are integers, 
The  CDF for the Binomial Distribution with $\head +\tail =\dd $ can be recovered via
\beq
\label{eq:PbB}
P_p(\SRV\geq \head) =I_p(\head , \tail +1  )=\Pbe(\BRV\leq p).
\eeq 
For the complementary  CDF using
$1- P_p(\SRV\geq \head) = 1- \Pbe(\BRV\leq p)$ 
gives
\beq
\label{eq:PbBcomp}
P_p(\SRV \leq \head -1)  = P_p(\SRV< \head) =  \Pbe(\BRV\geq p) . 
\eeq

 \subsubsection{Equipoints and medians}
 \label{sec:equi}
 Our results depend on the nature and estimates of  medians, means  or {\bf equipoints} (defined below) all  being measures
 of central tendency.
 Recall   for any random variable
  a \df{median} is defined to be an  $\head$ in the sample space
 satisfying
 $P(\SRV \leq  \head ) \geq \frac 1 2 $ and 
 $P(\SRV \geq  \head ) \geq \frac1 2 $. 
 For a  Binomial distributed random variable $P_\sdd$
 the median is  the mean is the mode is $\head$
  when $\head, \dd$ are integers.

 Given a binomially distributed random variable $\SRV$,
 we call  $\eih \in [0,1]$   an {\bf equipoint of $ \head $},\index{equipoint}
provided
\begin{equation}
\label{eq:halfwayprob-again}
 P_{\eih } (\SRV \geq  \head ) = P_{\eih} ( \SRV \leq   \head ) .
 \end{equation}
 Here  $ \head, \tail \in \N $,  and
 $\dd=\head + \tail$.
 Since  $P_{\eih } (\SRV \geq  \head ) + P_{\eih} ( \SRV \leq   \head ) \geq 1$, Equation 
  \eqref{eq:halfwayprob-again}
 implies  $\head $ is a median for  $\SRV$.
A median is in $\N$ for  Binomial and in $\RR$ for 
Beta distributed random variables. 
 In practice equipoints and  means  are close.
 For example,
 when the PDF is ${\rm Bin}(10, \frac{\head}{10})$ the mean is $\frac{\head}{10}$ 
one can compute $\eih$ 
for $\head=1, \dots {10}$:
\[
\resizebox{.99\hsize}{!}{
$
\begin{array}{c|ccccccccccccccccccc}
\head & 1 & 2 & 3 & 4 & 5 & 6 & 7 & 8 & 9 & 10 \\
\hline
e_{\head,10-\head} &
0.111223 & 0.208955 & 0.306089 & 0.403069 
& 0.5&  0.596931 & 0.693911 & 0.791045
& 0.888777 & 1
\end{array}
$
}
\]

 In contrast, the Beta Distribution is continuous,
 so for $\head,\tail \in \RR_{\geq0}$ with  $\dd=\head + \tail >0$
 we define $\eih$ by
 \beq
\label{eq:probEquiBet}
P^{b(\head +1 , \tail)} (\BRV\leq \eih)   = 
  \Pbe (\BRV\geq \eih),
  \eeq
and we call $\eih$ the \df{equipoint} of the  ${\rm Beta}(\head,\tail)$
distribution.  Equivalently,
$$
  \Pbe (\BRV\leq \sih) \ + \  P^{b(\head +1 , \tail)} (\BRV\leq \sih)  =1.
  $$
    In terms of the regularized beta function,  $\eih$ is determined by 
\begin{equation}\label{eq:sigiNTRO}
 I_{\eih }\left( s,t+1\right) + I_{\eih}\left(s+1, t\right) = 1.
\end{equation}

 When $\head,\tail$  are integers, the probabilities in 
 \eqref{eq:halfwayprob-again} and \eqref{eq:probEquiBet}  
 coincide, so the two definitions give the same $\eih$. Verifying this statement is an   
 exercise in the notations.  The connection between equipoints and the theory of the matrix cube 
 emerges in Section \ref{sec:simmhalf}.

 \begin{example} 
 Here is a concrete probabilistic interpretation of the equipoint $\eih$.
 Joe  flips a biased coin with probability $p$\index{coin flipping}
 of coming up  heads, but does not know $p$.
 After  $\head-1$ heads and 
 $\tail-1$ tails, the probability that $p$ is less than $r$ is 
 $I_r(\head,\tail)$ by Bayes' Theorem\footnote{\url{https://en.wikipedia.org/wiki/Checking_whether_a_coin_is_fair}}. 
 
 The equipoint $\sih$ pertains to the next toss of the coin. 
 If it is a head (resp.~tail), then $ b(\head +1, \tail)$ (resp.~$b(\head,\tail+1)$) 
 becomes the new distribution for  estimating  $p.$
 From \eqref{eq:probEquiBet},  the equipoint is defined  so that with the next toss
 {\it the probability after a head that $p$  is  at most $\sih$
 equals 
the probability after a tail that $p$  is at least $\sih$.}
\end{example}

The next two subsections contain more information on
equipoints and how they compare to means and medians.

\subsubsection{Equipoints compared to medians}
 \label{sec:equinature}
 Here is a  basic property of\index{equipoint}\index{median}
equipoints versus  medians and means.
Let $\NNhalf$ denote the set of all positive half-integers, i.e., all $\head= \frac s 2 $ 
with $s \in \NN$.

\begin{thm}
\label{thm:simmonsProb}
  For $\dd \in\NNhalf$ and $\head, \tail \in \NNhalf$  with $\head+  \tail=\dd$, 
  if $\ddhalf \le s<\dd$,  then  
\begin{align}
 \label{eq:someconj22}
P_{\frac  \head  \dd}(\SRV<  \head )\  & \leq 
  P_{\frac  \head  \dd} (\SRV> \head )
 \qquad \text{provided }\ \head, \tail, \dd \in \N; \\
 \label{eq:someconj22b}
P^{b(\head, \tail + 1)}\left(\BRV \geq \sdd\right)
& \leq  P^{b(\head +1, \tail) } \left(\BRV \leq \sdd\right).
\end{align}
 Both \eqref{eq:someconj22} and \eqref{eq:someconj22b} are equivalent to
 \beq
 \label{eq:sigmed}
  \sih \le \frac{\head}{\dd}.
\eeq
We also have the lower bound 
\beq\label{eq:siglow}
     \frac{\head+1}{\head +\tail+2}\le  \sih,
\eeq
for real numbers $\head \ge \tail\geq1$.
\end{thm}

\begin{rem}\rm
 \label{rem:simmons}
The inequality \eqref{eq:someconj22b} for integer $\head, \tail$
 is Simmons' Theorem, cf.~\cite{PR07}.\index{theorem!Simmons'}
The
lower bound is new. 
  For half-integer $\head, \tail$
  both our upper and lower bounds are new.
\end{rem}

\begin{proof}  
 The inequality  \eqref{eq:someconj22b} implies
 \eqref{eq:someconj22} by \eqref{eq:PbB} and \eqref{eq:PbBcomp}.
 However,  \eqref{eq:someconj22b}
 and  $ \sih \le \frac{\head}{\dd}$
 is the content of  Theorem \ref{simmons+}.
 The lower bound \eqref{eq:siglow}  is Proposition \ref{prop:preLowBdSIS}.
\end{proof}

 Computer experiments lead us to believe \eqref{eq:sigmed} is
 true for  real numbers:

\begin{conj}
  For  $\head, \tail \in \RR_{>0} $ with
  $\head\geq\tail$, inequality
\eqref{eq:someconj22b} holds. Equivalently,
$\ds  \sih \le \frac{\head}{\head+\tail}.$
\end{conj}

As a side-product of our quest for bounds on the equipoint
we obtain new upper bounds on the median $m_{\al,\be}$ of the Beta Distribution $\text{Beta}(\al,\be).$

\begin{cor}\label{cor:newmedINTRO}
Suppose $\head,\tail\in\mathbb R$. If  $1\le\head\le\tail$ and $\head+\tail\ge3$, then
$$
\mu_{\head,\tail}:=\frac{\head}{\head+\tail}\le \ m_{\head,\tail} \  \le  \mu_{\head,\tail}+\frac{\head-\tail}{(\head+\tail)^2}.
$$
\end{cor}

\begin{proof}
The lower bound is known, see \cite{GM77,pyy}.
The upper bound is proved in Corollary \ref{cor:newmed}.
\end{proof}

\sssec{Monotonicity  of the CDF}
A property of the functions 
\beq
\Phi(\head):=   P^{b(\head,\dd-\head+1)}(\BRV\leq e_{\head, \dd - \head} )
\qquad \text{and}  \qquad
\hat \Phi(\head):=   
P^{b(\head,\dd -\head  +1)}
\left(\BRV \leq \frac \head {\dd} \right), 
\eeq
where $\BRV$ is a Beta distributed random variable, 
is {\it one step monotonicity.}

\begin{thm}
\label{thm:incr}
Fix $0< \dd \in \RR$.
\ben[\rm(1)]
\item
\label{it:hatPhiInc}
$\hat \Phi(\head ) \leq   \hat \Phi(\head + 1) $ 
for $\head \in \RR$ with $ \frac \dd 2 \leq \head < \dd -1 $;
\item
\label{it:PhiInc}
$ \Phi(\head ) \leq    \Phi(\head + 1) $ 
for $\head, \dd \in \NNhalf$ with $ \frac \dd 2 \leq \head < \dd -1 $.
\een
\end{thm}

\begin{proof}
 See Section \ref{sec:monot}.
\end{proof}

Computer experiments lead us to believe monotonicity of $\Phi$ 
holds for real numbers  $\head, \tail.$ 

\begin{conj}
   $ \Phi(\head ) <    \Phi(\tilde \head) $ 
for $\head, \tilde \head,\dd \in \RR$ with  $ 0 < \frac \dd 2 \leq \head < \tilde \head
< \dd  $.
\end{conj}

The functions $\Phi$ and $\hPhi$ are based on the CDF.
Analogous results hold for the PDF and these appear in 
Section \ref{sec:CDFmonot}.

The monotonicity result of Theorem \ref{thm:incr} allows us to identify  
the minimizers of $\Phi$.
Indeed the following theorem restates 
Theorem \ref{thm:thetaExplicit}  in probabilistic terms.

\begin{thm}
For $ \dd \in \NNhalf$ and $\frac{\dd}2 \leq \head < \dd -1 $,
the function 
$\Phi$ of $\head \in \NNhalf$ takes its minimum at
\ben[\rm(1)]
\item
$\ds\head=\tail = \ddhalf$ \ if $\dd \in  \NN$;\\[.1cm]
\item
$\ds\head= \dd + \frac 1 2$ and $\tail = \dd - \frac 1 2$
if $\dd \in\NNhalf\setminus\NN$.
\een
\end{thm}

\subsection{Reader's guide}
 \label{sec:guide}
  The rest of this article is organized as follows.  Results relating dilations to free spectrahedral inclusions
  needed for the proofs of the results for the matrix cube are collected in Section \ref{sec:liftslite}.
  Further general results on free spectrahedral inclusions and dilations appear in Section \ref{sec:generally free}.
  The results of Section \ref{sec:expectaverage} simplify the identification of the optimum $\vartheta(d)$ as
  defined by Equation \eqref{eq:BTbd}. They also  identify, implicitly,  constrained versions of this optimum.  The results of the previous sections
  are combined in Section \ref{sec:thetaBest} to prove Theorem \ref{thm:dilate} and Theorem \ref{thm:sharp},
  as well as the weaker version of  Theorem \ref{thm:BtN} which asks that the  matrices $B$ have size $d$,
  and not just rank at most $d$.  In Section \ref{sec:prequest}, the constrained optima from Section \ref{sec:expectaverage}
  are identified, still implicitly, in terms of the regularized incomplete beta function, a result needed to
  complete the proof of Theorem \ref{thm:BtN} in Section \ref{sec:rankvsize} as well as in the remaining sections of the paper.  
  Theorem \ref{thm:thetaExplicit} is reformulated in terms of the beta function in advance of the following
  three sections which together establish Theorem \ref{thm:thetaExplicit}. 
  A half-integer generalization of Simmons' Theorem, inspired by the strategy in \cite{PR07} using 
  two step monotonicity,  is the topic of Section \ref{sec:simmhalf}.  A new lower bound for the median of the 
  Beta distribution and bounds for the equipoint appear in Section \ref{sec:lowBdSIS} and the bounds for the
  equipoint are used in Section \ref{sec:nov} to complete the proof of Theorem \ref{thm:thetaExplicit}.  
  Estimates for $\vartheta(d)$ in the case that $d$ is odd appear in Section \ref{sec:odd}. 
Section \ref{sec:balls} considers the problem of including the unit ball in $\R^g$ into a spectrahedron, and uses dilation theory to prove the worst case error inherent in its free relaxation, namely $g$.
  Finally,
  further probabilistic results and their proofs are exposited in Section \ref{sec:shortprob}.

The reader interested only in probabilistic results can 
 proceed to Sections \ref{sec:simmhalf}, \ref{sec:lowBdSIS}
  and \ref{sec:shortprob}.
 The reader interested only in the matrix cube problem can skip Section \ref{sec:generally free}; whereas the
 reader interested only in 
 dilation results 
 (absent formulas for $\vartheta(d)$)
 can focus on the sections up through 
 and including Section \ref{sec:generally free}. 

The original version of this manuscript\footnote{\url{https://arxiv.org/abs/1412.1481}} treated 
the case where the rank (size) $d$ of the LMI defining pencil of the containing spectrahedron is fixed, but the number of variables $g$ is not.
Subsequently, we and Davidson, Dor-On, Shalit and  Solel independently  consider inclusion problems
for balls (where $g$ is fixed, but $d$ is not). Our results in this direction appear in Section \ref{sec:balls}. 
We highly recommend the posting \cite{DDSS} for its many interesting results, including a far reaching generalization of 
the symmetry based inclusion result of Proposition \ref{prop:small g} and a fascinating connection with the theory of frames.
In addition \cite{DDSS} extends many spectrahedral inclusion results to the setting of (infinite dimensional) operators.
(See also \cite{Zal} which considers the operator setting with an emphasis on the unbounded domains
as well as various polynomial Positivstellens\"atze.) 
They also show  if a tuple $X$ of matrices  dilates to a commuting tuple $N$ of normal operators,
then $X$ dilates to a commuting tuple of normal matrices $T$ satisfying the inclusion $\sigma(T)\subset \sigma(N).$
In particular, using the (infinite dimensional) optimum bound $\vartheta(d)$ identified in this article, it follows that given a tuple $X$ (finite set) of
$d\times d$ symmetric contractions, there exists a tuple $T$ of commuting symmetric matrices and an isometry $V$ such that $X= \vartheta(d) V^* TV$.


\section{Dilations and Free Spectrahedral Inclusions}
\label{sec:liftslite}
\def\insc{r}
 This section presents preliminaries on free spectrahedral inclusions
 and dilations,\index{dilation}\index{free spectrahedral inclusion} tying 
 the existence of dilations to appropriate   commuting tuples (the commutability index)
 to free spectrahedral inclusion.

The following proposition gives a sufficient condition for the 
 inclusion of one spectrahedron in another. It will later be applied to $\cube,$  the  free 
cube.  Recall the definitions of $L_A(x),$ $L_A(X)$, $\cD_{L_A}$ and $\fs_{L_A}$ from
Section \ref{sec:op-dil}.   Any $r>0$ (and necessarily $1\ge r$) with the property that the inclusion $\fs_{L_A}\subseteq \fs_{L_B}$
 implies the  inclusion $r\cD_{L_A}\subseteq \cD_{L_B}$ provides an estimate for the
 error in testing spectrahedral inclusion using the free spectrahedral inclusion as a relaxation.
 Indeed, suppose that $\fs_{L_A}\subset \fs_{L_B}$, but, for $t>1$, that $t \fs_{L_A}\not\subseteq \fs_{L_B}.$
 The free relaxation amounts to finding the largest $\rho$ such that  $\rho\cD_{L_A}\subseteq\cD_{L_B}$ and concluding
 that necessarily $\rho \fs_{L_A}\subseteq \fs_{L_B}$.  Since $\rho\ge r$, it follows that $r$ then provides
 a lower bound for the error.

\begin{prop}
 \label{prop:commute-include}
Suppose $A$ is a $g$-tuple of symmetric $m\times m$ matrices, 
$d$ is a   positive integer,  $\insc >0$ and 
for each $X\in\cD_{L_A}(d)$ there is a Hilbert space $\mathcal H$, an isometry $W\colon\R^d\to \mathcal H$ and a tuple $T=(T_1,\dots,T_g)$ of
commuting bounded self-adjoint operators $T_i$ on $\mathcal H$ with joint spectrum contained in $\fs_{L_A}$
such that $\insc  X_i=W^*T_iW$ for all $i\in\{1,\dots,g\}$. If $B$ is a tuple of $d\times d$ symmetric matrices and 
$\fs_{L_A}\subset \fs_{L_B},$ then  $\insc  \cD_{L_A} \subset \cD_{L_B}$.
\end{prop}

\begin{rem}\rm
 It turns out that in the case of $\cD_{L_A}=\cubeg$,\index{cube}\index{free cube} one only needs to assume that $B$ is a tuple of $m\times m$ symmetric matrices
  each of rank at most $d$. See \cite{BtN} or  Section \ref{sec:rankvsize} of this paper, where an elaboration on the 
  argument below plus special properties of the $\cubeg$ are used to establish this result. 
\end{rem}

 The proof of the proposition employs the following lemma which will also be used in Section \ref{sec:rankvsize}.

\begin{lemma}
 \label{lem:ddoesit}
  If $A$ is a $g$-tuple of symmetric $m\times m$ matrices, $B$ is a 
  $g$-tuple of symmetric $d\times d$ matrices and 
 if $\rh \cD_{L_A}(d)\subset  \cD_{L_B}(d)$, then
  $\rh  \cD_{L_A}(n)\subset \cD_{L_B}(n)$ for every $n$. 
\end{lemma}

\begin{proof}
Suppose $\rh \cD_{L_A}(d)\subset  \cD_{L_B}(d)$, $n\in\N$, $n\ge d$ and $(X_1,\dots X_g)\in\cD_{L_A}(n)$. 
  We have to show that $\rh (X_1,\dots X_g)\in \cD_{L_B}(n)$.
 Given $x\in\mathbb R^d \otimes \mathbb R^n$, write
\[
  x =\sum_{s=1}^d e_s \otimes x_s,
\]
where $e_s$ are the standard basis vectors of $\R^d$. 
  Let $M$ denote the span of $\{x_s:1\le s\le d\}$
  and let  $P$ denote the projection onto $M$. Then
\begin{equation}
 \label{eq:reftag}
 \begin{split}
\big   \langle (\sum B_j\otimes X_j )x,x\big\rangle 
    & = \sum_{j,s,t} \langle B_je_s,e_t \rangle \, \langle X_jx_s,x_t\rangle\\
  & = \sum_{j,s,t} \langle B_je_s,e_t \rangle \, \langle PX_jP x_s,x_t\rangle\\
  & =\big \langle (\sum B_j \otimes PX_jP ) x,x\big\rangle .
 \end{split}
\end{equation}
 Now $\rh PXP= \rh (PX_1P,\dots,PX_gP)\in \rh \cD_{L_A}(r)\subseteq \cD_{L_B}(r)$ where $r\le d$
  is the dimension of $M$. Hence, by Equation \eqref{eq:reftag},
\[
\rh \big \langle (\sum B_j\otimes X_j) x,x \big\rangle \le  \|x\|^2.  
\]
 
  For $n<d$, simply taking a direct sum with $0$ (of size $n-d$) produces the tuple $X\oplus 0 \in \cD_{L_A}(d)$ and
  hence $\rh X\oplus 0 \in \cD_{L_B}(d)$ by hypothesis.  Compressing to the first summand gives $\rh X\in\cD_{L_B}(n)$ and
  the proof is complete.  
\end{proof}

\begin{proof}[Proof of Proposition {\rm\ref{prop:commute-include}}]
Suppose $\fs_{L_A}\subset \fs_{L_B}$ and let $X\in\cD_{L_A}(d)$. 
Choose a Hilbert space $\mathcal H$, an isometry $W\colon\R^d\to \mathcal H$ and a tuple $T=(T_1,\dots,T_g)$ of
commuting bounded self-adjoint operators $T_i$ on $\mathcal H$ with joint spectrum contained in $\fs_{L_A}$
such that $\insc  X_i=W^*T_iW$ for all $i\in\{1,\dots,g\}$. Then the joint spectrum of $T$ is contained in $\fs_{L_B}$.
Writing $B=\sum_{i=1}^gB_ix_i$ with symmetric $B_i\in\smatn$, we have to show that
$\insc  \sum_{i=1}^gB_i\otimes X_i\preceq I_{dn}$.
Let $E$ denote the joint spectral measure of $T$ whose support is contained in $\fs_{L_B}\subseteq\R^g$.
Then
\begin{align*}
 \insc  \sum_{i=1}^gB_i\otimes X_i&=\sum_{i=1}^gB_i\otimes W^*T_iW\\
&=\sum_{i=1}^gB_i\otimes W^*\left(\int_{\fs_{L_B}}y_idE(y)\right)W\\
&=\int_{\fs_{L_B}}\underbrace{\left(\sum_{i=1}^gB_iy_j\right)}_{\preceq I_d}\otimes\underbrace{W^*dE(y)W}_{\succeq0}\\
&\preceq\int_{\fs_{L_B}}I_d\otimes W^*dE(y)W\\
&= I_d\otimes W^*\left(\int_{\fs_{L_B}}dE(y)\right)W\\
&= I_d\otimes W^*\id_{\mathcal H}  W= I_d\otimes W^*W= I_d\otimes I_n= I_{dn}.
 \end{align*}
Hence $X\in\cD_{L_B}(d)$.  An application of Lemma \ref{lem:ddoesit} completes the proof.
\end{proof}

\section{Lifting and Averaging}
 \label{sec:average}
 This subsection details the connection between averages of matrices over the orthogonal
 group and the dilations\index{averaging}\index{dilation} 
 of tuples of symmetric matrices  to commuting tuples of
contractive self-adjoint operators, a foundation for the proof of Theorem \ref{thm:dilate}.

 Let $M_d$ denote the collection of $d\times d$ matrices. 
 Let $O(d)\subset M_d$ denote the orthogonal group and let $dU$ denote 
the Haar measure on $O(d)$.  Let $\mfD(d)$  denote the collection 
 of measurable functions $D\colon O(d) \to M_d$ which take diagonal
 contractive values. Thus, letting $\mathfrak{M}(O(d),M_d)$ denote 
  the measurable functions from $O(d)$ to $M_d$, 
\begin{equation}
 \label{eq:defDd}
  \mfD(d) =\{ D\in \mathfrak{M}(O(d),M_d) \, : \,  D(U) \mbox{ is diagonal and } \|D(U)\|\le 1 \mbox{ for every } U\in O(d)\}.
\end{equation}

 Let $\mathcal H$ denote the Hilbert space \index{$\mathcal H$}
\begin{equation}
 \label{eq:H}
   \mathcal H=\R^d\otimes L^2(O(d))=L^2(O(d))^d=L^2(O(d),\R^d).
\end{equation}
 Let $V:\mathbb R^d\to \mathcal H$ denote the mapping \index{$V$}
\begin{equation}
\label{eq:V}
 Vx(U) =x.
\end{equation}
 Thus, $V$ embeds $\mathbb R^d$ into $\mathcal H$ as the constant functions. 
 For $D\in \mfD(d)$, define $\TM_D:\mathcal H\to \mathcal H$ by 
\[
(\TM_Df)(U) = U D(U) U^* f(U)
\]
for all $U\in O(d)$. Because $D(U)$ is pointwise a symmetric contraction for all $U\in O(d)$, $\TM_D$ is a self-adjoint contraction on $\mathcal H$.
Let
\begin{equation}
 \label{eq:defcCd}
 \fC_d =\{\TM_D:D\in \mfD(d)\}.
\end{equation}

\begin{rem}\rm
 \label{rem:VvD}
    Alternately one could define $V$ by $Vx(U) = U^* x$ instead of conjugating $D(U)$ by $U$ and $U^\ast$.   
\end{rem}

\begin{lemma}
 \label{lem:lift1}
 Each  $\TM_D$ is a self-adjoint contraction. 
\end{lemma}

\begin{lemma}
 \label{lem:lift2}
 If $D,E\in \mfD(d),$ then $\TM_{DE}=\TM_{D}\circ \TM_{E}=\TM_{E}\circ \TM_{D}$.
 Thus, $\TM_{D}$ and $\TM_{E}$ commute. 
\end{lemma}

\begin{proof}
  The result follows from the fact that $D$ and $E$ pointwise commute and
  hence the functions $U\mapsto U D(U)U^\ast$ and $U\mapsto UE(U)U^\ast$ 
  pointwise commute. 
\end{proof}

\begin{lemma}
 \label{lem:Viso}
 The mapping $V$ is an isometry and its adjoint 
\[V^*\colon L^2(O(d),\R^d)\to\R^d\] is given by
\[V^*(f)=\int_{O(d)}f(U)dU\]
for all $f\in L^2(O(d),\R^d)$.
\end{lemma}

\begin{proof}
For all $x\in\R^n$ and $f\in L^2(O(d),\R^d)$, we have
\[\langle Vx,f\rangle=\int_{O(d)}\langle x,f(U)\rangle dU=\left\langle x,\int_{O(d)}f(U)dU\right\rangle,\]
thus computing $V^*$. Moreover, $V$ is an isometry as
\[\langle Vx,Vx\rangle=\left\langle x,\int_{O(d)}xdU\right\rangle=\langle x,x\rangle.\qedhere\]
\end{proof}

\begin{lemma}
 \label{lem:averagelift}
 For  $D\in\mfD(d),$ 
\[
   V^*\TM_DV =  \int_{O(d)} U D(U) U^* dU.
\]
\end{lemma}

For notation purposes, let \index{$C_D$}
\begin{equation}
 \label{eq:defineCsub}
  C_D = \int_{O(d)} U D(U) U^* dU.
\end{equation}

\begin{proof}
For  $x\in\R^d$, we have,
\begin{align*}
C_Dx&=\left(\int_{O(d)} U D(U)U^\ast dU\right)x\\
&=\int_{O(d)} U D(U)U^\ast x\,dU\\
&=\int_{O(d)} U D(U)U^\ast \big((Vx)(U)\big)\,dU\\
&=\int_{O(d)}(\TM_D(Vx))(U)\,dU\\
&=V^*(\TM_D(Vx)). \qedhere
\end{align*}

\end{proof}

\begin{rem}\rm
 \label{rem:choosewisely}
   Suppose $\mathcal S$ is a subset of $\mfD(d)$ and consider the family of symmetric matrices $(C_D)_{D\in \mathcal S}$.
    The lemmas in this subsection imply that this family dilates to the commuting family of 
   self-adjoint  contractions $(\TM_D)_{D\in \mathcal S}$. 
   Let $\mu(D):=\frac{1}{\|C_D\|}$ for $D\in \mfD(d)$ and suppose
\[
 \mu := \sup\{\mu (D):D\in\mathcal S\}
\]
  is finite. 
  It follows that the collection of symmetric contractions $(\mu(D) C_D)_{D\in\mathcal S}$
  dilates to the commuting family $(\TM_{\mu(D) D} =\mu(D) \TM_D)_{D\in\mathcal S}$ of self-adjoint operators 
   of operator norm at most $\mu,$ 
 \[
  \mu(D) C_D  =  \mu(D) V^* M_D V
\]
 for all $D\in\mathcal S$.

Our aim, in the next few sections, is to turn this construction around. Namely, given a family $\mathcal C \subset M_d$
  of symmetric contractions (not necessarily commuting), we hope to find a $t\in[0,1]$ (as large as possible)
  and a family $(F_C)_{C\in\mathcal C}$ in $\mfD(d)$ such that 
\[
   t C = V^* \TM_{F_C} V.
\]
for all $C\in\mathcal C$. Any $t$ so obtained feeds into the hypotheses of Proposition \ref{prop:commute-include}.
\end{rem}

\section{A Simplified Form for $\vartheta$}
 \label{sec:expectaverage}
 Given a symmetric matrix $B$, let $\sign_0(B)=(p,n),$ 
  where $p,n\in\mathbb N_0$ denote the number of nonnegative  and negative eigenvalues of $B$
 respectively. 
 It is valuable to think of the optimization problem \eqref{eq:BTbd}
 over symmetric matrices $B$
 in two stages
 based on the signature.
 \index{$\sign_0(B)=(s,t)$}
 Let
 \beq
 \label{eq:BTbdst}
\ka_*(s,t)= \min_{\substack{B\in \smat_d\\
 \sign_0(B)=(s,t)\\[.5mm]\trace|B|=d}}
 \int_{S^{d-1}} | \xi^*B\xi | \,  \, d\xi
 \eeq
 and note that the minimization  \eqref{eq:BTbd}
 is 
 \beq
 \label{eq:BTbdm}
\frac{1}{\th(d)} =\ka_*(d):= \min\{\ka_*(s,t): s+t=d\}. 
 \eeq

  \index{$J(s,t;a,b)$}
 Given $s,t\in\mathbb N_0$ and $a,b\ge 0$, let 
\[
 J(s,t;a,b):= aI_s \oplus -b I_t.
\] 
  Thus $J(s,t;a,b)$ is the diagonal matrix whose diagonal reads
\[\underbrace{a,\dots,a}_{\text{$s$ times}},\underbrace{-b,\dots,-b}_{\text{$t$ times}}.\]

We simplify the optimization problem \eqref{eq:BTbd} as follows.  
   The first step  consists of showing, for fixed $s,t\in\mathbb N_0$ (with $s+t=d$),   that the optimization  can be performed over 
   the set $J(s,t;a,b)$ for  $a,b\ge 0$ such that $as+bt=d$.
  The second step is to establish, again for  fixed integers $s,t$, an implicit criteria to identify
the values of $a,b$ which optimize \eqref{eq:BTbdst}. 
 Toward this end we introduce the following notations. 
Define
\begin{equation}
 \label{eq:kstab}
  \kappa(s,t;a,b):=\int_{S^{d-1}} |\xi^* J(s,t;a,b)\xi| \, d\xi. 
\end{equation}\index{$\kappa(s,t;a,b)$}
 Finally, let for $1\le j\le s$, 
\begin{equation}
 \label{def:al}
\alpha(s,t;a,b) = \int_{S^{d-1}} \sgn \big[\xi^* J(s,t;a,b)\xi\big] \, \xi_j^2 \, d\xi,
\end{equation}\index{$\alpha(s,t;a,b)$}
 and, for $s+1 \le j \le d$, 
\begin{equation}
 \label{def:be}
\beta(s,t;a,b) = -\int_{S^{d-1}} \sgn \big[\xi^* J(s,t;a,b)\xi\big] \xi_j^2 \, d\xi.
\end{equation}\index{$\beta(s,t;a,b)$}
 (It is straightforward to check that $\alpha$ and $\beta$ are independent
 of the choices of $j$.)

\index{$\alpha$}\index{$\alpha(s,t;a,b)$} \index{$\beta$} \index{$\beta(s,t;a,b)$}

\begin{rem}\rm
 \label{rem:alpha}
 The quantities  $\alpha = \alpha(s,t;a,b)$ and $\beta=\beta(s,t;a,b)$ are
  interpreted in terms of the regularized beta function in Lemma \ref{lem:compalpha}.
\end{rem}

\begin{prop}
 \label{prop:starastarb}
  For each $d \in \N$, $s,t\in\N_0$  with $s+t=d$,  the 
 minimum in Equation \eqref{eq:BTbdst}   is  achieved at a $B$ of the form $J(s,t;a,b)$ where $a,b>0$
  and $as+bt=d$,  and
 \begin{equation*}\index{$\ka_*(s,t)$}
  \ka_*(s,t) =\min_{\substack{
  as + bt =d}} 
 \int_{S^{d-1}} |\xi^* J(s,t;a,b)\xi| \, d\xi.
\end{equation*}
Moreover, $\ka_*(d,0)=ka_*(0,d)\ge\ka_*(s,t)$ and for $s,t\in\mathbb N$, the minimum occurs at a 
 pair $\aofst,\bofst$ uniquely determined by  $\aofst,\bofst \ge 0$, $s\aofst + t \bofst =d$ and 
$$
   \alpha(s,t;\aofst,\bofst) = \beta(s,t;\aofst,\bofst),
$$
 so that $ \ka_*(s,t) =\kappa(s,t;a(s,t),b(s,t))$.
In particular, 
\begin{equation*}
  \ka_*(s,t) 
  =d \alpha(s,t;\aofst,\bofst) = d  \beta(s,t;\aofst,\bofst).
  \end{equation*}
  Consequently, the minimum in Equation \eqref{eq:BTbd} 
  is achieved and is
 \beq
\label{eq:kst=beta}
\min_{\substack{
   s+t =d   
  }}  \ka_*(s,t)=  \min_{\substack{
   s+t =d   
  }}  \ \  d \ \beta(s,t;\aofst,\bofst). 
\eeq
  It  occurs at a $B$ of 
  the form $J(s,t;a(s,t),b(s,t))$.
\end{prop}

\begin{proof} 
   Let $\mathcal T$ denote the set of symmetric $d\times d$ matrices $B$ such that $\trace(|B|)=d$
  and note that $\cT$
   is a compact subset of $\mathbb S_d$. Hence \eqref{eq:BTbd} is well defined in that 
the    infimum in the definition
  of $\vartheta(d)$ is indeed a minimum.  
  Fix a $B\in\mathcal T$
and   suppose $B$ has $s$ nonnegative eigenvalues and $t$ negative eigenvalues
  (hence $s+t=d$).  
  Without loss of generality, assume that $B$ 
  is diagonal 
   with first $s$ diagonal entries $a_1,\dots,a_s$ nonnegative
 and last $t$ diagonal entries $-b_{s+1},\dots,-b_d$ negative (thus $b_j>0$). 
 Thus,
\[
  \int_{S^{d-1}} |\xi^* B \xi| \, d\xi =
    \int_{S^{d-1}} \Big| \sum_{j=1}^s  a_j \xi_j^2  - \sum_{j=s+1}^{d} b_j \xi_j^2 \Big| \, d\xi.
\]
  
  Let $\Sigma$ denote the subgroup of the group of permutations
  of size $n$ which leave invariant the sets $\{1,\dots,s\}$ and  
  $\{s+1,\dots,n\}$. Each 
  $\sigma\in \Sigma$ gives rise to a permutation matrix 
  $V_\sigma$.  It is readily checked that 
\[
  \int_{S^{d-1}} |\xi^* V_\sigma^\ast B V_\sigma \xi | \, d\xi  
    = \int_{S^{d-1}} \Big| \sum_{j=1}^s a_{\sigma(j)} \xi_j^2 - \sum_{j=s+1}^{d} b_{\sigma(j)} \xi_j^2 \Big|   \, d\xi .
\]
 Let $N$ denote the cardinality of $\Sigma$ and 
  note that $a = \frac{1}{N} \sum_{\sigma \in \Sigma} a_{\sigma(j)}$,
  $b = \frac{1}{N} \sum_{\sigma \in \Sigma} b_{\sigma(j)}$
    are independent of $j$.
Thus,
 \[
   \frac{1}{N} \sum_{\sigma\in \Sigma} V_\sigma^* B V_\sigma 
     = \sum_{\sigma\in\Sigma} \mbox{diag}\begin{pmatrix} a_{\sigma(1)} & \dots & a_{\sigma(s)} & - \ b_{\sigma(s+1)} & \dots & - b_{\sigma(d)}\end{pmatrix}
     = J(s,t;a,b).
\]

  Further, $as+bt=d$ (which depends on averaging over the subgroup $\Sigma$
  rather than the full symmetric group) and hence $J(s,t;a,b)\in\mathcal T$.  Therefore, 
  \[
 \begin{split}
 \int_{S^{d-1}} \big |\xi^* J(s,t;a,b)\xi \big | \, d\xi &= 
 \int_{S^{d-1}} \Big | \xi^* \big(\frac{1}{N} \sum_{\sigma\in\Sigma}  V_\sigma^\ast B V_\sigma  \big) \xi \Big | \, d\xi \\
  & \le  \frac{1}{N}  \sum_{\sigma\in\Sigma} \int_{S^{d-1}}| \xi^* V_\sigma B V_\sigma \xi |\, d\xi 
  = \frac 1 N
  \sum_{\sigma\in\Sigma}   \int_{S^{d-1}} |\xi^* B\xi | \, d\xi\\
 & =  \int_{S^{d-1}} |\xi^* B\xi | \, d\xi.
\end{split}
\]
Thus with $s,t \geq 0$ and $a,b \ge 0$ 
$$
 \min_{\substack{ s+t =d      
 \\[.5mm] as + bt =d}} 
 \int_{S^{d-1}} \big |\xi^* J(s,t;a,b)\xi \big | \, d\xi
\leq
 \min_{\substack{B\in \smat_d\\[.5mm]\trace|B|=d}}
 \int_{S^{d-1}} | \xi^*B\xi | \,  \, d\xi.
$$
By compactness of the underlying set for fixed $d$
the minimum  is attained; of course on a diagonal  matrix.

 Finally to write $\ka(s,t;a,b)$ in terms of $\alpha$ and $\beta$, note that
\[
\begin{split}
  \ka(s,t;a,b) &= \int_{S^{d-1}} \sgn \big[\xi^* J(s,t;a,b)\xi\big]  
\  \big(\xi^* J(s,t;a,b)\xi\big)
   \, d\xi \\
& = \int_{S^{d-1}} \sgn \big[\xi^* J(s,t;a,b)\xi\big]  \big( a\sum_{j=1}^s \xi_j^2 - b\sum_{j=s+1}^d \xi_j^2 \big) \, d\xi \\
&     = a s\alpha(s,t;a,b)+bt \beta(s,t;a,b).
  \end{split}
\]

To this point it has been established that there is a minimizer 
of the form $B=J(s,t;a,b)$.
 First  note that $\alpha,\beta\le \frac{1}{d}$, since, for instance,
\[
  d\, \alpha(s,t;a,b) \le d\, \int_{S^{d-1}} \xi_j^2 \, d\xi = \int_{S^{d-1}} \sum_{m=1}^d \xi_m^2\, d\xi = \int_{S^{d-1}} d\xi = 1.
\]
 Hence,
\[
 \ka(s,t;a,b) = as \alpha(s,t;a,b)+ bt \beta(s,t;a,b) \le \frac{1}{d}(as+bt)=1.
\]
Moreover, in the case that $s=d$ and $t=0$, then $B=I$ and $\ka_*(d,0)=1$. 
 Hence, $\ka_*(d,0)\ge \ka_*(s,t)$.  Turning to the case $s,t\in\mathbb N_0$,
 observe if $b=0$, then $a=\frac{d}{s}$ and  $\ka(s,t;\frac ds,0)=1$ and similarly, $\ka(s,t;0,\frac dt)=1$.
 Hence, for such  $s,t$,  the minimum is achieved at a point in the interior of the set $\{as+bt=d: a,b\ge 0\}$. 
 Thus, it can be assumed that minimum occurs at $B=J(s,t;a_*,b_*)$ for some $a_*,b_* >0$.

 Any other $J(s,t,a,b)$, for $a,b$ near $a_*,b_*,$ can be written as
\[
  J(s,t;a_*,b_*) + \lambda J(s,t;t,-s)
\]
 (in particular, $a_* + \lambda t >0$ as well as $b_* - \lambda s >0$). 
 By optimality of $a_*,b_*$, 
\[
 \begin{split}
  0 & \le  \int |\xi^* J(s,t;a,b) \xi|\, d\xi - \int |\xi^* J(s,t;a_*,b_*) \xi| \, d\xi \\
     & = \int \big ( \sgn [\xi^* J(s,t;a,b) \xi] - \sgn [\xi^* J(s,t;a_*,b_*) \xi] \big) \, \xi^* J(s,t;a_*,b_*)\xi  \, d\xi \\
         & \phantom{={}}  + \lambda \int \sgn [\xi^* J(s,t;a,b) \xi] \, \xi^* J(s,t;t,-s) \xi \, d\xi,        
 \end{split}
\]
  where the integrals are over $S^{d-1}$. Observe that the integrand of the first integral on the right hand side is
 always nonpositive and is negative on a set of positive measure. Hence this integral is negative. 
Hence,
\begin{equation}
 \label{eq:opt}
 0 <   \lambda \int \sgn [\xi^* J(s,t;a,b) \xi] \, \xi^* J(s,t;t,-s) \xi \, d\xi.
\end{equation}
 Choosing $\lambda>0$,  dividing Equation \eqref{eq:opt} by $\lambda$ and letting $\lambda$ tend to $0$ gives
\[
 0\le \int \sgn [\xi^* J(s,t;a_*,b_*) \xi] \, \xi^* J(s,t;t,-s) \xi \, d\xi.
\]
On the other hand, choosing $\lambda<0$, dividing by $\lambda$ and letting $\lambda$ tend to $0$ gives
the reverse inequality. 


 Hence,
\[
  0 = \int \sgn [\xi^* J(s,t;a_*,b_*) \xi] \, \xi^* J(s,t;t,-s) \xi \, d\xi
      = st \, \left ( \alpha(s,t;a_*,b_*)-\beta(s,t;a_*,b_*) \right).
\]
 
  Finally, the uniqueness of $a_*,b_*$ follows from the strict inequality in Equation \eqref{eq:opt},
  since $\lambda\ne 0$ corresponds exactly to  $(a,b)\ne (a_*,b_*)$.  We henceforth denote
  $(a_*,b_*)$ by $(a(s,t),b(s,t))$. In particular, $(a(s,t),b(s,t))$ is uniquely determined 
  by $a,b\ge 0$, $as+bt=d$ and $\alpha(s,t;a,b)=\beta(s,t;a,b)$.
\end{proof}

Note from the proof a limitation
of our $\alpha =\beta$ optimality condition.
It was obtained by fixing
$s,t$ and then optimizing over $a,b\geq0$.
Thus it sheds no light on subsequent minimization over
$s,t$.  To absorb this information 
requires many of the subsequent sections of this paper.

\section{$\th$ is the Optimal Bound}
\label{sec:thetaBest}
  In this section we establish Theorems \ref{thm:sharp} and  \ref{thm:dilate}. We also state and prove a version
  of  Theorem \ref{thm:BtN} under the assumption that $B$ is a tuple 
  of $d\times d$ matrices, rather than  the (weaker) assumption that it is
  a tuple of $n\times n$ matrices each with rank at most $d.$
   In  Section \ref{sec:averagemat2}
    we begin to connect, in the spirit of
  Remark \ref{rem:choosewisely}, the norm of $C_D$ to that of $M_D$.  
  The main results are in Section \ref{sec:optkstar}.

\subsection{Averages over $O(d)$ 
equal   averages over $S^{d-1}$ } 
 \label{sec:averagemat}

The next trivial lemma allows us to
 replace  certain averages over $O(d)$ 
with  averages over $S^{d-1}.$\index{averaging} 

\begin{lem}\label{lem:sph-ort}
 Suppose $\cB$ is a Banach space and let $S^{d-1}\subseteq\R^d$ denote  the unit sphere. 
If 
\beq
f: S^{d-1} \to \cB,
\eeq
 is an integrable function and $\gamma\in S^{d-1}$, then
\begin{equation}\label{eq:OrtSphere}
 \int_{S^{d-1}}f(\xi)\,d\xi=\int_{O(d)} f(U \gamma)\,dU.
\end{equation}
In particular,
\[
\int_{O(d)} f(U \gamma)\,dU
\]
does not depend on $\ga\in S^{d-1}$.
\end{lem}

We next apply Lemma \ref{lem:sph-ort} to represent the matrices $C_D$ defined in Equation \eqref{eq:defineCsub}
 as integrals over $S^{d-1}$. 
 Given $J\in\smatd$ and choosing an arbitrary unit vector  $\gamma$  in $\mathbb R^d$, define
 a matrix $E_J$ by \index{$E_J$}
\[
  E_J:= \int_{O(d)}\sgn\big[\gamma^* U^* J U\gamma\big]\,  
  U\gamma \gamma^* U^* \, dU.
\]
Note that $E_J$ depends only on $J$  but not on $\ga$. In fact, 
\begin{equation}
 \label{eq:EJ} 
   E_J = \int_{S^{d-1}}\sgn[\xi^*J\xi]\xi\xi^* \, d\xi
\end{equation}
by Lemma \ref{lem:sph-ort}.
 Given $J\in\smatd$, define 
 \[
  D_J: O(d) \to M_{d}
\]
 into the diagonal matrices given by
\begin{equation}
 \label{eq:defineDsubJ}
     D_J(U): = \sum_{j=1}^{d}\sgn[e_j^*U^* JU e_j] \, e_je_j^*.
\end{equation}
  In particular $D_J\in \mfD(d),$ where $\mfD(d)$ is defined in Equation \eqref{eq:defDd}.
 Recall, from \eqref{eq:defineCsub}, the definition of $C_D$. In particular, 
\[
  C_{D_J}= \int_{O(d)} U D_J(U) U^* dU.
\]

\begin{lemma}
 \label{lem:whywelikekappa}
   For $J\in\smatd$, independent of $j$, 
\[
  C_{D_J} = d \; E_J = 
\sum_{j=1}^d
  \int_{O(d)}  \sgn[e_j^* U^* JUe_j ] e_j e_j^* dU = d\int_{S^{d-1}} \sgn[\xi^* J\xi] \xi \xi^* \, d\xi.
\]
 Further, if $J\ne 0$, then $C_{D_J}\ne 0$. 
\end{lemma}

\begin{proof}
 The first statement follows from the definitions and Equation \ref{eq:EJ}. 
 For the second statement, observe that
\[
 \trace(E_J J) = \int_{S^{d-1}} |\xi^* J \xi| \, d\xi >0. \qedhere
\] 
\end{proof}

\begin{rem}
 \label{rem:nomoreOd}
 The rest of this paper 
 uses averaging over the sphere and not the orthogonal group.
 But it should be noted that various arguments in Section \ref{sec:average}
use  integration over the orthogonal group in an  essential way.
\end{rem}

\ssec{Properties of matrices gotten as averages}
\label{sec:averagemat2}
 This section describes properties of the matrices $E_J$ defined by Equation \eqref{eq:EJ}.\index{averaging}

\begin{lem}
 \label{lem:EV}
 If $U$ is an orthogonal matrix 
 and $J\in\smatd,$ then 
\[
  E_{U^*JU}=U^*E_JU.
\]
\end{lem}



\begin{proof}
 Using the invariance of Haar measure, 
\begin{align*}
E_{U^*JU}&=\int_{S^{d-1}}\sgn[\xi^*U^*JU\xi]\, \xi\xi^* \, d\xi\\
&=\int_{S^{d-1}}\sgn\big[(U^*\xi)^*U^*JU(U^*\xi)\big]\, (U^*\xi)(U^*\xi)^* \, d\xi\\
&=\int_{S^{d-1}}\sgn[\xi^*UU^*JUU^*\xi] \, U^*\xi\xi^*U \, d\xi,\\
&=U^*E_JU. \qedhere
\end{align*}
\end{proof}

Recall the definitions of $\alpha(s,t;a,b)$ and $\beta(s,t;a,b)$ from 
 Equations \eqref{def:al} and \eqref{def:be}. 

\begin{lemma}
 \label{lem:EJisJ}
   For $s,t\in\N$ and $a,b \ge 0,$ 
   \[
  E_{J(s,t;a,b)} = J\big(s,t;\alpha(s,t;a,b),\beta(s,t;a,b)\big) = \alpha(s,t;a,b) I_s \oplus -\beta(s,t;a,b) I_t. 
\]
 Moreover if $a,b$ are not both $0$, then $\alpha(s,t;a,b)$ and $\beta(s,t;a,b)$ are not both zero. 
\end{lemma}

\begin{proof}
  If $X,Y$ are $s\times s$ and $t\times t$ orthogonal matrices respectively, then $U=X\oplus Y\in O(d)$
 commutes with $J:=J(s,t;a,b)$ and by Lemma \ref{lem:EV},
\[
  U^* E_J U= E_J.
\]
 It follows that there exists $\alpha_0$ and $\beta_0$ such that  
$E_J = J(s,t;\alpha_0,\beta_0).$   That not both $\alpha_0$ and $\beta_0$ are zero
  follows from Lemma \ref{lem:whywelikekappa} which says $E_J\ne 0$.  Finally, in view of 
 Equation \eqref{eq:EJ}, 
\[
 \alpha_0 =  \langle E_J e_1,e_1 \rangle =  \int_{S^{d-1}} 
 \sgn [\xi^* J(s,t;,a,t)\xi]  \, \xi_1^2 \, d\xi = \alpha(s,t;a,b)
\]
 and similarly $\beta_0=\beta(s,t;a,b)$. 
\end{proof}

\begin{lemma}
 \label{lem:magicab}
  Let  $s,t\in\mathbb N_0$  be given and let $d=s+t$. Let $a(s,t),b(s,t)$ denote the
  pair from Proposition {\rm\ref{prop:starastarb}}. Thus, $a(s,t),b(s,t)$ is uniquely determined by
\begin{enumerate}[\rm (i)]
 \item $\aofst, \bofst \ge 0$;
 \item $s\, \aofst + t\, \bofst = d$; 
 \item \label{it:al=be} $\alpha(s,t;\aofst,\bofst) = \beta(s,t:\aofst,\bofst)$;
\end{enumerate}
 and produces the minimum,
\begin{enumerate}[\rm (i)]
 \item[\rm (iv)]  $\ka_*(s,t) = \ka(s,t;\aofst,\bofst)=d\alpha(s,t,\aofst,\bofst)$.
\end{enumerate}
  Abbreviate  $J_* = J\big(s,t;\aofst,\bofst\big)$;  then
\begin{equation}
\label{it:magicab2} 
   \ds E_{J_*} = \frac{\ka_*(s,t)}{d} J(s,t;1,1).
\end{equation}
\end{lemma}
\index{$J_*$}

\begin{proof}
 Since Proposition \ref{prop:starastarb} contains the first four items, it only remains to establish Equation \eqref{it:magicab2}.
 From Lemma \ref{lem:EJisJ} and Item \eqref{it:al=be},
\[
  d E_{J_*} = d \; \alpha(s,t;\aofst,\bofst) J(s,t;1,1) = \ka_*(s,t) J(s,t;1,1).\qedhere
\]
\end{proof}

\subsection{Dilating to commuting self-adjoint operators}
A $d\times d$ matrix $R$ is a \df{signature matrix} if 
  $R=R^\ast$ and $R^2=I.$  Thus, a symmetric
  $R\in\smatd$ is a signature matrix if its spectrum lies
  in the set $\{-1,1\}$. 
 Let $\mathscr E(d)$ denote $d\times d$ signature matrices and $\cube(d)$
 the set of $d\times d$ symmetric contractions. 
 Thus $\mathscr E(d)\subset \cube(d)$.

\begin{lemma}
 \label{lem:extpoints}
The set of extreme points of the set of $\cube(d)$ is  $\mathscr E(d)$.
 Moreover, each element of $\cube(d)$ is a (finite) convex combination
of elements of $\mathscr E(d)$.
\end{lemma}

By Caratheodory's Theorem, there is a bound on the number of summands needed
in representing an element of $\cube(d)$ as a convex combination of elements
of $\mathscr E(d)$.

\begin{proof}
Suppose for the moment that the set of extreme points of $\cube(d)$ is a 
 subset of $\mathscr E(d)$.  Since $\mathscr E(d)$ is closed and $\cube(d)$ is compact, it follows that
$\cube(d)=\overline{\mbox{co}(\mathscr E(d))} = \mbox{co}(\mathscr E(d))$,
 where $\mbox{co}$ denotes convex hull. Thus the  second part of the lemma
 will follow once it is shown that the extreme points of $\cube(d)$ lie in $\mathscr E(d)$.

 Suppose $X\in \cube(d)$ is not in $\mathscr E(d)$.  Without loss of generality,
 we may assume $X$ is diagonal with diagonal entries $\delta_{k}$. 
 In particular, there is an $\ell$ such that $|\delta_{\ell}| \ne 1$. Thus,
 there exists $y,z\in (0,1)$ such that $y+z=1$ and $y-z=\delta_{\ell}$.
 Let $Y$ denote the diagonal matrix with $k$-th diagonal entries $\delta_{k}$
 for $j\ne \ell$ and with $\ell$-th diagonal entry $1$. Define $Z$ similarly, but with $\ell$-th diagonal entry $-1$.
 It follows that $yY+zZ=X$. Thus $X$ is not an extreme point of $\cube(d)$  and therefore
 the set of extreme points of $\cube(d)$ is a subset of $\mathscr E(d)$.

 The proof that each $E\in\mathscr E(d)$ is an extreme point is left to the interested reader.
\end{proof}

Our long march of lemmas now culminates in the following lemma.
Recall the definitions of $\ka_*(s,t)$ and $\ka_*(d)$ from 
the outset of Section \ref{sec:expectaverage}.  A symmetric matrix $R$
 is a  \df{symmetry matrix} if $R^2$ is projection. Equivalently,
 $R$ is a symmetry matrix if $R=R^*$ and the spectrum of $R$ lies in the set $\{-1,0,1\}$.
For a symmetric matrix $D$, the triple $\sign(D)=(p,z,n)$,
  called the \df{signature} of $D$, 
  denotes the number of positive, zero and negative eigenvalues of $D$ respectively.
Note that a symmetry matrix $R$ is determined, up to unitary equivalence, by its signature.

\begin{lemma}
 \label{lem:prerhodoesit}
 If $R$ is a signature matrix with $s$ positive eigenvalues and $t$ negative eigenvalues, then there exists
 $D\in\mfD(d)$ such that 
\[
  \ka_*(s,t)R = V^* \TM_D V.
\]
 In particular, replacing $D$ with $D^\prime = \frac{\ka_*(d)}{\ka_*(s,t)} D$ and noting that $\ka_*(d)\le \ka_*(s,t)$,
 we have  $D^\prime \in\mfD(d)$ and 
\[
 \ka_*(d)R = V^* \TM_{D^\prime} V.
\]
 Here $V$ is the isometry from Lemma {\rm\ref{lem:Viso}}. 
 We emphasize that the $V$ does not depend on $R$
 or even on $s,t$.
\end{lemma}

\begin{proof}
 There is an $s,t$ and unitary $W$ such that $R=W^* J(s,t;1,1)W$.
  From Lemma \ref{lem:magicab}, there exists
  a $J_*\in\smatd$ such that $E_{J_*} = \frac{ \ka_*(s,t)} d  J(s,t;1,1)$.
 Using Lemmas \ref{lem:averagelift},  \ref{lem:whywelikekappa} and \ref{lem:EV},
\[
 \begin{split}
  V^* M_{D_{W^* J_* W}} V & = 
  C_{D_{W^* J_* W }}\\ 
& = d E_{W^* J_* W } \\
& = W^* d E_{J_*} W \\ 
   & =  W^* d \frac{\ka_*(s,t)}{d} J(s,t;1,1) W \\ 
  & = \ka_*(s,t) W^* J(s,t;1,1) W \\
  & = \ka_*(s,t) R. \qedhere
\end{split} 
\]
\end{proof}


\begin{thm}
\label{thm:rhodoesit}
  Given $d$, there exists family $\fC_d$ of commuting self-adjoint  contractions on a $($common$)$ Hilbert\index{dilation}
  space $\mathcal H$ and an  isometry $V\colon\mathbb R^d\to \mathcal H$ such that 
  for each contraction $C\in \mathbb S_d$, there is a $T_C\in\fC_d$ with
\[
  \ka_*(d)C = V^* T_C V.
\]
\end{thm}

\begin{proof}
Set $\mathcal H:=L^2(O(d),\R^d)$ and let $V\colon\R^d\to \mathcal H$ denote the isometry from Lemma \ref{lem:Viso}.
 Let $\fC_d$ denote the collection of operators $M_D$ for $D\in\mfD(d)$.  By Lemma \ref{lem:lift2},
  $\fC_d$ is a collection of commuting operators. By Lemma \ref{lem:lift1}, each $M_D$ is a self-adjoint  contraction.
  Finally, observe that $\fC_d$ is convex. 


 By Lemma \ref{lem:extpoints} there exists an $h$ and signature matrices
 $R_1,\dots,R_h$  such that  $C=\sum_{j=1}^n  c_j R_j,$ where $c_j\ge 0$ and $\sum c_j = 1$.
 By Lemma \ref{lem:prerhodoesit},  there exists 
 $S_1,\dots,S_h\in\fC_d$ 
 such that  \ 
  $\ka_*(d) R_k=V^*S_kV$ for $k\in\{1,\dots,h\}$. 
   Hence, $\ka_*(d) C= V^* S V,$ where $S=\sum c_j S_j \in \fC_d$.
\end{proof}

\subsection{Optimality of $\ka_*(d)$}
 \label{sec:optkstar}
  The following theorem contains the optimality statement of Theorem \ref{thm:dilate}
  and Theorem \ref{thm:sharp}. It also contains 
  a preliminary version of Theorem \ref{thm:BtN}. Recall
  $\cubeg$ is the sequence $(\cubeg(d))_d$ and $\cubeg(d)$ is the set of $g$-tuples
  of symmetric $d\times d$ contractions.

\begin{thm}
 \label{thm:nextbest}
   For each $g$ and $d$, if $B$ is any $g$-tuple of symmetric $d\times d$ matrices, then\index{cube}\index{free cube}   
  $[-1,1]^g\subset \fs_{L_B}$ implies $\ka_*(d)\cubeg \subset \cD_{L_B}$. 
  
   Conversely, 
   if $\ka>\ka_*(d)$, then there exists a $g$ and a $g$-tuple of symmetric matrices $B$ such
   that $[-1,1]^g\subset \fs_{L_B}$, but $\ka\cubeg \not\subset \cD_{L_B}$.

   In particular,
   $\vartheta(d) = (\ka_*(d))^{-1}$ is the optimal constant in Theorem {\rm\ref{thm:dilate}}. 
\end{thm}

\begin{proof}
  Starting with the proof of the second statement,  fix $d$ and  suppose $\ka >\ka_*(d)$.
  Let $(\sopt,\topt)$ be a pair for which $\ka_*(d)= \ka_*(\sopt,\topt)$.
  Let $(\aopt,\bopt)$ be a pair of positive numbers such that $\sopt \aopt + \topt\bopt=d$
  and $\ka_*(d) = \ka(\sopt,\topt;\aopt,\bopt)$ coming from Lemma \ref{lem:magicab}.
  Let  $\Jopt = J(\sopt,\topt;\aopt,\bopt)$ and define
  the distinguished (infinite variable pencil)  $\Lopt: \ L^\infty(O(d))\to \mathbb S_{d}$
  by 
\begin{equation}
 \label{eq:Lstab}
  \Lopt(x)
   =   \frac{1}{\ka_*(d)} \int_{O(d)} U^* \;\Jopt \; U \ x(U) \, dU,
\end{equation}
  for  $x\in L^\infty(O(d)).$  By analogy with the sets $\cD_{L_B}$, let
 $\cD_{\Lopt}(n)$ denote those measurable $X:O(d)\to \mathbb S_{n}$ such that
\[
  \Lopt(X) = \frac{1}{\ka_*(d)} \int_{O(d)} U^* \;\Jopt \; U \otimes  X(U) \, dU,
\]
 satisfies $I-\Lopt(X)\succeq 0$. 

 Let $\cube^\infty$ denote the sequence of sets $(\cube^\infty(n))$, where 
 elements of $\cube^\infty(n)$ are measurable functions
\[
  X:O(d) \to \mathbb S_{n},
\]
 such that  $X(U)$ is a symmetric contraction for each $U\in O(d)$. 
  In particular, $x\in\cube^\infty(1)$ is an element of $L^\infty(O(d))$
 of norm at most one and $\cube^\infty$ can be thought of
  as an infinite dimensional matrix cube. 
 
 Given $x\in \cube^\infty(1)$  and a
  unit vector $e$,  note that
\[
 e^* \Lopt( x )e  \le  
 \frac{1}{\ka_*(d)} \int_{O(d)} | e^* U^*\Jopt Ue |\, dU 
 =    \frac{1}{\ka_*(d)}    \int_{S^{d-1}} | \xi^* \Jopt \xi |\, d\xi
         = 1. 
\]
Thus $I - \Lopt(x) \succeq 0$.
 Hence  $\cube^\infty(1)\subset \cD_{\Lopt}(1)=\fs_{\Lopt}$.

   Now consider the mapping $X:O(d)\to O(d)$ defined by
\[
  X(U) = U^* J ({{\sopt}},{{\topt}} ;1,1)U. 
\]
 In particular, $X$ pointwise has norm one and thus $X\in \cube^\infty(d).$
 We next show that $X \not \in  \frac{1}{\ka} \cD_{\Lopt}(d).$
  
 For  $U\in O(d)$, 
 let $Z(U)= U^* \Jopt  U$.
 With  $\ds\bE=\frac{1}{\sqrt{d}} \sum_{j=1}^{d} e_j \otimes e_j,$ 
\[
\bE^* (Z(U)\otimes X(U)) \bE = \frac{1}{d} \sum_{s,t=1}^{d} e_s^*Z(U)e_te_s^*X(U)e_t=\frac1d \langle Z(U),X(U) \rangle_{\trace},
\] 
 where $\langle \cdot,\cdot\rangle_{\trace}$ is the trace inner product,
\[
 \langle A,B\rangle_{\trace} = \trace(AB^\ast) = \sum_{j,k} e_j^* A e_k e_k^* Be_j.
\]
 Now,
\[
 \begin{split}
 \trace(Z(U)X(U)) & = \trace( U^* \Jopt J({\sopt},{\topt};1,1) U)\\
   & = \trace(\Jopt  J({\sopt},{\topt};1,1)) = \sopt a({\sopt},{\topt})+\topt b({\sopt},{\topt}) \\
   & = d.
 \end{split}
\]
  Hence 
\[
\begin{split}
   \bE^* \Lopt(X) \bE & = \frac{1}{\ka_*(d)} \int \bE^* (Z(U)\otimes X(U))\bE \, dU \\
     & =  \frac{1}{\ka_*(d)}\frac{1}{d} d. 
 \end{split}
\]
 Thus 
$
  \|\Lopt(X)\| \ge \frac{1}{\ka_*(d)} > \frac{1}{\ka}
$, so
$$
\frac{1}{\ka}  I  - \Lopt(X)  \not \succeq  0
$$
as predicted.

   We next realize $\Lopt$ as a  limit of pencils $L_B$ with $B\in\mathbb S_d^g$. Suppose $(\mathcal P_k)$ 
   is a sequence of (measurable) partitions of $O(d)$ and write
  $\mathcal P_k=\{P_{k,1},\dots,P_{k,g_k}\}$. 
Consider the corresponding $g_k$-tuples $A^k=(A^k_1,\dots,A^k_{g_k})\in\mathbb S_d^{g_k}$, where 
\[
  A^k_j  = \frac{1}{\ka_*(d)} 
     \int_{P_{k,j}}  U^* \; \Jopt \; U  \, dU  = \int_{P_{k,j}} Z(U)\, dU,
\]
 and the associated homogeneous linear pencil,
\[
 L_k(x) = \sum_{j=1}^{g_k} A^k_j  x_j.
\]
  Given  $y\in \cube^{(g_k)}(1)=[-1,1]^{g_k}$, let $x= \sum_{j=1}^{g_k}  y_j \chi_{P_{k,j}},$ where $\chi_P$ denotes
  the characteristic function of the set $P$. Since $x\in \cube^{\infty}(1)$ and 
\[
  L_k(y) = \Lopt(x), 
\]
 it follows that $y\in \fs_{L_{A^k}}$. Thus, $[-1,1]^{g_k}\subset \fs_{L_{A^k}}$.

 Suppose that $U_{k,j}$ are given points in $P_{k,j}$. 
  Given $k$, let  $X^k=(X^k_1,\dots,X^k_{g_k})$ where
\[
  X^k_j = X(U_{k,j}) = U^*_{k,j} J(\sopt,\topt;1,1) U_{k,j}.
\]
  In particular, $\|X^k_j\|\le 1$.  Evaluate,
\[
 L_k(X^k) = \frac{1}{\ka_*(d)} \sum_{j=1}^{g_k} A_{k,j} \otimes X_{k,j}.
\]
Hence
\[
  \Lopt(X) - L_k(X^k) = \sum_{j=1}^{g_k} \int_{P_{k,j}}  Z(U)\otimes \big( X(U)-X(U_{k,j}) \big) \, dU.
\]
  The uniform continuity of $X(U)$ implies there exists a choice of $P_{k,j}$ and $U_{k,j}$  such
 that $L_k(X^k)$ converges to $\Lopt (X)$. Hence, $\|L_k(X^k)\|>\ka$ for sufficiently large $k$.
 Consequently $X\in \cube^{g_k}(d)$, but $\ka X\notin \cD_{L_{A^k}}(d)$ and the proof of 
 the second statement is complete.

 Turning to the first part of the theorem,  suppose that $B$ 
  is a $g$-tuple of $d\times d$ symmetric matrices and
  $[-1,1]^g\subset \fs_{L_B}$.  Given a $g$-tuple 
  $X\in\cubeg(d)$, Theorem \ref{thm:rhodoesit} produces
 a Hilbert space $\mathcal H$, a $g$-tuple of commuting self-adjoint contractions on $\mathcal H$,
  an isometry $V:\mathbb R^d\to \mathcal H$ such that
\[
  \ka_*(d)  X_j  = V^* T_j V,
\]
 a relationship  summarized by $\ka_*(d)X=V^* TV.$ 
  By Proposition \ref{prop:commute-include},
  $\ka_*(d) \cubeg \subset \cD_{L_B}$ and the proof of the first statement of the theorem is complete.

  For the last statement in the theorem, suppose $\ka$ has the property that for every $g$ 
  each $g$-tuple of commuting symmetric matrices of size $d$ dilates
  to a tuple of commuting symmetric contractions on Hilbert space. 
  Proposition \ref{prop:commute-include} implies $\ka \cubeg \subset \cD_{L_B}$
  for any $g$-tuple $B$ of symmetric matrices of size $d$ such that
  $[-1,1]^g\subset \fs_{L_B}$.  Hence, by what has already been proved,
  $\ka\le \ka_*(d)$. 
\end{proof}

\section{The Optimality Condition $\al=\be$ in Terms of Beta Functions}
 \label{sec:prequest}
  In this section $\alpha(s,t;a,b)$ and $\beta(s,t;a,b)$ 
  which were defined 
  in Equations \eqref{def:al} and \eqref{def:be} (see also Lemma \ref{lem:EJisJ})
  are computed in terms of the regularized  incomplete beta function. 
  See Lemma \ref{lem:compalpha}.
  A consequence is the relation of Equation \eqref{eq:theta-beta}.
  Lemma \ref{oplus0} figures in the proof of Theorem \ref{thm:BtN}
  in Section \ref{sec:rankvsize}. 

Let $\Ga$ denote the Euler gamma function \cite{Rai71}.\index{Euler function}\index{gamma function}\index{Euler gamma function}

\begin{lemma}\label{igamma}
Suppose $m\in\R_{\ge0}$. Then
\[\int_0^\infty r^me^{-r^2}dr=\frac12\Ga\left(\frac{m+1}2\right).\]
\end{lemma}

\begin{proof}
Setting $s:=r^2$, we have $r^m=s^{\frac m2}$ and $\frac{ds}{dr}=\frac{dr^2}{dr}=2r$, i.e., $dr=\frac{ds}{2r}=\frac{ds}{2\sqrt s}$. Then
\[\int_0^\infty r^me^{-r^2}dr=\int_0^\infty\frac{s^{\frac m2}}{2\sqrt s}ds=\frac12\int_0^\infty s^{\frac{m-1}2}ds=\frac12\Ga\left(\frac{m+1}2\right).
\qedhere
\]
\end{proof}

\begin{lemma}\label{gauss1}
\[\int_{\R^n}e^{-\|x\|^2}dx=\pi^{\frac n2}.\]
\end{lemma}

\begin{proof}
\[\int_{\R^n}e^{-\|x\|^2}dx=\int_\R\dots\int_\R e^{-x_1^2}\dots e^{-x_n^2}dx_n\dots dx_1=\left(\int_\R e^{-x^2}dx\right)^n\overset{\ref{igamma}}=\Ga\left(\frac12\right)^n=\pi^{\frac n2}.\qedhere\]
\end{proof}

We equip the unit sphere in $S^{n-1}\subseteq\R^n$ with the unique rotation invariant probability measure.

\begin{rem}\label{surfarea}
Recall that the surface area of the $n-1$-dimensional unit sphere $S^{n-1}\subseteq\R^n$ is
\[
\area(S^{n-1})=
\frac{n\pi^{\frac n2}}{\Ga(1+\frac n2)}=\frac{2\pi^{\frac n2}}{\Ga(\frac n2)}.\]
\end{rem}

Now we come to a key step, converting integrals over the sphere $S^{d-1}$ to integrals over $\mathbb R^d$.

\begin{lemma}\label{sphere2gauss}
Suppose $A\in\R^{d\times d}$ and $f\colon\R^d\to\R$ is quadratically homogeneous, i.e.,
$f(\la x)=\la^2f(x)$ for all $x\in\R^d$ and $\la\in\R$. Suppose furthermore that $f|_{S^{d-1}}$ is integrable on $S^{d-1}$. Then
\begin{align*}
\int_{S^{d-1}}f(\xi)d\xi&=\frac2{d\pi^{\frac d2}}\int_{\R^d}f(x)e^{-\|x\|^2}dx\qquad\text{and}\\[.1cm]
d\int_{S^{d-1}}f(\xi)d\xi&=\frac1{(2\pi)^{\frac d2}}\int_{\R^d}f(x)e^{-\frac{\|x\|^2}2}dx.
\end{align*}
\end{lemma}

\begin{proof}
The first equality follows from
\begin{align*}
\int_{\R^d}f(x)e^{-\|x\|^2}dx&=\int_{S^{d-1}}\int_0^\infty\area(rS^{d-1})f(r\xi)e^{-\|r\xi\|^2}drd\xi\\
&=\int_{S^{d-1}}\int_0^\infty r^{d-1}\area(S^{d-1})r^2f(\xi)e^{-r^2}drd\xi\\
&=\area(S^{d-1})\left(\int_0^\infty r^{d+1}e^{-r^2}dr\right)\int_{S^{d-1}}f(\xi)d\xi\\
&=\frac{d\pi^{\frac d2}}{\Ga(1+\frac d2)}\frac12\Ga\left(1+\frac d2\right)\int_{S^{d-1}}f(\xi)d\xi.
\end{align*}
where the last equation uses Remark \ref{surfarea} and Lemma \ref{igamma}. The proof of the second equality is similar and is left as an exercise for the reader.
\end{proof}

\begin{lemma}\label{oplus0}
Suppose $J\in\R^{d\times d}$ is any matrix and consider the zero matrix $0_u:=0\in\R^{u\times u}$. Suppose also $i,j\in\{1,\dots,d\}$. 
Then there is some $\ga\in\R$ such that
\[\int_{S^{d+u-1}}\sgn\big[\xi^*(J\oplus0_u)\xi\big] \; \xi_i\xi_j \; d\xi=
\begin{cases}
\displaystyle\frac d{d+u}\int_{S^{d-1}}\sgn[\xi^*J\xi] \; \xi_i\xi_j \; d\xi
&\text{if $i,j\in\{1,\dots d\}$}\\
\ga&\text{if $i=j\in\{d+1,\dots d+u\}$}\\
0&\text{otherwise}
\end{cases}
\]
\end{lemma}

\begin{proof}
Set
\[C:=\frac12(d+u)\pi^{\frac{d+u}2}\qquad\text{and}\qquad c:=\frac12d\pi^{\frac d2}.\]
By Lemma \ref{sphere2gauss}, the left hand side equals
equals
\[\frac1C\int_{\R^{d+u}}\sgn[ x^*(J\oplus0_s)x ] x_ix_je^{-\|x\|^2}dx.\]
If $i,j\in\{1,\dots,d\}$, then this in turn equals
\begin{align*}
\frac1C\int_{\R^d}\int_{\R^u}\sgn[y^*Jy]\; 
y_iy_j \; e^{-\|y\|^2-\|z\|^2} \; dzdy
&=\frac1C\left(\int_{\R^u}e^{-\|z\|^2}dz\right) \int_{\R^d}
\sgn[y^*Jy] \; y_iy_je^{-\|y\|^2} \; dy\\
&=\frac cC\pi^{\frac u2}\int_{S^{d-1}}\sgn[\xi^*J\xi] \xi_i\xi_jd\xi
\end{align*}
where the last equality follows from Lemmas \ref{gauss1} and \ref{sphere2gauss}.
If $i,j\in\{d+1,\dots,d+u\}$, then it equals
\[\frac1C\int_{\R^d}\int_{\R^u}\sgn[y^*Jy] \; z_{i-d}z_{j-d} \; 
e^{-\|y\|^2-\|z\|^2}dzdy\]
which equals up to a constant depending only on $J$ the integral
\[
\int_{\R^u}z_{i-d}z_{j-d}e^{-\|z\|^2}dz
\]
which is zero for symmetry reasons if $i\ne j$ and which 
depends only on $u$ if $i=j$. The remaining case
where one of $i$ and $j$ is in $\{1,\dots,d\}$ and the other one in $\{d+1,\dots,d+u\}$ follows similarly.
\end{proof}

\begin{lemma}
\label{lem:compalpha}
Let $s,t\in\N$, $d:=s+t$ and $a,b\in\R_{\ge0}$ with $a+b>0$. Then
\begin{equation}\label{eq:alphaUltra}
\al(s,t;a,b):=\int_{S^{d-1}}\sgn[\xi^*J(s,t;a,b)\xi] \xi_i^2d\xi
=\frac1d\left(2I_{\frac a{a+b}}\left(\frac t2,\frac s2+1\right)-1\right)
\end{equation}
for all $i\in\{1,\dots,s\}$. Analogously,
\begin{equation}\label{eq:betaUltra}
\be(s,t;a,b):=-\int_{S^{d-1}}\sgn[ \xi^*J(s,t;a,b)\xi]  \xi_i^2d\xi
=\frac1d\left(2I_{\frac b{a+b}}\left(\frac s2,\frac t2+1\right)-1\right)
\end{equation}
for all $i\in\{s+1,\dots,s+t\}$. 
An additional obvious property is
\begin{equation}\label{eq:abmix}
\al(s,t;a,b)=\be(t,s;b,a).\end{equation}
\end{lemma}
\index{$\al(s,t;a,b):
=\frac1d\left(2I_{\frac a{a+b}}\left(\frac t2,\frac s2+1\right)-1\right)$}

\index{$\be(s,t;a,b):=\frac1d\left(2I_{\frac b{a+b}}\left(\frac s2,\frac t2+1\right)-1\right)$}

\begin{proof}
We have
\begin{align*}
\int_{S^{d-1}}&\sgn[\xi^*J(s,t;a,b)\xi]  \xi_i^2d\xi\\
&\overset{\ref{sphere2gauss}}=\frac2{d\pi^{\frac d2}}\int_{\R^d}\sgn[x^*J(s,t;a,b)x]  x_i^2e^{-\|x\|^2}dx\\
&=\frac2{sd\pi^{\frac d2}}\int_{\R^d}\sgn[x^*J(s,t;a,b)x] \;
(x_1^2+\dots+x_s^2)e^{-\|x\|^2}dx\\
&=\frac2{sd\pi^{\frac d2}}\int_{\R^s}\int_{\R^t}\sgn[a\|y\|^2-b\|z\|^2] \
\|y\|^2e^{-\|y\|^2-\|z\|^2}~dz~dy\\
&=\frac2{sd\pi^{\frac d2}}\int_0^\infty\area(\si S^{s-1})\int_0^\infty\area(\ta S^{t-1})\sgn[a\si^2-b\ta^2] \; \si^2e^{-\si^2-\ta^2}\;d\ta\;d\si\\
&=\frac2{sd\pi^{\frac d2}}\int_0^\infty\si^{s-1}\frac{2\pi^{\frac s2}}{\Gamma(\frac s2)}\int_0^\infty\ta^{t-1}\frac{2\pi^{\frac t2}}{\Gamma(\frac t2)}
\sgn[a\si^2-b\ta^2] \; \si^2e^{-\si^2-\ta^2}\;d\ta\;d\si \\
&=\frac8{sd\Gamma(\frac s2)\Gamma(\frac t2)}\int_0^\infty\int_0^\infty\si^{s+1}\ta^{t-1}\ \sgn[a\si^2-b\ta^2] \; e^{-\si^2-\ta^2}\;d\ta\;d\si\\
&=\frac8{sd\Gamma(\frac s2)\Gamma(\frac t2)}\int_0^\infty r\int_0^{\frac\pi2}(r\cos\ph)^{s+1}(r\sin\ph)^{t-1} \; 
\sgn[ a(\cos\ph)^2-b(\sin\ph)^2] \; e^{-r^2}\;d\ph\;dr\\
&=\frac8{sd\Gamma(\frac s2)\Gamma(\frac t2)}\left(\int_0^\infty r^{d+1}e^{-r^2}dr\right)\int_0^{\frac\pi2}(\cos\ph)^{s+1}(\sin\ph)^{t-1}\
\sgn[a(\cos\ph)^2-b(\sin\ph)^2] d\ph\,dr\\
&\overset{\ref{igamma}}=\frac{4\Ga\left(\frac d2+1\right)}{sd\Gamma(\frac s2)\Gamma(\frac t2)}\int_0^{\frac\pi2}(\cos\ph)^{s+1}(\sin\ph)^{t-1}
\; \sgn[a(\cos\ph)^2-b(\sin\ph)^2] \;d\ph\\
&=\frac{\Ga\left(\frac d2\right)}{s\Gamma(\frac s2)\Gamma(\frac t2)}\int_0^1(1-x)^{\frac{s+1-1}2}x^{\frac{t-1-1}2}\;
\sgn[a(1-x)-bx]\;dx\\
&=\frac1{sB\left(\frac s2,\frac t2\right)}\int_0^1(1-x)^{\frac s2}x^{\frac t2-1} \;
\sgn[a-(a+b)x]\;dx
\end{align*}
using a change of variable $x=(\sin\ph)^2$ which makes
\[\frac{dx}{d\ph}=2(\sin\ph)(\cos\ph)=2\sqrt x\sqrt{1-x}.\]
Now suppose that $a,b\in\R_{\ge0}$ with $a+b>0$. Then the integral in the last expression equals
\[
\begin{split}
\int_{0}^{\frac a{a+b}}(1-x)^{\frac s2}x^{\frac t2-1}dx-& \int_{\frac a{a+b}}^1(1-x)^{\frac s2}x^{\frac t2-1}dx\\
=&B_{\frac a{a+b}}(\frac t2,\frac s2+1)-\int_0^{\frac b{a+b}}x^{\frac s2}(1-x)^{\frac t2-1}dx\\
=&B_{\frac a{a+b}}\left(\frac t2,\frac s2+1\right)-B_{\frac b{a+b}}\left(\frac s2+1,\frac t2\right).
\end{split}
\]
Using
\[B\left(\frac s2,\frac t2\right)=\frac{\Ga\left(\frac s2\right)\Ga\left(\frac t2\right)}{\Ga\left(\frac d2\right)}=\frac ds\frac{\frac s2\Ga\left(\frac s2\right)\Ga\left(\frac t2\right)}{\frac d2\Ga\left(\frac d2\right)}
=\frac ds\frac{\Ga\left(\frac s2+1\right)\Ga\left(\frac t2\right)}{\frac d2\Ga\left(\frac d2+1\right)}=\frac dsB\left(\frac s2+1,\frac t2\right),\]
we see that
\[\al(s,t;a,b)=\frac1d\left(I_{\frac a{a+b}}\left(\frac t2,\frac s2+1\right)-I_{\frac b{a+b}}\left(\frac s2+1,\frac t2\right)\right)\]
where $I$ denotes the regularized (incomplete) beta function.
Finally, \eqref{eq:alphaUltra} follows using
\[
I_{1-p}(\zeta,\eta)=1-I_p(\eta,\zeta).
\]
The proof of \eqref{eq:betaUltra} is similar.
\end{proof}

\section{Rank versus Size for the Matrix Cube}
\label{sec:rankvsize}
  In this section we show how to pass from size $d$ to rank $d$
  in the first part of Theorem \ref{thm:nextbest}, thus completing
  our dilation theoretic proof of Theorem \ref{thm:BtN}.
  Accordingly fix, for the remainder of this section, 
  positive integers $d\le m$.

Given positive integers $s,t,u$ and numbers
  $a,b,c$, let
\[
 J(s,t,u;a,b,c) = aI_s\oplus -b I_t \oplus cI_u.
\]

\begin{lemma}
 \label{lem:precutdown0}
   Given positive integers $s,t$ with $s+t=d$ and nonnegative numbers
  $a,b,c$, there exists real numbers  $\alpha,\beta,$ and $\gamma$  such that
\[
  J(s,t,m-d;\alpha,\beta,\gamma) = {m\int_{S^{m-1}}
  \sgn[\xi^* J(s,t,m-d;a,b,c)\xi] \, \xi \xi^* \, d\xi}.
\]
\end{lemma}

\begin{proof}
  Given $U_v\in\mathcal O(v)$, for $v=s,t,m-d$, let $U$ denote the block diagonal
  matrix with entries $U_s,U_t,U_{m-d}$. Thus, $U\in\mathcal O(m)$ and $U$
  commutes with $J(s,t,m-d;a,b,c).$  The conclusion now follows, just as in Lemma \ref{lem:EJisJ}.
\end{proof}
 
\begin{lemma}
 \label{lem:precutdown1}
   For each $s,t$ with $s+t=d$,   there exists a $\gamma=\gamma(s,t)$ such that
\beq
\label{eq:sign0}
  \ka_*(s,t) J(s,t,u;1,1,\gamma(s,t)) = 
    {m\int_{S^{m-1}}\sgn[\xi^* J(s,t,m-d;\aofst,\bofst,0)\xi] \, \xi \xi^* \, d\xi}.
\eeq
 Here $\ka_*(s,t), \aofst$ and $\bofst$ are the optimal choices from
 Proposition \ref{prop:starastarb}.
\end{lemma}

\begin{proof}
  Denote the right hand side 
  of \eqref{eq:sign0} by $E$. Then by Lemma \ref{oplus0},
\[
e_iE e_j=
\begin{cases}
\displaystyle d\int_{S^{d-1}}\sgn[\xi^*J(s,t;\aofst,\bofst)\xi]
 \ e_i\xi\xi^*e_j \ d\xi&   \text{if  $i=j\in\{1,\dots,d\}$}\\
\ga&\text{if $i=j\in\{d+1,\dots,m\}$}\\
0&\text{otherwise}
\end{cases}
\]
for some $\ga\in\R$ and all $i,j\in\{1,\dots,m\}$.
 On the other hand, from Lemma \ref{lem:magicab},
\[
 \frac{\ka_*(s,t)}{d} J(s,t;1,1) 
 = \int_{S^{d-1}}\sgn[\xi^*J(s,t;\aofst,\bofst)\xi] \, \xi\xi^* \, d\xi.
\]
 Hence, with $P$ denoting the projection of $\mathbb R^d\oplus \mathbb R^{m-d}$ onto the first $d$ coordinates,
\[
 PEP = \ka_*(s,t)J(s,t;1,1)
\]
 and the conclusion of the lemma follows.
\end{proof}

 Let $\mathcal H$ denote the Hilbert space $\mathbb R^m\otimes L^2(O(m))$
  and let $V:\mathbb R^m\to \mathcal H$ denote the isometry,
\[
 Vx(U) = x.
\]
 Thus $\mathcal H$ and $V$ are the Hilbert space and isometry (with $m$ in place of $d$) from
  Equations \eqref{eq:H} and \eqref{eq:V}. Recall too the collection $\mfD(m)$
  of contractive measurable mappings $D:O(m)\to M_m$ taking diagonal values, and, for $D\in\mfD(m)$, 
  the contraction  operator $M_D:\mathcal H\to \mathcal H$. 

\begin{lemma}
 \label{lem:precutdown2}
  For each $m\times m$ symmetry matrix $R$ of rank $d$ there exists 
 a $D\in\mfD(m)$ such that
\[
  \ka_*(d) \; PRP = P V^* M_D V P,
\]
 where $P$ is the projection onto the range of $R$. 
\end{lemma}

\begin{proof}
  The proof is similar to the proof of Lemma \ref{lem:prerhodoesit}. 
  Let $s$ and $t$ denote the number of positive and negative eigenvalues of $R$.
  Hence, $R=W^* J(s,t,m-d;1,1,0)W$ for some $m\times m$ unitary $W$. 
  Let $J_*= J(s,t,m-d;\aofst,\bofst,0)$ and define  $D\in O(m)$ by
\[
 D(U) = \sum_{j=1}^m  \sgn[e_j^* U^* W^* JW U e_j] \,  e_j e_j^* \, dU.
\]
  Now, by Lemma \ref{lem:averagelift},  Remark \ref{rem:nomoreOd} and Lemma \ref{lem:precutdown1}, 
\[
 \begin{split} 
  V^* M_{D_{W^* J_* W}} V & = 
  C_{D_{W^* J_* W }}\\ 
 & = \int_{O(m)} U D(U) U^* \, dU \\
 & = \sum_{j=1}^m \int_{O(m)} \sgn[e_j^* U^* W^* JW U e_j] \,  U e_j e_j^* U^* \, dU \\
 & =  \sum_{j=1}^m \int_{O(m)} \sgn[e_j^* U^* J U e_j] \,  W^* U e_j e_j^* U^* W \, dU \\
 & = W^*  \big( \sum_{j=1}^m \int_{O(m)} 
 \sgn[e_j^* U^* J U e_j] \,  U e_j e_j^* U^*  \, dU \big) W \\
 & = m W^* \big (\int_{S^{m-1}} \sgn [\xi^* J \xi] \, \xi \xi^* \, d\xi \big) W \\
 & = \ka_*(s,t)\,  W^* J(s,t,m-d;1,1,\gamma(s,t)) W. 
\end{split} 
\]
 The observation 
\[
  P W^* J(s,t,m-d;1,1,\gamma(s,t))W P = P W^* J(s,t,m-d;1,1,0)W P
\]
 completes the proof. 
\end{proof}

 Given a $g$-tuple $\cM=(\cM_1,\dots,\cM_g)$ of $d$-dimensional subspaces of
  $\mathbb R^m$, let $\cube(\cM)$ denote  the collection of $g$-tuples
 of $m\times m$ symmetric contractions $C=(C_1,\dots,C_g)$
  where each $C_j$ has range in $\cM_j$.  

\begin{lemma}
 \label{lem:precutdown3}
   The set $\cube(\cM)$ is closed and convex and its extreme points
   are the tuples of the form  $E=(E_1,\dots,E_g)$, where
   each $E_j$ is a symmetry  matrix with rank $d$.
\end{lemma}

\begin{proof}
  Given a subspace $\mathcal N$ of $\mathbb R^m$ of dimension $d$, note that the set $n\times n$ symmetric contractions
  with range in $\mathcal N$ is a convex set whose extreme points are symmetry matrices whose range is exactly $\mathcal N$
  (cf.~Lemma \ref{lem:extpoints}). 
  Since $\cM = \times_{j=1}^g \cM_j$ the result follows. 
\end{proof}

\begin{lemma}
 \label{lem:cutdown}
   Suppose $X=(X_1,\dots,X_g)$ is a tuple of $m\times m$ symmetric contractions. If 
   $P_1,\dots,P_g$ is a tuple of rank $d$ projections, then there exists
   a tuple of $m\times m$ {symmetric} contractions $Y=(Y_1,\dots,Y_g)$ such that
\begin{enumerate}[\rm(i)]
 \item  $P_jX_jP_j = P_j Y_j P_j$; and
 \item  there exists  a tuple of commuting self-adjoint  contractions $Z=(Z_1,\dots,Z_g)$ on a Hilbert space  $\mathcal H$
      such that $\ka_*(d)Y$ lifts to $Z$.
\end{enumerate}
  Thus, there exists an isometry $Q:\mathbb R^m\to\mathcal H$
  such that 
  \[Y_j= \frac{1}{\ka_*(d)} W^* Z_j W\quad\text{ and }\quad P_j W^* Z_j W P_j = P_j X_j P_j.\]
\end{lemma}

\begin{proof} 
  Let $\cM_j$ denote the range of $P_j$. 
  Let $C_j=P_jX_j P_j$.  By Lemma \ref{lem:precutdown3}, there
 exists a positive integer $N$ and extreme points $E^1,\dots,E^N$
 in $\cube(\cM)$ and positive numbers $\epsilon_1,\dots,\epsilon_N$ such 
  that
\[
  C_j = \sum_{k=1}^N  \epsilon_k E^k_j.
\]
  For each $k,j$ there exist positive integers $s^k_j,t^k_j$ such that 
  $s^k_j + t_j^k =1$ and a unitary matrix $W_j^k$ such that $(W_j^k)^* \mathbb R^d\oplus \{0\}=\cM_j$ and
\[
  E^k_j = (W_j^k) ^* J(s_j^k,t^k_j,m-d,1,1,0)W_j^k.
\]
 In particular,
\[
 E^k_j = P_j E^k_j P_j. 
\]
 By Lemma \ref{lem:precutdown2}, there exists $D_j^k\in\mfD (m)$ such that
\[
  \ka_*(d) E^k_j =  P_j V^*  M_{D^k_j} V P_j.
\]
 Let 
\[
  Z_j  = \sum_{k} M_{D^k_j}.
\]

 Thus $Z$ is a $g$-tuple of commuting contractions
  and 
\[
 \begin{split}
  P_j V^* Z_j W P_j & =  \sum_{k=1}^N \epsilon_k P_j V^* \cM_{D^k_j} V P_j \\
    & = \sum_{k=1}^N \epsilon_k E^k_j \\ 
    & = C_j.
 \end{split}
\]
 Choosing $Y_j =V^* Z_j V$  completes the proof since the $Z_j$ are commuting
  self-adjoint  contractions (and $V$ is an isometry independent of $j$).
\end{proof}

\subsection{Proof of Theorem \ref{thm:BtN}}
 \label{subsec:proofBtN}
 Our dilation theoretic proof of Theorem \ref{thm:BtN} concludes in this subsection. Accordingly, suppose 
    $B=(B_1,\dots,B_g)$ is a given $g$-tuple of $m\times m$ symmetric matrices of rank at most $d$ 
   and  $[-1,1]^g \subset \fs_{B}$. We are to  show  $\ka_*(d)\cubeg \subset \cD_{B}$.

  Let
\[
 \Lambda_B(x)= \sum_{j=1}^g B_j x_j
\]
be the homogeneous linear pencil associated with $B$.
 The aim is to show that $\Lambda_B({\ka_*(d)}X)\preceq I$ for tuples $X\in\cubeg$ 
  and, by Lemma \ref{lem:ddoesit},  
  it  suffices to suppose $X$ has size $m$.   Let
  $x\in\mathbb R^m\otimes \mathbb R^m$ be a given unit vector. The proof 
  reduces to showing
\[
  \ka_*(d) \langle \Lambda_B(X) x, x\rangle \le 1.
\]

  Fix $j$ and let $\{f_{j1},f_{j2},\dots,f_{jd}\}$ denote an orthonormal basis for the range of $B_j$
 (or any $d$-dimensional subspace that contains the range of $B_j$). 
  This uses the rank at most $d$ assumption. 
  Extend this basis to an orthonormal  basis $\{f_{j1},\dots,f_{jm}\}$ of all $\mathbb R^m$.
  Note that $f_{jp}\in\{f_{j1},f_{j2},\dots,f_{jd}\}^\perp\subseteq(\im B_j)^\perp=\ker B_j$ for all $j\in\{1,\dots,g\}$ and $p\in\{d+1,\dots,m\}$
  since $B_j$ is symmetric.
  The unit vector $x$ can be written in $g$ different ways (indexed by $j\in\{1,\dots,g\}$) as 
\[
  x=\sum_{p=1}^m f_{jp} \otimes x_{jp},
\]
  for vectors $x_{jp}\in\mathbb R^m$. 
   Let   $P_j$ be the orthogonal projection onto
\[
  \cM_j := \mbox{span}(\{x_{j1},\dots,x_{jd}\})
\]
 and compute for $j$ fixed and any $m\times m$ tuple $Y$ such that
 $P_jY_jP_j = P_j X_j P_j$, 
\[
 \begin{split}
  \big\langle(B_j\otimes X_j)x,x\big\rangle & = \sum_{p,q=1}^m  \langle B_j f_{jp},f_{jq} \rangle \, \langle X_j x_{jp},x_{jq}\rangle\\
   & = \sum_{p,q=1}^d  \langle B_j f_{jp},f_{jq} \rangle \, \langle X_j x_{jp},x_{jq}\rangle\\
   & = \sum_{p,q=1}^d \langle B_j f_{jp},f_{jq} \rangle \, \langle P_j X_j P_j x_{jp},x_{jq}\rangle \\
   & = \sum_{p,q=1}^d \langle B_j f_{jp},f_{jq} \rangle \, \langle P_j Y_j P_j x_{jp},x_{jq}\rangle \\
   & = \sum_{p,q=1}^d \langle B_j f_{jp},f_{jq} \rangle \, \langle  Y_j x_{jp},x_{jq}\rangle \\
   & = \sum_{p,q=1}^m  \langle B_j f_{jp},f_{jq} \rangle \, \langle Y_j x_{jp},x_{jq}\rangle \\
   & =\big \langle(B_j \otimes Y_j)x,x \big\rangle. 
 \end{split}
\]

 From Lemma \ref{lem:cutdown} 
 there exists a Hilbert space $\mathcal K$ (infinite dimensional generally), an isometry
 $V\colon\mathbb R^m\to \mathcal K$ and a tuple of commuting self-adjoint  contractions $Z=(Z_1,\dots,Z_g)$
 acting on $\mathcal K$ such that $Y_j,$ defined by  $\ka_*(d) Y_j = V^* Z_j V,$  satisfies $P_j Y_j P_j = P_j X_j P_j$. Hence, 
\[
 \begin{split}
  \ka_*(d) \langle \Lambda_B(X)x,x\rangle  & = \ka_*(d) \langle \Lambda_B(Y)x,x\rangle \\
    & = \big\langle (I_m\otimes V^*) \Lambda_B(Z) (I_m\otimes V) x,x \big\rangle  \\ 
    & =  \langle \Lambda_B(Z) z,z\rangle,
 \end{split}
\]
 where $z=(I\otimes V)x$. In particular, $z$ is a unit vector. Since $Z$ is a commuting tuple of self-adjoint
  contractions, just as in Proposition \ref{prop:commute-include}, the inclusion $[-1,1]^g \subset \fs_{L_B}$ implies,
\[
 \langle \Lambda_B(Z) z,z \rangle \le 1. 
\] 
 The final conclusion is
\[
 \ka_*(d) \langle \Lambda_B(X)x,x\rangle \le 1.
\rlap{$\hspace{5.8cm} \qedsymbol$}
\]

\section{Free Spectrahedral Inclusion Generalities}
 \label{sec:generally free}

 This section begins with a bound on the inclusion scale
 which depends little on the LMIs involved,\index{free spectrahedral inclusion}
 Section \ref{sec:inclscbd}.
 In Subsection \ref{sec:isisci} we prove that
 the   inclusion scale equals the  commutability index,
 that is, Theorem \ref{thm:scottDilated-intro}.
In summary,  all the claims made in 
Section \ref{sec:op-dil} are established here.

\subsection{A general bound on the inclusion scale}
 \label{sec:inclscbd} 
 This subsection gives a bound on the inclusion scale
 which depends little on the LMIs involved.\index{inclusion scale}
 Recall $\fs_{L_A}$ is the spectrahedron $\cD_{L_A}(1)$ determined by the
 tuple $A$.

\begin{prop}
 \label{prop:gend}
   Suppose 
$A$ and $B$ are $g$-tuples of symmetric matrices,    
   where the $B_j$ are $d \times d $ matrices.    
Suppose further that $-\cD_{L_A}\subseteq\cD_{L_A}$.   
   If $\fs_{L_A} \subset \fs_{L_B}$, then 
   $\cD_{L_A}(n) \subset d\  \cD_{L_B}(n)$    for each $n.$
\end{prop}

\begin{lemma}
\label{lem:dblock}
  Suppose $T=(T_{j,\ell})$ is a $d\times d$ block matrix with blocks of equal square size. If
  $\|T_{j,\ell}\|\le 1$ for every $j,\ell,$ then $\|T\|\le d$. 
\end{lemma}

\begin{proof}
Recall that the Cauchy-Schwarz inequality applied with one of the vectors being the all ones vector
 gives the relation between the $1$-norm and $2$-norm on $\mathbb R^d$, namely
\[
  \big(\sum_{j=1}^d a_j\big)^2  \le d \sum_{j=1}^da_j^2\quad\text{ for all $a_1,\dots,a_d\in\R$}.
\]
Consider a vector $x=\sum_{\ell=1}^d x_\ell \otimes e_\ell$ and estimate,
\[
 \begin{split}
 \|Tx \|^2 & = \sum_{j=1}^d \big\| \sum_{\ell=1}^d T_{j,\ell} x_\ell \big\|^2 \\
&    \le   \sum_{j=1}^d (\sum_{\ell=1}^d \|x_\ell\|)^2  \\
& \le  \sum_{j=1}^d   d \sum_{\ell=1}^d \|x_\ell\|^2 \\
    & = \sum_{j=1}^d d \|x\|^2 \\
    & = d^2 \|x\|^2. 
 \end{split}
\]
  Thus, $\|Tx\|\le d\|x\|.$ 
\end{proof}

\begin{proof}[Proof of Proposition {\rm\ref{prop:gend}}]
   Let $\{e_s\}$ 
  denote the standard orthonormal basis for $\mathbb R^n$.
 Fix $1\le s\ne t\le n$ 
  and set
  $p_{s,t}^\pm:=\frac{1}{\sqrt{2}} (e_s\pm e_t)\in\R^n$.
  In particular, with 
\[
 P_{s,t}^{\pm}  = I_d \otimes p_{s,t}^\pm,
\]
 the orthonormality of the basis gives,
\[ 
  (P_{s,t}^\pm)^* P_{s,t}^\pm = I_d.
\]
  Moreover, for  $d \times d$ matrix $C$ and 
  $n \times n$ matrix $M,$
\[
   (P_{s,t}^\pm)^* (C\otimes M) P_{s,t}^\pm =  C\otimes \big(\frac{1}{2}(M_{s,s} \pm  M_{s,t} \pm M_{t,s}+M_{t,t})\big).
\]
  Hence,
\[
   (P_{s,t}^+)^* (C\otimes M) P_{s,t}^+ -  (P_{s,t}^-)^* (C\otimes M) P_{s,t}^- = C\otimes (M_{s,t}+M_{t,s}).
\]
  In particular, if $M$ is symmetric, then the right hand side is 
   $2C\otimes M_{s,t}$.

  \def\cDA{{\cD_{L_A}}}
  \def\cDB{{\cD_{L_B}}}

 Let $X\in\cD_{L_A}(n)$ be given and let
\[
   Z=  \sum_j A_j\otimes X_j.
\] 
By hypothesis, $-X\in\cD_{L_A}(n)$ too,
   so that both $\pm Z \preceq I_n.$ 
 Thus     
   $\pm (P_{s,t}^\pm)^* Z  P_{s,t}^\pm \le 1$. Hence
 $\pm (p_{s,t}^*)^\pm X p_{s,t}^\pm \in \fs_{L_A}$ 
    for each $0 \leq s,t \leq n$.
Convexity of   $\fs_{L_A}$ 
   implies
  $$
 \frac 1 2 \big( (p_{s,t}^+)^* X  p_{s,t}^+ - (p_{s,t}^-)^* X p_{s,t} ^- \big)
=  X_{s,t} :=((X_1)_{s,t},\dots,(X_g)_{s,t}) \in \fs_{L_A}.
$$
  By hypothesis, $X_{s,t} \in\fs_{L_B}$ and therefore,
\[
   T_{s,t}:= \sum B_j (X_j)_{s,t} \preceq I_d, \qquad  0 \leq s,t \leq n.
\]
Apply  Lemma \ref{lem:dblock} to the $n \times n$ block matrix 
$$
T = \sum_j X_j \otimes  B_j 
$$
to get
\[
  \| \sum B_j\otimes X_j \| \le n
\]
Likewise for $-X$, 
and therefore,
\[
  \sum B_j \otimes  X_j \preceq n I_{dn}. 
\] 
  Hence $\frac{1}{n} X \in \cD_{L_B}$. 
  At this point we have 
$\cDA(n) \subset n   \cD_{L_B}(n)$. 
  
  Since $B$ has size $d$ and
   $\cDA(d)\subset d\cD_{L_B}(d)$, it follows from
   Lemma \ref{lem:ddoesit}  that $\cDA(n)\subset d\cD_{L_B}(n)$
  for all $n$; that is,    $\cDA \subset d   \cD_{L_B}$.
\end{proof}

\begin{example}\rm
 \label{ex:sharp-two}
This example shows that, in the case $d=2$, the estimate 
    $r(A)(d) \cD_{L_A} \subset \cD_{L_B} (d) $
    of Proposition \ref{prop:gend} 
   is sharp. 
 
In this example  $ \fs_{L_A} =  \fs_{L_B}  $ is 
 the unit disc $\mathbb D=\{(x,y)\in\mathbb R^2: x^2+y^2\le 1\}$.
 We take   $L_A(X) \preceq 0$ 
 to be the infinite set\footnote{For cp fans this actually is the
  the minimal operator system structure for $\mathbb D$.} of scalar inequalities
  \[
    \sin(t) X_1 + \cos(t) X_2 \preceq I_n \quad \text{for all } t.
 \]
Next define $L_B$ to be the pencil with coefficients
$$
    B_1 =\begin{pmatrix} 1 & 0 \\ 0 & -1 \end{pmatrix}, \qquad
    B_2 = \begin{pmatrix}0 & 1 \\ 1 & 0\end{pmatrix}.
 $$
 Of course $d= \text{size} (L_B)$ is 2.

Now we show that
$\cD_{L_A}(2) \not \subset (2-\eps)  \cD_{L_B}(2)$ 
for $\eps>0$
by
selecting 
 $X_j=B_j$. Evidently,   $X\in\cD_{L_A}(2)$
  but, up to unitary equivalence, 
\[
 L_B(X) = \begin{pmatrix} 1 & 0 & 0 & 1 \\ 0 & -1 & 1 & 0\\
            0 & 1 & -1 & 0 \\ 1 & 0 & 0 & 1\end{pmatrix}.
\]
 Thus $2 I_4 -L_B(X)\succeq 0$, but if $\rho<2$, then
  $\rho I_4-L_B(X)\not\succeq 0$.   

  To complete the example, we show that $A$ can be viewed as a limit
 of tuples of matrices. Let $\{t_j:j\in\mathbb N\}$ denote an a countably dense subset of $[0,2\pi)$.
  Given $n\in\mathbb N$, let $T_n=\{t_1,\dots,t_n\}$ and let $ A^{(n)}_1$ denote the
   $n\times n$ diagonal matrix with $j$-th diagonal entry $\sin(t_j)$ and
 let $A^{(n)}_2$ denote the $n\times n$ diagonal matrix with $j$-th diagonal entry
  $\cos(t_j)$.  In particular, if $L_{A^{(n)}}(X)\succeq 0$ for all $n$, then 
  $L_A(X)\succeq 0$.  Thus, the smallest $\rho$ such that $L_{A^{(n)}}(X)\succeq 0$
  implies $L_B(\frac{1}{\rho} X)\succeq 0$  is $2$. 
\end{example}

\subsection{The inclusion scale equals the commutability index}
 \label{sec:isisci}     
The goal here is to prove Theorem \ref{thm:scottDilated-intro}
which we essentially restate as Theorem \ref{thm:scottDilated}
and then prove.\index{inclusion scale}\index{commutability index}

  Fix a tuple $A\in\mathbb S_m^g$ and a positive integer $d$.  Assume $\fs_{L_A}\subset \mathbb R^g$ is bounded.  Let
\[
 \Omega_A(d)  = \{ r\ge 0: \mbox{if } B\in\mathbb S_d^g  \mbox{ and }   \fs_{L_A} \subset \fs_{L_B}, \mbox{ then }
                    r\cD_{L_A}\subset \cD_{L_B} \}.
\] 
 Observe that $\Omega \subset [0,1]$. 
 Let $\mathcal F_A$ denote the collection of tuples $T=(T_1,\dots,T_g)$ of commuting self-adjoint operators 
  on Hilbert space whose joint spectrum lies in $\fs_{L_A}$.  Let 
\[
 \Gamma_A(d) = \{t\ge 0: \mbox{ if } X\in \cD_{L_A}(d), \mbox{ then } tX  \mbox{ dilates to a } T\in\mathcal F_A \}.
\]
 That $\Gamma_A(d) \subset [0,1]$ follows by noting that $x$ is in the boundary of $\fs_{L_A}$ 
 and if $t>1$, then $tx$ can not dilate to a $T\in\mathcal F_{A}$. 

\begin{thm}\label{thm:scottDilated}
Fix $g \in \N$.
 Assuming $\fs_{L_A}$ is bounded,
 the sets $\Gamma_A(d)$ and $\Omega_A(d)$ contain non-zero positive  numbers
 and are closed and equal. In particular, for each fixed $d\in\N$, 
\[\sup \Omega_A(d) = \sup \Gamma_A(d).\]
\end{thm}

 The supremum of $\Omega_A(d)$ is the optimal free spectrahedral inclusion constant for $A$ and $d$. Namely,
 it is the largest number with the property that if $B\in\mathbb S_d^g$ and $\fs_{L_A}\subset \fs_{L_B}$,
 then $\Omega_A(d) \cD_{L_A} \subset \cD_{L_B}$.  On the other hand, the supremum of $\Gamma_A(d)$ is the optimal
 scaling constant for $A$ and $d$ in the sense that if $X\in\cD_{L_A}(d)$ then $\Gamma_A(d)X$ dilates
  to a tuple in $\mathcal F_A$. 

\sssec{Matricial Hahn-Banach background}
 The proof of this theorem given here uses the Effros-Winkler matricial Hahn-Banach Separation Theorem \cite{EW}
 for matrix convex sets.\index{theorem!Hahn-Banach}\index{theorem!Effros-Winkler}\index{theorem!separation of matrix convex sets}
 With $g$ fixed let $\mathbb S^g$ denote the sequence $(\smatng)_n$.  
 A \df{matrix convex} subset $\cC$ of $\mathbb S^g$ containing $0$  is a sequence $(\cC(n))$
  such that 
\begin{enumerate}[(a)]
 \item  $\cC(n)\subset \smatng$ for each $n$;
 \item  $0$ is in the interior of $\cC(1);$ 
 \item  $\cC$ is  \df{closed under direct sums}:  If $X\in \cC(n)$ and $Y\in \cC(m)$, then 
 $X\oplus Y = (X_1\oplus Y_1,\dots,X_g\oplus Y_g) \in \cC(n+m)$, where
\[
 X_j \oplus Y_j = \begin{pmatrix} X_j & 0 \\ 0 & Y_j \end{pmatrix}.
\]
  \item $\cC$ is \df{closed under simultaneous conjugation by contractions}: If
 $X\in \cC(n)$ and $M$ is an $n\times m$ contraction, then
\[
  M^* X M = (M^* X_1 M,\dots,M^* X_g M) \in \cC(m).
\]
\end{enumerate}

 The matrix convex set $\cC$ is closed if each $\cC(n)$ is closed.  A version of the matricial Hahn-Banach theorem immediately applicable here can be found in \cite{HM12}.
  It says, in the language of this article, if $\cC\subset \mathbb S^g$ is closed and matrix convex 
  and if $X\in\mathbb S_d^g\setminus \cC(d)$, then
  there exists $B\in\mathbb S_d^g$ such that $\cC\subset \cD_{L_B}$, but $X\notin\cD_{L_B}(d)$.  In particular,
\[
  \cC= \bigcap\{\cD_{L_B}: B\in\mathbb S^g, \ \  \cC\subset \cD_{L_B}\}.
\]
 
\subsubsection{Proof of Theorem {\rm\ref{thm:scottDilated}} }
  For notational convenience, since $A$ and $d$ are  fixed, let $\Omega=\Omega_A(d)$, $\Gamma=\Gamma_A(d)$ and $\mathcal F=\mathcal F_A$.

  That $\Omega$ is closed is easily seen. To see that $\Omega$ 
 contains a positive number, first note that the assumption that $\fs_{L_A}$ is bounded 
  implies there exists a constant $C>0$ such that $\cD_{L_A} \subset C \cube^g.$ On the other hand, since $\cD_{L_A}$ is the set of tuples
  $X\in\mathbb S^g$ such that $I-\sum A_j\otimes X_j \succeq 0$, there is a constant $c>0$ such that $c\cube^g \subset \cD_{L_A}$. 
  By Theorem \ref{thm:BtN} there is
  a constant $s>0$ such that if $B\in\mathbb S_d^g$ and $[-1,1]^g \subset \fs_{L_B}$, then $s\cube^g \subset \cD_{L_B}$. Hence, if
  instead $\fs_{L_A}\subset \fs_{L_B}$, then 
\[
  c[-1,1]^g \subset  \fs_{L_A}  \subset \fs_{L_B}.
\]
 It follows that  $\cube^g \subset \cD_{\frac{B}{sc}}$. Thus,
\[
  \cD_{L_A} \subset C \cube^g \subset \frac{C}{sc} \cD_{L_B}.
\]
  Hence $\frac{sc}{C}\in\Omega$.

  To prove the sets $\Omega$ and $\Gamma$ are equal, first observe that Proposition \ref{prop:commute-include} implies $\Gamma \subset \Omega.$  To prove the converse,
  suppose $r\in\Omega$.  Let $\Sigma$ denote the smallest closed matrix convex set with the property that
   $\Sigma(1)=\cD_{L_A}(1)=\fs_{L_A}$.  The equality,
\[
  \Sigma(d) = \bigcap\{ \cD_{L_B}(d): B\in\mathbb S_d^g \mbox { and } \Sigma \subset \cD_{L_B}\}
\]
  is a  consequence of the Effros-Winkler matricial Hahn-Banach Separation Theorem \cite{EW}.
  To prove this assertion, first note the inclusion $\Sigma(d)$ into the set on the right hand side is obvious. On the other hand,
  if $X\not\in\Sigma(d)$, then by Effros-Winkler theorem produces a $B\in\mathbb S_d^g$ such that
  $\Sigma \subset \cD_{L_B}$, but $X\not\in\cD_{L_B}(d)$ and the reverse inclusion follows.  
  Now the definition of $\Sigma$ implies $\Sigma \subset \cD_{L_B}$ if and only if 
  $\fs_{L_A}=\Sigma(1) \subset \fs_{L_B}$. Hence,
\[
 \Sigma(d) = \bigcap\{\cD_{L_B}(d): B\in\mathbb S_d^g \mbox { and } \fs_{L_A} \subset \fs_{L_B}\}.
\]
 Thus, as $\fs_{L_A} \subset \fs_{L_B}$ implies $r\cD_{L_A} \subset \cD_{L_B}$,
\[
 \Sigma(d) \supseteq \bigcap \{\cD_{L_B}(d): B\in\mathbb S_d^g \mbox { and } r \cD_{L_A} \subset \cD_{L_B}\}
\]
 and therefore $\Sigma(d)  \supseteq r \cD_{L_A}(d)$.

 It remains to show, if $Z\in\Sigma$, then $Z$  dilates to some $T\in\mathcal F$.
  For positive integers $n$,  let
\[
  \Lambda(n) = \{X\in\smatng: X \mbox{ dilates to some } T\in\mathcal F\}.
\]
   The sequence $\Lambda = (\Lambda(n))_n$ is a matrix convex set with 
   $\Lambda(1)=\Sigma(1)$. To prove that $\Lambda(n)$ is closed, suppose
   $(X^k)_k$ is a sequence from $\Lambda(n)$ which converges to $X\in\smatng$.
   For each $k$ there is a Hilbert space $\mathcal H_k$, a sequence of
   commuting self-adjoint contractions $T^k=(T^k_1,\dots,T^k_g)$ on $\mathcal H_k$
   with joint spectrum in $\fs_{L_A}$ 
  and an isometry $V_k:\mathbb R^n \to \mathcal H_k$ such that 
\[
   X^k=  V_k^* T^k V_k.
\]
   Let $T$ denote the tuple $\oplus T^k$ acting on the Hilbert space $\mathcal H= \oplus \mathcal H_k$.
   The fact that each $T^k$ has joint spectrum in the bounded set 
  $\fs_{L_A}$ and that each $T^k_j$ is self-adjoint, implies  the sequence $(T^k)_k$ is uniformly bounded.
  Hence $T$ is a bounded operator.
   Let $\mathcal S$ denote the operator system equal to the span of $\{I,T_1,\dots,T_g\}$ (this set 
    is self-adjoint since each $T_k$ is self-adjoint) and let  $\phi:\mathcal S \to M_n$ denote
    the unital map determined by \[\phi(T_j)=X_j.\] 
     It is straightforward to check that $\phi$ is well defined. On the other hand,
     next it will be shown that $\phi$ is completely positive, an argument which also shows that $\phi$ 
    is in fact well defined.  
     If $C=(C_0,\dots,C_g) \in\mathbb S_m^{g+1}$ and
    $C_0\otimes I + \sum C_j \otimes T_j \succeq 0$, then $C_0\otimes I + \sum C_j\otimes T_j^k \succeq 0$
   for each $k$. Thus, $C_0 \otimes I + \sum C_j \otimes X^k_j \succeq 0$ for all $k$ and
  finally $C_0 \otimes I + \sum C_j \otimes X_j \succeq 0$.  Thus $\phi$ is completely positive.
  Given a Hilbert space $\mathcal E$, let $B(\mathcal E)$ denote the C-star algebra of bounded operators on $\mathcal E$. 
  By the standard application of Stinespring-Arveson (\cite[Corollary 7.7]{Pau}) there exists a Hilbert space $\mathcal K,$  a representation 
   $\pi:B(\mathcal H)\to B(\mathcal K)$ and an isometry $W:\mathbb R^n\to \mathcal K$ such that
\[
   X_j = \phi(T_j) = W^* \pi(T_j) W.
\]
 Since $\pi$ is a representation, the tuple $\pi(T) = (\pi(T_1),\dots,\pi(T_g))$ is a commuting tuple
  of self-adjoint contractions on the Hilbert space $\mathcal K$ with joint spectrum in $\fs_{L_A}$. Hence $X\in\Lambda(n)$.

  Now $\Lambda$ is a closed matrix convex set with 
  $\Lambda(1)\supseteq \Sigma(1)$. Hence, $\Sigma\subset \Lambda$ by the definition of $\Sigma.$
  In particular, $r\cD_{L_A}(d) \subset \Sigma(d)\subset \Lambda(d)$ and the proof is complete.      
\hfill\qedsymbol

\subsubsection{Matrix cube revisited}
  Returning to the special case of the matrix cube,\index{cube}\index{free cube} 
  for $g,d \in \N$ define \index{$\rho_g(d) $} 
\[
 \rho_g(d)=
  \sup\{r\ge 0:  \mbox{if } B\in\mathbb S_d^g  \mbox{ and }   [-1,1]^g\subset \fs_{L_B}, \mbox{ then }
                    r\cubeg \subset \cD_{L_B} \}.
\] 
  For $d$ fixed, the sequence $(\rho_g(d))_{g=1}^\infty$ is evidently decreasing and hence converges to some 
  $ \rho(d)$. 
   Similarly,  let $\mathcal F_g$ denote the collection of tuples $T=(T_1,\dots,T_g)$ of commuting self-adjoint contractions
  on Hilbert space 
  and  let \index{$\tau_g(d)$}
\[
 \uptau_g(d)  = \sup \{t\ge 0: \mbox{ if } X\in \cubeg(d), \mbox{ then } tX  \mbox{ dilates to a } T \in\mathcal F_g \}.
\]
  The sequence $(\uptau_g(d))_{g=1}^\infty$ also 
  decreases and hence converges to some $ \uptau(d)$. 
  By Theorem \ref{thm:scottDilated}, $\uptau_g(d)=\rho_g(d)$ for all $g,d$. 

\begin{cor}
 \label{cor:yetanotherproof}
   $\uptau(d)= \lim \uptau_g(d) =\lim \rho_g(d)= \frac{1}{\vartheta(d)}.$
\end{cor}

 \begin{rem}
 To this point $\uptau(d)=\lim \rho_g(d)$  has been derived through  operator theoretic means
  not involving $\vartheta$ and \cite{BtN}.  Of course, in view of Theorems \ref{thm:BtN} and \ref{thm:sharp},
   $\uptau(d)=\frac{1}{\vartheta(d)}$.
   On the other hand, it is  not obviously possible to recover Theorem \ref{thm:dilate} from this corollary. 
   See Remark \ref{rem:cube-dil}. 
 \end{rem}

\section{Reformulation of the Optimization Problem}
\label{sec:opt}

 The goal here is to bring pieces together
 in order to lay out our key classical optimization problem
\eqref{eq:BTbd} in terms of regularized Beta functions (see Problem \ref{prob:opt} and Proposition \ref{prop:thetafirst}).
The reformulated optimization problem is then solved
in 
Section \ref{sec:nov} after preliminary work in 
Sections \ref{sec:simmhalf} and \ref{sec:lowBdSIS}.

Recall that $\ds\frac1{\th(d)}$ for $d\ge2$ equals the minimum over all $s,t\in\N$ and $a,b\in\R_{>0}$ such that $s+t=d=sa+tb$ of
\[2a\frac sd\ I_{\frac a{a+b}}\left(\frac t2,1+\frac s2\right)+2b\frac td\ I_{\frac b{a+b}}\left(\frac s2,1+\frac t2\right)-1.\]
(Combine Lemma \ref{lem:compalpha} with Proposition
 \ref{prop:starastarb}.)
Note that the constraint $d=sa+tb$ is just a matter of scaling of $a$ and $b$ with the same factor which won't affect the substitution
\[p=\frac b{a+b}\in(0,1)\]
which we are now going to make. This substitution entails $\ds1-p=\frac a{a+b}$, $d=(a+b)(sp+t(1-p))$,
\[\frac ad=\frac a{(a+b)(sp+t(1-p))}=\frac{1-p}{sp+t(1-p)}\]
and
\[\frac bd=\frac b{(a+b)(sp+t(1-p))}=\frac{p}{sp+t(1-p)}.\]
By continuity, we can let $p$ range over the compact interval $[0,1]$. We therefore observe that
$\ds\frac1{\th(d)}$ equals the minimum over all $s,t\in\N$ with $s+t=d$ of the minimum value of the function $f_{s,t}\colon[0,1]\to\R$
given by
\beq\label{eq:fst}
f_{s,t}(p)=\frac{2(1-p)sI_{1-p}\left(\frac t2,1+\frac s2\right)+2ptI_p\left(\frac s2,1+\frac t2\right)}{(1-p)s+pt}-1
\eeq
for $p\in[0,1]$.
Using the standard identities $\ds I_p(x,y)=\frac{B_p(x,y)}{B(x,y)}$, $\ds \frac\partial{\partial p}B_p(x,y)=p^{x-1}(1-p)^{y-1}$, $\ds B\left(\frac s2,1+\frac t2\right)=\frac t{s+t}B\left(\frac t2,\frac s2\right)$
and $\ds B\left(\frac t2,1+\frac s2\right)=\frac s{s+t}B\left(\frac s2,\frac t2\right)$, one can easily verify that the derivative $f_{s,t}'$ of $f_{s,t}$ takes the surprisingly simple form given
by
\[f_{s,t}'(p)=\frac{2st}{((1-p)s+pt)^2}\left(I_p\left(\frac s2,1+\frac t2\right)-I_{1-p}\left(\frac t2,1+\frac s2\right)\right)\]
for $p\in[0,1]$
(two of the six terms cancel when one computes the derivative using the product and quotient rule).
This shows that $f_{s,t}$ is strictly decreasing on $[0,\sigma_{s,t}]$ and strictly increasing on $[\sigma_{s,t},1]$ where $\sigma_{s,t}\in(0,1)$ is defined by
\begin{equation}\label{eq:defSIS}
I_{\sigma_{s,t}}\left(\frac s2,1+\frac t2\right)=I_{1-\sigma_{s,t}}\left(\frac t2,1+\frac s2\right).
\end{equation}

We shall use (in Section \ref{sec:nov})
bounds on $\sis$ for $s,t\in\N$. 
Lower bounds are given in
Corollary \ref{cor:lowBdSIS}, while upper bounds
are presented in Theorem \ref{simmons+}, cf.~\eqref{eq:chain1}.

\begin{problem}\label{prob:opt}\rm
   Given a positive integer $d$, minimize \[
   \ds f_{s,t}(\sigma)=\frac{2(1-\sigma)sI_{1-\sigma}\left(\frac t2,1+\frac s2\right)+2\sigma tI_{\sigma}\left(\frac s2,1+\frac t2\right)}{(1-\sigma)s+\sigma t}-1\]
 subject to the constraints 
\begin{enumerate}[(i)]
 \item $s,t\in\mathbb N$ and $s+t=d$;
 \item $\ds s\ge \frac{d}{2};$ 
 \item $0\leq\sigma\leq1$; and 
 \item \label{it:determinesp}  $\ds I_{\sigma}\left(\frac s2, \frac t2 +1\right) = I_{1-\sigma}\left(\frac t2,\frac s2 +1\right).$ 
\end{enumerate}
\end{problem}

Since Problem \ref{prob:opt} computes $\vartheta(d)$,
  Theorem \ref{thm:thetaExplicit} can be rephrased as follows.

\begin{prop}
 \label{prop:thetafirst}
   When $d$ is even the minimum in Problem {\rm\ref{prob:opt}} occurs when $s=t=\frac{d}{2}$
  and in this case $\sigma=\frac 12$. 
  When $d$ is odd, the minimum in Problem {\rm\ref{prob:opt}}
  occurs when $s=\frac{d+1}{2}$ and $t=\frac{d-1}{2}$.
  In this case
  $\sigma$, and hence the optimum, is implicitly determined by condition 
  \eqref{it:determinesp}.  
\end{prop}

  \def\sist{\sigma_{s,t}}
  
The proof of Proposition \ref{prop:thetafirst} is organized as follows. The next section contains an improvement
 of Simmons' Theorem from probability. It is used to obtain the bound
\[
   \sist \le \frac{s}{s+t}
\]
 valid for $s\ge \frac{d}{2}.$ 
In Section \ref{sec:lowBdSIS}
we present the lower bound 
\[
 \frac{s+2}{s+t+4} \le \sist
\]
 valid for $s\ge \frac{d}{2}$.
Finally, the proof of Proposition \ref{prop:thetafirst} is completed
  in Section \ref{sec:nov}. 
  
\section{Simmons' Theorem for Half Integers}

\label{sec:simmhalf}
This material has been motivated by the Perrin-Redside \cite{PR07}
proof of Simmons' inequality from discrete probability
which has the following simple interpretation.\index{theorem!Simmons'}
Let $s,d\in\N$ with $s\geq\frac d2$.
Toss a coin
whose probability for head is $\frac sd$, 
 $d$ times.
(So the expected number of head is $s$.)
Simmons' inequality then states that 
the probability of getting  $<s$  heads
\emph{is smaller than}
the probability of getting $>s$ heads.

Theorem \ref{simmons+} below is a half-integer  generalization of  Simmons' Theorem.

\begin{thm}
 \label{simmons+}
For $d\in\mathbb N$ and $s,t\in\mathbb N$  with $s+t=d$, if $\frac{d}{2}\le s<d$,  then  
\begin{equation}
 \label{eq:someconj}
   I_{\frac sd}\left(\frac s2+1,\frac t2\right) 
\geq  1- I_{\frac sd}\left(\frac s2,\frac t2+1\right).
\end{equation}
 
 Equivalently,\index{equipoint} 
\beq\label{ineq:simm}
  \sis \le \frac{s}{d}
\eeq
 for $\frac d2 \leq s<d$ with $s\in\N$.
\end{thm}

The proof of this consumes this whole section and we begin
by setting notation.
For $s,t\in\R$ with $s>-2$ and $t>0$ (in the sequel, $s$ and $t$ will mostly be integers with $s\ge-1$ and $t\ge 1$; we really need the
case $s=-1$, for example after \eqref{eq:dome1})
and $d:=s+t$, let $f_s$ denote the density function of the Beta distribution 
\[\frac d2 \ B\left(\frac s2+1,\frac t2\right),\]
i.e.,
\begin{equation}\label{eq:deff}
f_s(x):= \begin{cases}
0 & x\leq0 \\
\displaystyle\frac1{B\left(\frac s2+1,\frac t2\right)} \left(\frac d2\right)^{-1-\frac s2} x^{\frac s2}
\left(1-\frac2dx\right)^{\frac t2-1} & 0< x < \frac d2 \\
0 & x\geq\frac d2.
\end{cases}
\end{equation}

Consider the function
\[
  \mathscr F(s,p) = I_p\left(\frac s2+1,\frac t2\right) + I_p \left(\frac s2,\frac t2+1\right) -1. 
\]
 Equation \eqref{eq:someconj} is the statement that $\mathscr F(s,\frac{s}{d}) \ge 0.$ 
 Since, for $s$ fixed, $\mathscr F(s,p)$ strictly increases with $p$ and $\sis$ is determined by
  $\mathscr F(s,\sis)=0$, the second part of the theorem 
  is obviously equivalent to the first one.

Let 
\begin{equation}
 \label{def:bs}
b_s:=\int_0^{\frac s2} f_s = 
I_{\frac sd}\left(\frac s2+1,\frac t2\right)
\end{equation}
 and 
\begin{equation}
 \label{eq:def:as}
  a_s:= \int_{\frac s2}^\infty f_{s-2} =   1- \int_0^{\frac s2} f_{s-2}
=
1-   I_{\frac sd}\left(\frac s2,\frac t2+1\right).  
\end{equation}
 Equation \eqref{eq:someconj} is equivalent to 
\beq\label{eq:defC}
 c_s := b_s -a_s=\mathscr F\left(s,\frac{s}{d}\right)  \ge 0
\eeq
 for $d,s\in\R_{>0}$ with $\frac{d}{2} \le s <d$.

\ssec{Two step monotonicity of $c_s$}

In this subsection, in Proposition \ref{prop:simmons},
we show  for $s,d\in\R$ with $\frac d2\leq s \leq d-4$, that
$c_{s+2}\geq c_s$.
Note that $\frac d 2 \leq d-4 $ implies $ d \geq 8$.

\begin{lem}\label{lem:derivF}
We have for $s\in\R$ with $0<s<d-2$,
\begin{equation}\label{eq:derivF}
\begin{split}
(x f_s)' & = \left(1+\frac s2\right) (f_s-f_{s+2}) \\
\big((d-2x) f_{s+2}\big)' & = (d-s-2) (f_s-f_{s+2}).
\end{split}
\end{equation}
\end{lem}

\begin{proof}
Straightforward.
\end{proof}

For notational convenience we introduce, for $s\in\R$ with $-2<s\le d-3,$
\[
\cI_s:=\int_{\frac s2}^{\frac s2+1} f_s.
\]

\begin{lem}\label{lem:diffAB}
For $s\in\R$ with $0<s< d-2$,
\begin{equation}\label{eq:diffAB}
\begin{split}
a_{s+2}-a_s & = f_s\left(\frac s2\right)-\cI_s \\
b_{s+2}-b_s & = \cI_s-f_s\left(\frac s2+1\right).
\end{split}
\end{equation}
\end{lem}

\begin{proof}
This is a consequence of the recursive formulas in Lemma \ref{lem:derivF}:
\[
\begin{split}
a_{s+2}-a_s & = \int_{\frac s2+1}^\infty f_s - \int_{\frac s2}^\infty f_{s-2} \\
&= \int_{\frac s2}^\infty (f_s- f_{s-2}) - \int_{\frac s2}^{\frac s2+1}  f_{s} \\
 & \stackrel{\eqref{eq:derivF}}{=} -\frac{d-2x}{d-s} f_s |_{\frac s2}^\infty - \cI_s \\
 & = f_s\left(\frac s2\right)-\cI_s.
 \end{split}
 \]
  Similarly,
 \[
\begin{split}
b_{s+2}-b_s & = \int_0^{\frac s2+1} f_{s+2}-\int_0^{\frac s2} f_s \\
&= \int_0^{\frac s2+1} (f_{s+2}-f_s) + \int_{\frac s2}^{\frac s2+1} f_s \\
& \stackrel{\eqref{eq:derivF}}{=} - \frac{x f_s}{1+\frac s2}|_0^{\frac s2+1} + \cI_s \\
&= \cI_s-f_s\left(\frac s2+1\right). \qedhere
\end{split}
 \]
 \end{proof}

\begin{lem}\label{lem:diffC}
If $s\in\R$ with $-2<s< d-2$, then
\begin{equation}\label{eq:diffC}
c_{s+2}-c_s = 2 \cI_s - f_s\left(\frac s2\right)-f_s\left(\frac s2+1\right).
\end{equation}
\end{lem}

\begin{proof}
This is immediate from \eqref{eq:defC} and Lemma \ref{lem:diffAB}.
\end{proof}

\begin{lem}\label{lem:f''}
For  $0< x< \frac d 2$ and $0 < s\leq d,$ the inequality 
$f_s^{\prime\prime}(x)<0$ holds  if and only if 
\begin{equation}\label{eq:f''}
\frac{4 (d-4) (d-2)}{d^2} x^2-\frac{4 (d-4) s }{d}x+(s-2) s <0.
\end{equation}
\end{lem}

\begin{proof}
Note that for $0<x<\frac d2$,
\[
f^{\prime\prime}_s(x) = \frac{2^{\frac{s}{2}-1} d^{-s/2} x^{\frac{s}{2}-2} \left(1-\frac{2
   x}{d}\right)^{\frac{d-s}{2}} \left(d^2 (s-2) s-4 (d-4) d s x+4 (d-4)
   (d-2) x^2\right)}{(d-2 x)^3 B\left(\frac{s+2}{2},\frac{d-s}{2}\right)}.
 \]
Pulling a factor of $d^2$ out of the last factor  in the numerator yields \eqref{eq:f''}.
 \end{proof}

\begin{lem}\label{lem:fConcave}
If  $d,s\in\R$ and $\frac d2\le s\leq d-4,$ then $f_s$ is concave on $[\frac s2,\frac s2+1]$.
\end{lem}

\begin{proof}
Since $s\geq\frac d2$ with $s\leq d-4$, then $d\geq8$ and thus $s\geq4$.
For very small $x>0$, the left-hand side of \eqref{eq:f''} is positive and has a positive leading coefficient. 
So it suffices to verify \eqref{eq:f''} for $x=\frac s2$ and $x=\frac s2+1$.
Let $F_s(x)$ denote the left-hand side of \eqref{eq:f''}.
Then
\[
\begin{split}
F_s\left(\frac s2\right)  & = 
\frac{2 s} {d^2} \left(-d^2+d s+4 s\right) \\
&\leq \frac{2 s} {d^2} 
\left(-d^2+d (d-4)+4 (d-4) \right) \\
& = -\frac{32 s}{d^2} \\& <0.
\end{split}
\]
Similarly,
\[
\begin{split}
F_s\left(\frac s2+1\right)  = 
-\frac{2 (d-s-2) }{d^2}\big(d (s-2)+4 (s+2)\big) 
<0. \qedhere
\end{split}
\]
\end{proof}

\begin{prop}\label{prop:simmons}
For 
$d,s\in\R$ with $ \frac d2\leq s\leq d-4$, we have
\begin{equation}\label{eq:cDrop}
c_{s+2} > c_s.
\end{equation}
Furthermore, 
\begin{equation}\label{eq:cZero}
c_{\frac d2}=0.
\end{equation}
\end{prop}

\begin{proof}
Since under the given constraints on $s$, 
the function $f_s$ is concave
on $\left(\frac s2,\frac s2+1\right)$, 
its integral 
$\cI_s$ over this interval is bigger than 
\[
\frac 12  \Big(f_s\left(\frac s2\right)+f_s\left(\frac s2+1\right) \Big).
\]
The Equation \eqref{eq:cDrop}
now follows from Lemma \ref{lem:diffC}.

Using $I_x(a,b)=1-I_{1-x}(b,a)$ we get that
\[
a_{\frac d2}  = 1-I_{\frac12} \left(\frac d4,\frac d4+1\right) 
 = I_{\frac12} \left(\frac d4+1 ,\frac d4\right) 
 = b_{\frac d2},
\]
whence $c_{\frac d2}=0$.
\end{proof}

For $d\ge 8$ and even, an implication of two step monotonicity in Proposition \ref{prop:simmons}
together with  \eqref{eq:cZero} 
 is that either
\[
 \min \{c_s: \frac{d}{2} \le s <d\} = c_{d-1}
\]
 or 
\beq
\label{eq:simeven}
\min \{c_s: \frac{d}{2} \le s <d, \, s \mbox{ even}\} = c_{\frac d 2} =0
\quad  \text{ and } \quad
\min \{c_s: \frac{d}{2} \le s <d, \, s\mbox{ odd}\}  = c_{\frac d 2  +  1 }.
\eeq
Likewise for $d\ge 8$ and odd either
\[
 \min\{c_s: \frac{d}{2} \le s <d \} = c_{d-1}
\]
 or 
\beq
\label{eq:simodd}
\min_{s  \ \rm even} c_s = c_{\frac d 2 +1}
\quad \text{ and } \quad
\min_{s \ \rm odd} c_s = c_{\frac d 2  +  \frac 3 2 }. 
\eeq

  Thus proving Theorem \ref{simmons+} for an even integer $d\ge 8$ reduces to establishing 
 that $c_{d-1}>0$ and $c_{\frac d2 +1}$ are nonnegative and proving the result
 for $d\ge 8$ odd reduces to showing in addition that $c_{\frac{d}{2}+1}$ and $c_{\frac{d}{2}+\frac 32}$
 are both  nonnegative, facts established in  Sections \ref{sec:simms-half-even}
and   \ref{sec:simms-half-odd} respectively.   Finally, that 
 $c_s \geq 0$ for all $ d,s$ with $\frac d 2 \leq s \leq d < 8$
 was checked symbolically by Mathematica.

\subsection{The upper boundary case}
 This subsection is devoted to proving the following lemma
 which is essential for proving both the even and odd cases 
  of Proposition \ref{prop:simmons}. 

\begin{lem}\label{lem:c-1}
$c_{d-1}>0$ for $d\in\N$ with $d\ge2$.
  In addition if $d\in\R$ with $d\geq6$ then $c_{d-1}>0$.
\end{lem}

\begin{proof}
It is easy to verify the given inequality 
by hand for $d=2,3,4,5$. Without loss of generality we assume 
$d\in\R$ with
 $d\ge6$ in what follows.

Recall that $c_{s}=b_s-a_s$. Observe that
\[
b_{d-1} = \int_0^{\frac{d-1}2} f_{d-1} 
 =1- \int_{\frac{d-1}2}^{\frac d2} f_{d-1},
\]
and recall that,
\[
\begin{split}
a_{d-1} &=  \int_{\frac{d-1}2}^{\frac d2} f_{d-3}.
\end{split}
\]
Hence $c_{d-1}\ge0$ iff
\begin{equation}\label{eq:dome1}
\int_{\frac{d-1}2}^{\frac d2} (f_{d-1}+f_{d-3}) \leq1.
\end{equation}
We next use Lemma \ref{lem:derivF} with $s=d-3\ge-1>-2$ to express
\[
f_{d-1}=f_{d-3} - \big( (d-2x) f_{d-1}\big)',
\]
whence the left-hand side of \eqref{eq:dome1} transforms into
\begin{equation}\label{eq:dome2}
\begin{split}
\int_{\frac{d-1}2}^{\frac d2} (f_{d-1}+f_{d-3}) & = 
2 \int_{\frac{d-1}2}^{\frac d2} f_{d-3} 
+ f_{d-1}\left(\frac {d-1}2\right).
\end{split}
\end{equation}

\begin{sublem}\label{sublem:f-3}
$f_{d-3}$ is concave on $(\frac{d-1}2,\frac d2)$.
\end{sublem}

\begin{proof}
First note that
\[
f_{d-3}''(x)=
\frac{2^{\frac{d-5}{2}} d^{\frac{1}{2} (-d-1)} x^{\frac{d-7}{2}}
   \left((d-5) (d-3) d^2+4 (d-4) (d-2) x^2-4 (d-4) (d-3) d
   x\right)}{(d-2 x) \sqrt{1-\frac{2 x}{d}}
   B\left(\frac{d-1}{2},\frac{3}{2}\right)}
   \]
   by a straightforward computation. Thus the sign of $f_{d-3}''(x)$
   is determined by that of
\begin{multline*}
(d-5) (d-3) d^2-4 (d-4) (d-3) d\
   x+4 (d-4) (d-2) x^2\\
   =4 (d-4) (d-2) \left(x-\frac{(d-3) d}{2 (d-2)}\right)^2-\frac{2 (d-3) d^2}{d-2}
\end{multline*}
Since $d>5$ and $\frac{(d-3) d}{2 (d-2)}\le\frac{d-1}2\le x\le\frac d2$, this expression can be bound as follows:
\begin{multline*}
(d-5) (d-3) d^2-4 (d-4) (d-3) d\
   x+4 (d-4) (d-2) x^2 
\\\le4 (d-4) (d-2) \left(\frac d2-\frac{(d-3) d}{2 (d-2)}\right)^2-\frac{2 (d-3) d^2}{d-2}=-d^2<0.    
    \qedhere
   \end{multline*}
\end{proof}

By Sublemma \ref{sublem:f-3}, 
\[
2 \int_{\frac{d-1}2}^{\frac d2} f_{d-3} \leq f_{d-3}\left(\frac d2-\frac 14\right).
\]
It thus suffices to establish
\begin{equation}\label{eq:dome3}
f_{d-3}\left(\frac d2-\frac 14\right)+f_{d-1}\left(\frac {d}2-\frac 12\right) \leq 1.
\end{equation}
The left-hand side of \eqref{eq:dome3} expands into
\[
\Psi(d):=\frac{2 d^{1-\frac{d}{2}}
   \left((d-1)^{\frac{d-3}{2}}+2^{1-\frac{d}{2}} (2
   d-1)^{\frac{d-3}{2}}\right) \Gamma
   \left(\frac{d}{2}\right)}{\sqrt{\pi } \Gamma
   \left(\frac{d-1}{2}\right)}.
   \]
We use \cite[(1.7)]{KV71}:
\[
\frac{\Gamma
   \left(\frac{d}{2}\right)}{\Gamma
   \left(\frac{d-1}{2}\right)}\leq\frac{ d^{\frac d2-\frac 12}}{\sqrt{2e}\ (d-1)^{\frac d2-1}}.
 \]
 Using this, $\Psi(d)\leq1$ will follow once we establish
 \[
 \sqrt{\frac{\pi e}{2 d}} (d-1)^{\frac d2-1} \geq 
  (d-1)^{\frac d2-\frac32}+\frac1{\sqrt2} 
  \left(d-\frac12\right)^{\frac d2-\frac32},
  \]
  i.e.,
  \begin{equation}\label{eq:dome4}
  \sqrt{\frac{\pi e}{2 }} \geq 
  \sqrt d \left(
  (d-1)^{-\frac12}+\frac1{\sqrt2} 
  \left(d-\frac12\right)^{-\frac12}
\left( 1+\frac1{2(d-1)}\right)^{\frac d2-1}  
  \right).
  \end{equation}
 
\begin{sublem}\label{sublem:e14}
The sequence
\begin{equation}\label{eq:defE}
\left( 1+\frac1{2(d-1)}\right)^{\frac d2-1}  
  \end{equation}
  is increasing with limit 
  \[
  \sqrt[4]{e}.
  \]
\end{sublem}

\begin{proof}
Letting $\delta=2(d-1)$, \eqref{eq:defE} can be
rephrased as
\[
\left(1+\frac1{\delta}\right)^{\frac{\delta}4} \left(1+\frac1{\delta}\right)^{-\frac12} .
\]
The first factor is often used to define $e$ and is well-known to form an increasing sequence. It is clear that the sequence $\left(1+\frac1{\delta}\right)^{-\frac12}$ is increasing.
\end{proof}

Now the right-hand side of \eqref{eq:dome4} can be bound above by
\begin{multline}\label{eq:dome5}
\sqrt d \left(
  (d-1)^{-\frac12}+\frac1{\sqrt2} 
  \left(d-\frac12\right)^{-\frac12}
\left( 1+\frac1{2(d-1)}\right)^{\frac d2-1}  
  \right) \\
  \leq
\sqrt d \left(
  (d-1)^{-\frac12}+\frac{\sqrt[4]e}{\sqrt2} 
  \left(d-\frac12\right)^{-\frac12}
  \right) =: \Phi(d).  
  \end{multline}
Both of the sequences
\[
\sqrt d 
  (d-1)^{-\frac12}
\qquad\text{and}\qquad  
  \sqrt d
  \left(d-\frac12\right)^{-\frac12}
  \]
  are decreasing, and the limit as $d\to\infty$ of the right-hand side of
  \eqref{eq:dome5} is
  \[
  1+\frac{\sqrt[4]e}{\sqrt2}\approx 1.90794 .
 \]
Since
\[
\Phi(6)=
\sqrt{\frac{6}{5}}+\sqrt{\frac{6}{11
   }} \sqrt[4]{e} 
\leq \sqrt{\frac{\pi e}2},\]
all this establishes \eqref{eq:dome3} for $d\geq6$ as was required.
\end{proof}

\subsection{The lower boundary cases for $d$ even}
\label{sec:simms-half-even}
Here  we prove
$c_{\frac d2+1}>0$ for $d \in \RR$ with $d \geq 2$ 
(the other lower boundary case, $c_{\frac{d}{2}} \ge 0$ was already proved),
thus establishing \eqref{eq:simeven} 
and proving Theorem \ref{simmons+} for $d$ even.

\begin{lem}\label{lem:c+1}
  $c_{\frac d2+1}>0$ for $d\in\R$ such that $d\ge 2$.
\end{lem}

\begin{proof}
We want to prove that
\begin{equation}\label{eq:todo1}
I_{\frac12+\frac1d}\left(\frac d4+\frac 32,\frac d4-\frac12\right)+
I_{\frac12+\frac1d}\left(\frac d4+\frac 12,\frac d4+\frac12\right)\geq1.
\end{equation}
Using
\begin{equation}\label{eq:Ix+-}
I_x(a+1,b-1)= I_x(a,b) - \frac{x^a(1-x)^{b-1}}{a\ B(a,b)}
\end{equation}
we rewrite the first summand in \eqref{eq:todo1} as
\begin{equation}\label{eq:insert1}
I_{\frac12+\frac1d}\left(\frac d4+\frac 32,\frac d4-\frac12\right)
= I_{\frac12+\frac1d}\left(\frac d4+\frac 12,\frac d4+\frac12\right)
- \frac{ \left(\frac12+\frac1d\right)^{\frac12+\frac d4} \left(\frac12-\frac1d\right)^{-\frac12+\frac d4}}{\left(\frac d4+\frac12\right) B\left(\frac d4+\frac 12,\frac d4+\frac 12\right)}.
\end{equation}
Upon multiplying \eqref{eq:todo1} with $B\left(\frac d4+\frac 12,\frac d4+\frac 12\right)$ and using \eqref{eq:insert1}, \eqref{eq:todo1} is equivalent to
\begin{equation}\label{eq:todo2}
2 B_{\frac12+\frac1d}\left(\frac d4+\frac 12,\frac d4+\frac12\right)
\geq B \left(\frac d4+\frac 12,\frac d4+\frac12\right) + \frac{ \left(\frac12+\frac1d\right)^{\frac12+\frac d4} \left(\frac12-\frac1d\right)^{-\frac12+\frac d4}}{\frac d4+\frac12}.
\end{equation}
The left-hand side of this inequality can be rewritten as
\[
\begin{split}
2 B_{\frac12+\frac1d}\left(\frac d4+\frac 12,\frac d4+\frac12\right)
& = 2 \int_0^{\frac12+\frac1d} x^{\frac d4-\frac 12}(1-x)^{\frac d4-\frac 12} 
\, dx \\
&= 2 B_{\frac12} \left(\frac d4+\frac 12,\frac d4+\frac12\right)
+ 2
\int_{\frac12}^{\frac12+\frac1d} x^{\frac d4-\frac 12}(1-x)^{\frac d4-\frac 12} 
\, dx \\
&= B \left(\frac d4+\frac 12,\frac d4+\frac12\right)
+2
\int_{\frac12}^{\frac12+\frac1d} x^{\frac d4-\frac 12}(1-x)^{\frac d4-\frac 12} 
\, dx.
\end{split}
\]
Now \eqref{eq:todo2} is equivalent to
\begin{equation}\label{eq:todo3}
2\int_{\frac12}^{\frac12+\frac1d} x^{\frac d4-\frac 12}(1-x)^{\frac d4-\frac 12} 
\, dx \geq \frac{ \left(\frac12+\frac1d\right)^{\frac12+\frac d4} \left(\frac12-\frac1d\right)^{-\frac12+\frac d4}}{\frac d4+\frac12}.
\end{equation}
The derivative of the  integrand on the left hand side 
equals
\[
-\frac{1}{4} (d-2) (2 x-1) ((1-x) x)^{\frac{d-6}{4}}
\]
and is thus nonpositive on $[\frac12,1]\supseteq[\frac12,\frac12+\frac1d]$. Hence
\[
\begin{split}
2 \int_{\frac12}^{\frac12+\frac1d} x^{\frac d4-\frac 12}(1-x)^{\frac d4-\frac 12} 
\, dx & \geq \frac 2d \left(\frac12+\frac1d\right)^{\frac d4-\frac 12}
\left(\frac12-\frac1d\right)^{\frac d4-\frac 12}\\
&= \frac{ \left(\frac12+\frac1d\right)^{\frac12+\frac d4} \left(\frac12-\frac1d\right)^{-\frac12+\frac d4}}{\frac d4+\frac12},
\end{split}
\]
as desired.
\end{proof}

This concludes the proof of 
Theorem \ref{simmons+}  for even $d$.

\subsection{The lower boundary cases for $d$ odd}
\label{sec:simms-half-odd}
  In this subsection we establish two key inequalities: 
  \[
c_{\frac d2+\frac12}>0 \quad \text{ and }\quad c_{\frac d2+\frac32}>0.
\]
We show that the first holds for $d\in \mathbb R$ with $d\ge 3$ and the second for $d\in\mathbb N$ with $d\ge 3$
  or for $d\in\mathbb R$ with $d\ge 16.$   Combined with Lemma \ref{lem:c-1} these inequalities show
  that the minimum of $c_s$ over $\frac{d}{2}\le s \le d-1$ is strictly positive and as a consequence
  proves Theorem \ref{simmons+} in the case of $d$ odd.

\begin{lem}\label{lem:c+1/2}
$c_{\frac d2+\frac12}>0$ for $d\in\R$ and $d\ge 3$.
\end{lem}

\begin{proof}
We claim that
\begin{equation}\label{eq:Todo1}
I_{\frac12+\frac1{2d}}\left(\frac d4+\frac 54,\frac d4-\frac14\right)+
I_{\frac12+\frac1{2d}}\left(\frac d4+\frac 14,\frac d4+\frac34\right)\geq1.
\end{equation}
Using \eqref{eq:Ix+-}
we rewrite the first summand in \eqref{eq:Todo1} as
\begin{equation}\label{eq:Insert1}
I_{\frac12+\frac1{2d}}\left(\frac d4+\frac 54,\frac d4-\frac14\right)
= I_{\frac12+\frac1{2d}}\left(\frac d4+\frac 14,\frac d4+\frac34\right)
- \frac{ \left(\frac12+\frac1{2d}\right)^{\frac14+\frac d4} \left(\frac12-\frac1{2d}\right)^{-\frac14+\frac d4}}{\left(\frac d4+\frac14\right) B\left(\frac d4+\frac 14,\frac d4+\frac 34\right)}.
\end{equation}
Upon multiplying \eqref{eq:Todo1} with $B\left(\frac d4+\frac 14,\frac d4+\frac 34\right)$ and using \eqref{eq:Insert1}, \eqref{eq:Todo1} rewrites to
\begin{equation}\label{eq:Todo2}
2 B_{\frac12+\frac1{2d}}\left(\frac d4+\frac 14,\frac d4+\frac34\right)
\geq B \left(\frac d4+\frac 14,\frac d4+\frac34\right) + 
\frac{4 \left(1+\frac1{d}\right)^{\frac14+\frac d4} \left(1-\frac1{d}\right)^{-\frac14+\frac d4}}{2^{\frac d2} (d+1)}
\end{equation}
The left-hand side of this inequality can be expanded as
\[
\begin{split}
2 B_{\frac12+\frac1{2d}}\left(\frac d4+\frac 14,\frac d4+\frac34\right)
& = 2 \int_0^{\frac12+\frac1{2d}} x^{\frac d4-\frac 34}(1-x)^{\frac d4-\frac 14} 
\, dx \\
&= 2 B_{\frac12} \left(\frac d4+\frac 14,\frac d4+\frac34\right)
+ 2
\int_{\frac12}^{\frac12+\frac1{2d}} x^{\frac d4-\frac 34}(1-x)^{\frac d4-\frac 14} 
\, dx. 
\end{split}
\]
Now \eqref{eq:Todo2} is equivalent to
\begin{multline}\label{eq:Todo3}
2 B_{\frac12} \left(\frac d4+\frac 14,\frac d4+\frac34\right)
+ 2
\int_{\frac12}^{\frac12+\frac1{2d}} x^{\frac d4-\frac 34}(1-x)^{\frac d4-\frac 14} 
\, dx \\
\geq 
B \left(\frac d4+\frac 14,\frac d4+\frac34\right) + \frac{4 \left(1+\frac1{d}\right)^{\frac14+\frac d4} \left(1-\frac1{d}\right)^{-\frac14+\frac d4}}{2^{\frac d2} (d+1)}.
\end{multline}

A brief calculation shows
\beq\label{eq:ToDo4}
\begin{split}
2 B_{\frac12} \left(\frac d4+\frac 14,\frac d4+\frac34\right)
- B \left(\frac d4+\frac 14,\frac d4+\frac34\right) 
& =\int_0^{1/2} x^{\frac d4-\frac 34}(1-x)^{\frac d4-\frac 34}
\left(\sqrt{1- x}-\sqrt{x}\right)\, dx.
\end{split}
\eeq

We next set out to provide a lower bound on \eqref{eq:ToDo4}.
First, rewrite
\beq\label{eq:fromtodo4}
\sqrt{1- x}-\sqrt{x} 
= \frac{1-2x}{\sqrt{1- x}+\sqrt{x}}
\eeq
and observe that
$\sqrt{1- x}+\sqrt{x}$ is increasing from $1$ to $\sqrt 2$
on $\left[0,\frac12\right]$. Hence \eqref{eq:fromtodo4} 
can be bound as
\beq\label{eq:fromtotodo4}
  \frac1{\sqrt2}(1-2x)\leq  \sqrt{1- x}-\sqrt{x}  \leq 1-2x.
\eeq
This gives
\beq\label{eq:todo5}
\begin{split}
\eqref{eq:ToDo4}
&\geq
\frac1{\sqrt2 }\int_0^{1/2} x^{\frac d4-\frac 34}(1-x)^{\frac d4-\frac 34}
\left(1-2x\right)\, dx \\
& =\frac1{ \sqrt2 }\int_0^{1/2} x^{\frac d4-\frac 34}(1-x)^{\frac d4-\frac 34}\, dx 
-\sqrt2\int_0^{1/2} x^{\frac d4+\frac 14}(1-x)^{\frac d4-\frac 34}
\, dx \\
& = \frac1{\sqrt2 } B_{\frac12} \left(\frac d4+\frac 14,\frac d4+\frac14\right)
- \sqrt2 B_{\frac12} \left(\frac d4+\frac 54,\frac d4+\frac14\right).
\end{split}
\eeq
Using the well-known formula
\[
  B(z,a+1,b)= \frac{aB(z,a,b)- z^a(1-z)^b}{a+b}
\]
on $B_{\frac12} \left(\frac d4+\frac 54,\frac d4+\frac14\right)$,
the final expression in \eqref{eq:todo5} simplifies into
\beq\label{eq:todoLHS}
\frac{2}{2^{\frac{d}{2}}(d+1)}.
\eeq

\begin{sublemma}\label{sublem:Todo3}
For $3\le d\in \R$, 
the integrand $\eta(x)$ in the second summand in \eqref{eq:Todo3} is decreasing and concave on $(\frac12,\frac12+\frac1{2d})$.
\end{sublemma}

\begin{proof}
This is routine. The derivative of $\eta$ is 
\[
\eta'(x)=\frac{1}{4} (1-x)^{\frac{d-5}{4}}
   x^{\frac{d-7}{4}} (-2 d x+4 x+d-3)
   \]
So its sign on   $(\frac12,\frac12+\frac1{2d})$ is
governed by that of $-2 d x+4 x+d-3$. However,
\[
-2 d x+4 x+d-3 \leq 2d\cdot \frac12+4 \left(\frac12+\frac1{2d}\right) + d -3 = -1+\frac 2d
\]
is negative for $d\geq3$. Thus $\eta$ is decreasing.

Further,
\[
\eta''(x)=
\frac{1}{16} (1-x)^{\frac{d-9}{4}}
   x^{\frac{d-11}{4}} \left(4 (d-4) (d-2)
   x^2-4 (d-4) (d-3) x+d^2-10 d+21\right).
   \]
 For $\ds\frac12\leq x\leq \frac12+\frac1{2d}$
   we have,
\[
 \begin{split}
   4 (d-4) (d-2)&
   x^2-4 (d-4) (d-3) x+d^2-10 d+21
  \\ \leq &
4 (d-4) (d-2)
   \left(\frac12+\frac1{2d} \right)^2-4 (d-4) (d-3) \cdot\frac12+d^2-10 d+21    \\
   = & \frac{8}{d^2}+\frac{10}{d}-6
\end{split}
\]
   and the last expression is negative for $d\geq3$,
   whence $\eta''(x)<0$.
\end{proof}

By Sublemma \ref{sublem:Todo3}, the
integral in \eqref{eq:Todo3} can be bound below by
\[
\frac 12 \frac1{2d} 
\left(
\left(\frac12+\frac1{2d}\right)^{\frac d4-\frac 34}\left(\frac12-\frac1{2d}\right)^{\frac d4-\frac 14} 
+
2^{1-\frac d2}
\right).
\]
Moving 
this to the right-hand side of \eqref{eq:Todo3} means we have to show
that \eqref{eq:ToDo4} is at least
\beq\label{eq:todoRHS}
\frac1{2^{\frac d2}\,d}
\left(
3 \left(1+\frac1d\right)^{\frac{d-3}4}
\left(1-\frac1d\right)^{\frac{d-1}4}
-1
\right).
\eeq
It suffices to replace \eqref{eq:ToDo4} with its lower bound
\eqref{eq:todoLHS}. That is, we shall prove
\beq\label{eq:Todo6}
\frac{3d+1}{d+1} \geq 3 \left(1+\frac1d\right)^{\frac{d-3}4}
\left(1-\frac1d\right)^{\frac{d-1}4}.
\eeq

Rearranging the right-hand side of \eqref{eq:Todo6} we get
\[
\begin{split}
3 \left(1+\frac1d\right)^{\frac{d-3}4}
\left(1-\frac1d\right)^{\frac{d-1}4}
&= 
3 \left(1+\frac1d\right)^{\frac{d+1}4}
\left(1-\frac1d\right)^{\frac{d-1}4}
\left(1+\frac1d\right)^{-1} \\
& = 
3 \left(1+\frac1d\right)^{\frac{d+1}4}
\left(1-\frac1d\right)^{\frac{d-1}4}
\frac d{d+1}.
\end{split}
\]
In particular, \eqref{eq:Todo6} is equivalent to
\beq\label{eq:Todo7}
3d+1\geq 3 d \left(1+\frac1d\right)^{\frac{d+1}4}
\left(1-\frac1d\right)^{\frac{d-1}4}.
\eeq
As before, the sequence $\ds\left(1+\frac1d\right)^{\frac{d+1}4}$
is increasing with limit $\ds e^{\frac14}$, so
\beq\label{eq:Todo8}
\begin{split}
3 d \ \left(1+\frac1d\right)^{\frac{d+1}4}
\left(1-\frac1d\right)^{\frac{d-1}4}
&\leq
3 d e^{\frac14}
\left(1-\frac1d\right)^{\frac{d-1}4}\\
&= 
3 d e^{\frac14}
\left(1-\frac1d\right)^{\frac{d}4}
\left(1-\frac1d\right)^{-\frac14}
\end{split}
\eeq
The sequence $\ds\left(1-\frac1d\right)^{\frac{d}4}$
is increasing with limit $\ds e^{-\frac14}$, so the right-hand side
of \eqref{eq:Todo8} is further at most
\[
3d \left(1-\frac1d\right)^{-\frac14}.
\]

Now \eqref{eq:Todo7} is implied by
\[
1+\frac1{3d}\geq \left(1-\frac1d\right)^{-\frac14},
\]
an inequality easy to establish using calculus.
\end{proof}

\begin{lem}\label{lem:c+3/2}
$c_{\frac d2+\frac32}>0$  for $3\le d\in\mathbb N$.
In addition, if $d\in\R$ and $d\geq 16$, then $c_{\frac d2+\frac32}>0$.
\end{lem}

\begin{proof}
We claim that
\begin{equation}\label{eq:TTodo1}
I_{\frac12+\frac3{2d}}\left(\frac d4+\frac 74,\frac d4-\frac34\right)+
I_{\frac12+\frac3{2d}}\left(\frac d4+\frac 34,\frac d4+\frac14\right)\geq1.
\end{equation}
Using \eqref{eq:Ix+-}
we rewrite the first summand in \eqref{eq:TTodo1} as
\begin{equation}\label{eq:IInsert1}
I_{\frac12+\frac3{2d}}\left(\frac d4+\frac 74,\frac d4-\frac34\right)
=
I_{\frac12+\frac3{2d}}\left(\frac d4+\frac 34,\frac d4+\frac14\right)
- \frac{ \left(\frac12+\frac3{2d}\right)^{\frac34+\frac d4} \left(\frac12-\frac3{2d}\right)^{-\frac34+\frac d4}}{\left(\frac d4+\frac34\right) B\left(\frac d4+\frac 34,\frac d4+\frac 14\right)}.
\end{equation}
Upon multiplying \eqref{eq:TTodo1} with $B\left(\frac d4+\frac 34,\frac d4+\frac 14\right)$ and using \eqref{eq:IInsert1}, \eqref{eq:TTodo1} rewrites to
\begin{equation}\label{eq:TTodo2}
2 B_{\frac12+\frac3{2d}}\left(\frac d4+\frac 34,\frac d4+\frac14\right)
\geq B \left(\frac d4+\frac 34,\frac d4+\frac14\right) + 
\frac{4 \left(1+\frac3{d}\right)^{\frac34+\frac d4} \left(1-\frac3{d}\right)^{-\frac34+\frac d4}}{2^{\frac d2} (d+3)}
\end{equation}

Further, using
\beq\label{eq:Bx+.}
  B_z(a+1,b)= \frac{aB_z(a,b)- z^a(1-z)^b}{a+b}
\eeq
on the two betas in \eqref{eq:TTodo2}, we get
\begin{multline*}
\frac{d-1}{d} B_{\frac12+\frac3{2d}} \left(\frac d4-\frac 14,\frac d4+\frac14\right)
- \frac{2 \left(1+\frac3{d}\right)^{-\frac14+\frac d4} \left(1-\frac3{d}\right)^{+\frac14+\frac d4}}{2^{\frac d2} \ d}
\\
\geq 
\frac{d-1}{2d} B \left(\frac d4-\frac 14,\frac d4+\frac14\right) + 
\frac{4 \left(1+\frac3{d}\right)^{\frac34+\frac d4} \left(1-\frac3{d}\right)^{-\frac34+\frac d4}}{2^{\frac d2} (d+3)},
\end{multline*}
or equivalently,
\beq\label{eq:TTodo3}
2 B_{\frac12+\frac3{2d}} \left(\frac d4-\frac 14,\frac d4+\frac14\right) -
B \left(\frac d4-\frac 14,\frac d4+\frac14\right) 
\geq
\frac{12 
   \left(1-\frac{3}{d}\right)^{\frac{d-3}{4}} 
   \left(1+\frac{3}{d}\right)^{\frac{d+3}{4}}}{ 2^{\frac d2} (d+3)}.
\eeq

The first summand on the left-hand side of this inequality can be expanded as
\[
2 B_{\frac12+\frac3{2d}}\left(\frac d4-\frac 14,\frac d4+\frac14\right)
 = 2 B_{\frac12} \left(\frac d4-\frac 14,\frac d4+\frac14\right)
+ 2
\int_{\frac12}^{\frac12+\frac3{2d}} x^{\frac d4-\frac 54}(1-x)^{\frac d4-\frac 34} 
\, dx. 
\]
As in Lemma \ref{lem:c+1/2},
\beq\label{eq:TToDo4}
2 B_{\frac12} \left(\frac d4-\frac 14,\frac d4+\frac14\right)
- B \left(\frac d4-\frac 14,\frac d4+\frac14\right) 
 =\int_0^{1/2} x^{\frac d4-\frac 54}(1-x)^{\frac d4-\frac 54}
\left(\sqrt{ 1-x}-\sqrt{x}\right)\, dx
\eeq
can be bound below by 
\beq\label{eq:TTodo5}
\frac4{2^{\frac d2}(d-1)}.
\eeq
Similarly, $\ds x^{\frac d4-\frac 54}(1-x)^{\frac d4-\frac 34} $ is decreasing and concave on $\left(\frac12,\frac12+\frac3{2d}\right)$
for $d\geq 5$, so
\beq\label{eq:TTodo6}
2 \int_{\frac12}^{\frac12+\frac3{2d}} x^{\frac d4-\frac 54}(1-x)^{\frac d4-\frac 34} 
\, dx \geq 
\frac{6 
   \left(\left(1-\frac{3}{d}\right)^{\frac{d-3}{4}
   }
   \left(1+\frac{3}{d}\right)^{\frac{d-5}{4}}+1\right)}{2^{\frac d2} d}
   .
   \eeq
   
   Using the two lower bounds \eqref{eq:TTodo5} and \eqref{eq:TTodo6} in \eqref{eq:TTodo3}, it suffices to establish
\beq\label{eq:TTodo7}
\begin{split}
   \frac4{d-1}
   & \geq
   \frac{12 
   \left(1-\frac{3}{d}\right)^{\frac{d-3}{4}} 
   \left(1+\frac{3}{d}\right)^{\frac{d+3}{4}}}{ d+3}-
   \frac{6 
   \left(\left(1-\frac{3}{d}\right)^{\frac{d-3}{4}
   }
   \left(1+\frac{3}{d}\right)^{\frac{d-5}{4}}+1\right)}{ d}    \\
   & =
   6  \left(1-\frac{3}{d}\right)^{\frac{d-3}{4}} 
   \left(1+\frac{3}{d}\right)^{\frac{d-5}{4}} 
\left(
2 \frac{ d+3}{d^2}   - \frac1d\right)
-\frac6d \\
&=    6  \left(1-\frac{3}{d}\right)^{\frac{d-3}{4}} 
   \left(1+\frac{3}{d}\right)^{\frac{d-5}{4}} 
\frac{d+6}{d^2}
-\frac6d 
   \end{split}
   \eeq
The sequences
\[
\left(1-\frac{3}{d}\right)^{\frac{d+1}{4}} ,
\qquad
   \left(1+\frac{3}{d}\right)^{\frac{d-5}{4}} 
   \]
   are increasing, the product of their limits is $1$. 
The inequality \eqref{eq:TTodo7} is easy to verify 
(by hand or using a computer algebra system)
for $d=1,2,\ldots, 16$.
Now assume $d\in\R$ with $d\geq 16$. It is enough
   to prove
   \beq\label{eq:TTodo8}
   \frac4{d-1}+\frac6d \geq 6\left(1-\frac{3}{d}\right)^{-1}\frac{d+6}{d^2}
= 6
\frac    {d+6}{d(d-3)}
   .
   \eeq
Equivalently,
\[
\frac{2 \left(2 d^2-33
   d+27\right)}{(d-3) (d-1) d} \geq0.
  \]
  But this holds for all $d\geq \frac{3}{4}
   \left(11+\sqrt{97}\right) \approx 15.6366.
   $
\end{proof}

 The proof of Theorem \ref{simmons+} is now complete.

\section{Bounds on the Median and the Equipoint of the Beta Distribution}\label{sec:lowBdSIS}
Like the median, the equipoint is a measure of central tendency in a probability distribution function  (PDF). 
 In this section we establish, for the Beta distribution, new 
 lower bound for the median and, by relating the equipoint to the median, bounds on the equipoint needed in the
 proof of Theorem \ref{thm:thetaExplicit}.\index{median}\index{equipoint}\index{beta distribution}

\def\pdf{\varrho}

As in Section \ref{sec:introprob} 
we follow the convention that $\salt,\talt\in\R_{>0},$ 
and consider the Beta distribution $\text{Beta}(\salt,\talt)$ supported on $[0,1]$. 
We denote by $\pdf_{\salt,\talt}\colon[0,1]\to\R$ the
probability density function of $\text{Beta}(\salt,\talt)$, i.e.,
\[\pdf_{\salt,\talt}(x)=\frac{x^{\salt-1}(1-x)^{\talt-1}}{B(\salt,\talt)}\]
for $x\in[0,1]$. 
The cumulative distribution function of 
$\text{Beta}(\salt,\talt)$ is 
$ I_x(\salt,\talt)$
 defined for 
$x \in [0,1]$.
We 
are interested in the \df{median} $m_{\salt,\talt}\in[0,1]$ 
of $\text{Beta}(\salt,\talt)$ and in the 
$(\salt,\talt)$-\df{equipoint} $e_{\salt,\talt}\in[0,1]$ defined by
\begin{equation}\label{eq:defm}
I_{m_{\salt,\talt}}(\salt,\talt)=\frac12
\end{equation}
and
\begin{equation}\label{eq:defe}
I_{e_{\salt,\talt}}(\salt,\talt+1)+I_{e_{\salt,\talt}}(\salt+1,\talt)=1
\end{equation}
respectively. 
Here we used that  
$I_x(\salt,\talt)$
 and $ I_x(\salt,\talt+1)+I_x(\salt+1,\talt)$ 
 are strictly monotonically increasing for  $x \in [0,1]$. 
We will continue to use this tacitly
throughout this section.

\subsection{Lower bound for the equipoint $\eiha$}

In \eqref{eq:defe}, if we move one of the two terms to the other side, we get the equivalent forms
\[I_{e_{\salt,\talt}}(\salt,\talt+1)=I_{1-e_{\salt,\talt}}(\talt,\salt+1), \ \ \
I_{e_{\salt,\talt}}(\salt+1,\talt)=I_{1-e_{\salt,\talt}}(\talt+1,\salt).\]

\begin{lemma}\label{reform}
For all $\salt,\talt\in\R_{>0}$ and $x\in[0,1]$, we have
\begin{enumerate}[\rm(a)]
\item \label{it:reforma} $\ds I_x(\salt,\talt+1)+I_x(\salt+1,\talt)=2I_x(\salt,\talt)+(\salt-\talt)\frac{x^\salt(1-x)^\talt}{\salt\talt B(\salt,\talt)}$
\item \label{it:reformb} $\ds I_x(\salt,\talt+1)+I_x(\salt+1,\talt)=2I_x(\salt+1,\talt+1)+(1-2x)\frac{(\salt+\talt)x^\salt(1-x)^\talt}{\salt\talt B(\salt,\talt)}$
\end{enumerate}
\end{lemma}

\begin{proof}
Use the identities (8.17.20) and (8.17.21) from \url{http://dlmf.nist.gov/8.17#iv}.
\end{proof}

Although there are some results on the median $m_{\salt,\talt}$ 
for special values of $\salt,\talt\in\R_{\geq0}$\footnote{see \url{http://en.wikipedia.org/wiki/Beta_distribution}}, 
about the only general 
thing that seems to be known \cite{pyy}   is that
\begin{equation}\label{betamodemean}
\mu_{\salt,\talt}:=\frac{\salt}{\salt+\talt}<m_{\salt,\talt}<\frac{\salt-1}{\salt+\talt-2}
\end{equation}
if $1<\talt<\salt$  \index{$\mu_{\salt,\talt}:=\frac{\salt}{\salt+\talt}$}
(see also \cite{Ker} for an asymptotic analysis and 
numerical evidence in support of better bounds).  
The lower bound in \eqref{betamodemean}  is actually the mean $\mu_{\salt,\talt}$ of $\text{Beta}(\salt,\talt)$ if $\salt,\talt>0$ 
and the upper bound is the mode of $\text{Beta}(\salt,\talt)$ if $\salt,\talt>1$.
In the next subsection we shall significantly improve the upper bound in
\eqref{betamodemean}.

Using Lemma \ref{reform}\eqref{it:reforma}, we see that
\[I_{m_{\salt,\talt}}(\salt,\talt+1)+I_{m_{\salt,\talt}}(\salt+1,\talt)=2\frac12+(\salt-\talt)\frac{m_{\salt,\talt}^\salt(1-m_{\salt,\talt})^\talt}{\salt\talt B(\salt,\talt)}\ge1\]
and therefore
\begin{equation}\label{ineq:em}
e_{\salt,\talt}\le m_{\salt,\talt}
\end{equation}
whenever $\salt, \talt \in\mathbb R$ and $0<\talt\le\salt$. Using Lemma \ref{reform}\eqref{it:reformb}, we get
\begin{multline*}
I_{m_{\salt+1,\talt+1}}(\salt,\talt+1)+I_{m_{\salt+1,\talt+1}}(\salt+1,\talt)=\\
2\frac12+(1-2m_{\salt+1,\talt+1})\frac{(\salt+\talt)(m_{\salt+1,\talt+1})^\salt(1-m_{\salt+1,\talt+1})^\talt}{\salt\talt B(\salt,\talt)}\le1.
\end{multline*}
since $\ds m_{\salt+1,\talt+1}\ge\frac{\salt+1}{\salt+\talt+2}\ge\frac{\frac\salt2+\frac\talt2+1}{\salt+\talt+2}=\frac12$ by \eqref{betamodemean}. This shows that
\begin{equation}\label{ineq:me}
m_{\salt+1,\talt+1}\le e_{\salt,\talt}
\end{equation}
whenever $\salt, \talt\in\mathbb R$ and  $0<\talt<\salt$.

These inequalities combine to give:

\begin{prop}
 For $ \head, \tail \in \RR_{>0}$,
\label{prop:preLowBdSIS}
\beq \label{eq:preLowBdSIS}
e_{\salt,\talt}\le m_{\salt,\talt}<\frac{\salt-1}{\salt+\talt-2}
\eeq
when $1<\talt<\salt$ and
\[\frac{\salt+1}{\salt+\talt+2}<m_{\salt+1,\talt+1}\le e_{\salt,\talt}\]
when $0<\talt<\salt$.
Later this  lower bound on $e_{\salt,\talt}$
proves important to us.
\end{prop}

\begin{proof}
 The first line of inequalities  \eqref{eq:preLowBdSIS}
 follows from 
 \eqref{betamodemean} and \eqref{ineq:em}.
 The second from   \eqref{betamodemean} and \eqref{ineq:me}.
\end{proof}

\begin{rem}\rm
  The inequality \eqref{eq:preLowBdSIS} is easier to prove than the inequality
  $e_{\salt,\talt} \le \frac{\salt}{\salt+\talt}$ from Theorem \ref{simmons+}; 
  however, this weaker inequality seems not to be strong enough to 
  prove Theorem \ref{thm:thetaExplicit}.  
\end{rem}

\subsection{New bounds on the median of the beta distribution}
\label{subsec:lowBdSIS2}

Having a lower bound for the equipoint $\eiha$ in terms of the median $m_{\salt +1,\talt +1}$, 
we now turn our attention  to the median of the beta distribution.\index{median}\index{beta distribution}

By Lemma \ref{reform}(a), 
Simmons' inequality
\eqref{ineq:simm} is equivalent to
\begin{equation}\label{ineq:simmc}
2I_{\mu_{\salt,\talt}}(\salt,\talt)+(\salt-\talt)\frac{(\mu_{\salt,\talt})^\salt(1-\mu_{\salt,\talt})^\talt}{\salt\talt B(\salt,\talt)}\ge1
\end{equation}
which is therefore conjectured for all $\salt,\talt\in\R$ with $0<\salt\le\talt$. Proposition \ref{prop:extra2} below
proves a weakening of \eqref{ineq:simmc} where an extra factor of $2$ is introduced in the second term on the left hand side.

\begin{lemma}\label{absdev}
Suppose $\salt,\talt\in\R_{>0}$ and set $\mu:=\mu_{\salt,\talt}$. Then
\begin{multline}\label{eq:notag}
\int_0^\mu(\mu-x)^{\salt-1}(1-\mu+x)^{\talt-1}x~dx=\int_0^{1-\mu}(\mu+x)^{\salt-1}(1-\mu-x)^{\talt-1}x~dx\\
=\frac{\mu^\salt(1-\mu)^\talt}{\salt+\talt}=\frac{\salt^\salt\talt^\talt}{(\salt+\talt)^{\salt+\talt+1}}.
\end{multline}
\end{lemma}

\begin{proof}
Reversing the direction of integration in the first integral and changing the domain of integration in the second integral, we get
\begin{align}
\int_0^{\mu}(\mu-x)^{\salt-1}(1-\mu+x)^{\talt-1}x~dx&=\int_0^\mu x^{\salt-1}(1-x)^{\talt-1}(\mu-x)dx,
\label{aio1}
\\
\int_0^{1-\mu}(\mu+x)^{\salt-1}(1-\mu-x)^{\talt-1}x~dx&=\int_\mu^{1}x^{\salt-1}(1-x)^{\talt-1}(x-\mu)dx.
\label{aio2}
\end{align}
If we subtract \eqref{aio1} from \eqref{aio2} and divide by $B(\salt,\talt)$, we get
\[\int_0^1\pdf_{\salt,\talt}(x)(x-\mu)dx=\mu-\mu=0\]
by the definition of the mean $\mu$. So the first equality is proved. On the other hand, if we add \eqref{aio1} and \eqref{aio2}
and divide again by $B(\salt,\talt)$, we get
\[\int_0^1\pdf_{\salt,\talt}(x)|x-\mu|dx\]
which is by the formula for the \emph{mean absolute deviation} of $\text{Beta}(\salt,\talt)$ (cf.~the proof of 
\cite[Corollary 1]{DZ91}) equal to
\[\frac{2\mu(1-\mu)}{\salt+\talt}\pdf_{\salt,\talt}(\mu),
\]
thus showing the second equation. The third equation in 
\eqref{eq:notag} is clear.
\end{proof}

\begin{prop}\label{prop:extra2}
Suppose $1\le\talt\le\salt$ such that $\salt+\talt\ge3$ and set $\mu:=\mu_{\salt,\talt}$. Then we have
\[2I_\mu(\salt,\talt)+2(\salt-\talt)\frac{\mu^\salt(1-\mu)^\talt}{\salt\talt B(\salt,\talt)}\ge1
\]
\end{prop}

\def\ee{\chi}

\begin{proof}
We have to show
\[I_\mu(\salt,\talt)+2(\salt-\talt)\frac{\mu^\salt(1-\mu)^\talt}{\salt\talt B(\salt,\talt)}\ge I_{1-\mu}(\talt,\salt)
\]
which is equivalent to
\[B_\mu(\salt,\talt)+2\ee \ge B_{1-\mu}(\talt,\salt)\]
where
\[\ee:=\frac{\salt-\talt}{\salt\talt}\mu^\salt(1-\mu)^\talt.
\]
This means
\[2\ee+\int_0^\mu x^{\salt-1}(1-x)^{\talt-1}dx\ge\int_0^{1-\mu}x^{\talt-1}(1-x)^{\salt-1}dx
\]
which we rewrite as
\[\ee+\int_0^\mu(\mu-x)^{\salt-1}(1-\mu+x)^{\talt-1}dx\ge-\ee+\int_0^{1-\mu}(\mu+x)^{\salt-1}(1-\mu-x)^{\talt-1}dx.
\]
We have $\frac12\le\mu\le1$ and therefore $0\le1-\mu\le\frac12\le\mu\le1$. In particular, the domain of integration is
smaller on the left hand side. The idea is to compare the two terms under the integral pointwise on $[0,1-\mu]$ after
correcting these two terms using $\ee$ and $-\ee$, respectively. The two terms agree when substituting $x=0$. The derivative
at $x=0$ of the term under the integral on the left hand side is by the product rule the negative term
\[
\mu^{\salt -2} (1-\mu )^{\talt -2} (\mu  (\salt +\talt -2)-\salt +1)=\mu^{\salt -2} (1-\mu )^{\talt -2}\frac{\talt-\salt}{\salt+\talt}
\]
and on the right hand side it is the additive inverse. We want to counterbalance the derivatives at $x=0$ by adding and subtracting
a multiple of the term from Lemma \ref{absdev} on the left and right hand side, respectively. The derivative of that latter
term at $x=0$ is of course
\[\mu^{\salt -1} (1-\mu )^{\talt -1}=\mu^{\salt -2} (1-\mu )^{\talt -2}\frac{\salt\talt}{(\salt+\talt)^2}.\]
We thus would like to add
\[c:=\frac{(\salt-\talt)(\salt+\talt)}{\salt\talt}=\frac{\salt^2-\talt^2}{\salt\talt}\]
times the term from Lemma \ref{absdev} on the left hand side and subtract it on the right hand side.
The miracle now is that this is exactly $\chi$.
Our claim can thus be rewritten as
\[\int_0^\mu(\mu-x)^{\salt-1}(1-\mu+x)^{\talt-1}(1+cx)dx\ge\int_0^{1-\mu}(\mu+x)^{\salt-1}(1-\mu-x)^{\talt-1}(1-cx)dx.
\]
The two terms under the integral now take the same value at $x=0$ and have the same derivative there. There is now a hope
to show for $x\in[0,1-\mu]$ that the term on the left hand side is pointwise less than or equal the term on the right hand side.
We will do this and thus even show the stronger claim that
\[\int_0^{1-\mu}(\mu-x)^{\salt-1}(1-\mu+x)^{\talt-1}(1+cx)dx\ge\int_0^{1-\mu}(\mu+x)^{\salt-1}(1-\mu-x)^{\talt-1}(1-cx)dx.
\]
If we define (noting that $1-cx\ge1-c(1-\mu)=1-c\frac\talt{\salt+\talt}=1-\frac{\salt-\talt}\salt=\frac\talt\salt>0$)
\[g\colon[0,1-\mu)\to\R,\ x\mapsto\left(\frac{\mu-x}{\mu+x}\right)^{\salt-1}\left(\frac{1-\mu+x}{1-\mu-x}\right)^{\talt-1}
\frac{1+cx}{1-cx},
\]
it is thus enough to show that $g(x)\ge1$ for all $x\in[0,1-\mu]$. Clearly we have $g(0)=1$. So it is enough to show that
$g'(x)\ge0$ for all $x\in[0,1-\mu)$. A straightforward calculation shows
\[
g'(x)=\frac{2 \left(\frac{\mu-x}{\mu +x}\right)^{\salt } \left(\frac{1-\mu+x}{1-\mu-x}\right)^{\talt }}{(1-c x)^2 (1-\mu +x)^2 (\mu -x)^2}+
x^2h(x^2)
\]
where
\[h\colon\begin{cases}
[0,(1-\mu)^2]\to\R\\
y\mapsto
\frac{(\salt-\talt) (\salt+\talt-3)}{\salt+\talt}-\frac{(\salt+\talt) \left(\salt^4-\salt^3 (2\talt+1)+2\salt^2\talt+2\salt(\talt-1)
   \talt^2-(\talt-1)\talt^3\right)}{\salt^2\talt^2}y.
   \end{cases}
\]
Since $h$ is linear, it is thus enough to show that $h(0)\ge0$ and $h((1-\mu)^2)\ge0$.
The first condition follows from the hypothesis $\salt+\talt\ge3$.
Another straightforward calculation shows
\[h((1-\mu)^2)=\frac{(\talt-1) \left(2\salt^2-3\salt\talt+\talt^2\right)}{\salt^2}.
\]
Because of $\talt\ge1$ it remains only to show that $2\salt^2-3\salt\talt+\talt^2\ge0$. Now we have
\[2\salt^2-3\salt\talt+\talt^2=(\salt-\talt)(2\salt-\talt)\ge0
\]
since $\salt\ge\talt$.
\end{proof}

The following corollary improves the previously known upper bound \eqref{betamodemean} on the median
$m_{\salt,\talt}$ in the case where $1<\talt\le\salt$ and $\salt+\talt\ge3$ because of the following lemma.

\begin{lemma}\label{lem:impr}
Suppose $\salt,\talt\in\R$ such that $\salt\ge\talt\ge1$ and $\salt+\talt>2$. Then
\[\frac\salt{\salt+\talt}+\frac{\salt-\talt}{(\salt+\talt)^2}\le\frac{\salt-1}{\salt+\talt-2}.\]
\end{lemma}

\begin{proof}
A straightforward calculation yields
\[\frac{\salt-1}{\salt+\talt-2}-\frac\salt{\salt+\talt}-\frac{\salt-\talt}{(\salt+\talt)^2}=\frac{2 (\salt-\talt)}{(\salt+\talt-2) (\salt+\talt)^2}\ge0.\qedhere
\]
\end{proof}

\begin{cor}\label{cor:newmed}
Suppose $1\le\talt\le\salt$ such that $\salt+\talt\ge3$. Then we have
\[\mu_{\salt,\talt}=\frac\salt{\salt+\talt}\le m_{\salt,\talt}\le\mu_{\salt,\talt}+\frac{\salt-\talt}{(\salt+\talt)^2}.
\]
\end{cor}

\begin{proof}
The first inequality comes from \cite{pyy}. To prove the second, we have to show that
\[
2I_\mu(\salt,\talt)+2\int_{\mu}^{\mu+\frac{\salt-\talt}{(\salt+\talt)^2}}\pdf_{\salt,\talt}(x)dx\ge1
\]
where $\mu:=\mu_{\salt,\talt}$. By Proposition \ref{prop:extra2}, it is henceforth enough to show that
\[
\int_{\mu}^{\mu+\frac{\salt-\talt}{(\salt+\talt)^2}}\pdf_{\salt,\talt}(x)dx\ge(\salt-\talt)\frac{\mu^\salt(1-\mu)^\talt}{\salt\talt B(\salt,\talt)}.
\]
This is trivial if $\salt=\talt$. If $1<\talt$ (and therefore $1<\salt$), then 
$\frac{\salt-1}{\salt+\talt-2}$ is the mode of $\text{Beta}(\salt,\talt)$ and by Lemma \ref{lem:impr} we have
\[\pdf_{\salt,\talt}(x)\ge \pdf_{\salt,\talt}(\mu)\]
for all $x\in[\mu,\mu+\frac{\salt-\talt}{(\salt+\talt)^2}]$. Therefore it is enough to show that
\[
\frac{\salt-\talt}{(\salt+\talt)^2}\pdf_{\salt,\talt}(\mu)\ge(\salt-\talt)\frac{\mu^\salt(1-\mu)^\talt}{\salt\talt B(\salt,\talt)}
\]
but this holds even with equality since
\[
\frac1{(\salt+\talt)^2}=\frac{\mu(1-\mu)}{\salt\talt}.\qedhere
\]
\end{proof}

The following table illustrates the quality of the lower bound $\mu_{\salt,\talt}$ on the median $m_{\salt,\talt}$ (for $1\le\talt\le\salt$)
and the quality of the new upper bound $\mu_{\salt,\talt}+\frac{\salt-\talt}{(\salt+\talt)^2}$ (for $1\le\talt\le\salt$ with $\salt+\talt\ge3$)
as compared to the less tight old upper bound $\frac{\salt-1}{\salt+\talt-2}$. If one assumes that \eqref{ineq:simmc} is true for
all real $\salt,\talt$ (as opposed to 
$\salt,\talt\in\frac12\N$ as given by Theorem \ref{simmons+})
with $1\le\talt\le\salt$ with $\salt+\talt\ge3$, then
one can deduce along the lines of Corollary \ref{cor:newmed}  an even better upper bound on $m_{\salt,\talt}$
for $1\le\talt\le\salt$ with $\salt+\talt\ge3$, namely $\mu_{\salt,\talt}+\frac{\salt-\talt}{2(\salt+\talt)^2}$ which we therefore also include in
the table.

\[
\begin{array}{cc|ccccc}
\salt&\talt&\mu_{\salt,\talt}&m_{\salt,\talt}&\mu_{\salt,\talt}+\frac{\salt-\talt}{2(\salt+\talt)^2}&\mu_{\salt,\talt}+\frac{\salt-\talt}{(\salt+\talt)^2}&\frac{\salt-1}{\salt+\talt-2}\\
\hline
 2.5 & 1 & 0.714286 & 0.757858 & 0.77551 & 0.836735 & 1 \\
 3 & 1 & 0.75 & 0.793701 & 0.8125 & 0.875 & 1 \\
 3 & 2 & 0.6 & 0.614272 & 0.62 & 0.64 & 0.666667 \\
 4 & 2 & 0.666667 & 0.68619 & 0.694444 & 0.722222 & 0.75 \\
 10 & 3 & 0.769231 & 0.783314 & 0.789941 & 0.810651 & 0.818182 \\
 10 & 7 & 0.588235 & 0.591773 & 0.593426 & 0.598616 & 0.6 \\
\end{array}
\]

\section{Proof of Theorem {\rm\ref{thm:thetaExplicit}}}\label{sec:nov}

In this section we prove Theorem \ref{thm:thetaExplicit} by establishing Proposition \ref{prop:thetafirst}.
We start by tweaking Problem \ref{prob:opt}:

\begin{problem}\label{prob:opt'}\rm
   Given a positive integer $d$, minimize \[
   \ds f_{s,t}(\sigma)=\frac{2(1-\sigma)sI_{1-\sigma}\left(\frac t2,1+\frac s2\right)+2\sigma tI_{\sigma}\left(\frac s2,1+\frac t2\right)}{(1-\sigma)s+\sigma t}-1\]
 subject to the constraints 
\begin{enumerate}[(i)]
 \item $s,t\in\mathbb N$ and $s+t=d$;
 \item $\ds s\ge \frac{d}{2};$ 
 \item $\ds 0\leq\sigma\leq\frac sd$; and 
 \item \label{it:determinesp'}  $\ds I_{\sigma}\left(\frac s2, \frac t2 +1\right) = I_{1-\sigma}\left(\frac t2,\frac s2 +1\right).$ 
\end{enumerate}
\end{problem} 

Problem \ref{prob:opt'} is equivalent to Problem \ref{prob:opt}.
Indeed, the only difference is the interval for $\sigma$ in (iii).
However, by Section \ref{sec:simmhalf} we know that
$\sigma_{s,t}=\sigma\in[0,1]$, the solution to  (iv) will automatically 
satisfy 
$\ds \sigma_{s,t} \leq \frac sd$.

\subsection{An auxiliary function}
For $s,t\in\R_{>0}$ let 
\beq\label{eq:defgst}
g_{s,t}(\sigma):= -1+ I_{\sigma}\left(\frac s2, \frac t2 +1\right) +
 I_{1-\sigma}\left(\frac t2,\frac s2 +1\right).
\eeq

\begin{lem} 
For $s,t\in\R_{>0}$, we have
\label{lem:fisg}
 $$
 f_{s,t}(\sist) = g_{s,t}(\sist)=
2 \;
I_{\sist}\left(\frac s2,1+\frac t2\right) 
 -1 .
$$
Thus at
the minimizer of $f_{s,t}$, the functions 
$f_{s,t}$ and $g_{s,t}$ have the same value.
\end{lem}

\begin{proof}
This is straightforward since $f_{s,t}$ assumes its minimum where the two incomplete beta expressions (appearing in both $f_{s,t}$ and $g_{s,t}$) are equal.
\end{proof}

\begin{lem}\label{lem:gst}
The function $g_{s,t}$ can be rewritten as
\beq\label{eq:gst}
g_{s,t}(\sigma)=
\sigma ^{s/2} (1-\sigma )^{t/2} \frac{\Gamma
   \left(\frac{s}{2}+\frac{t}{2}+1\right)}{\Gamma
   \left(\frac{s}{2}+1\right) \Gamma
   \left(\frac{t}{2}+1\right)}.
  \eeq
\end{lem}

\begin{proof}
First recall that $I_{1-x}(b,a)=1-I_x(a,b)$ and apply this
to the second incomplete beta summand in the definition
of $g_{s,t}$:
\[
g_{s,t}(\sigma)=
I_{\sigma}\left(\frac s2, \frac t2 +1\right) -
 I_{\sigma}\left(\frac s2+1,\frac t2 \right).
 \]
Now use recursive formulas for $I_\sigma$\footnote{Equations (8.17.20) and (8.17.21) in \url{http://dlmf.nist.gov/8.17}.} and simplify.
\end{proof}

\begin{lem}\label{lem:qIncr}
The function $g_{s,t}$ is monotonically increasing on 
$\ds \left[0,\frac s{s+t}\right]$ whenever $s,t\in\R_{>0}$ 
with $s\geq t$.
\end{lem}

\begin{proof}
Using Lemma \ref{lem:gst} it is easy to see that
\[
g'_{s,t}(\sigma)=
-\frac{1}{2} \sigma ^{\frac{s}{2}-1} (1-\sigma
   )^{\frac{t}{2}-1} \big(s (\sigma -1)+\sigma  t\big)
   \frac{\Gamma
   \left(\frac{s}{2}+\frac{t}{2}+1\right)}{\Gamma
   \left(\frac{s}{2}+1\right) \Gamma
   \left(\frac{t}{2}+1\right)}. \qedhere
   \]
\end{proof}

We shall exploit bounds on $\sist$.
The lower bound can be deduced from our
results in Section \ref{sec:lowBdSIS}:

\begin{cor}\label{cor:lowBdSIS}
For $s,t\in\N$ with $s\geq t$ we have
\beq\label{eq:lowBdSIS}
\sist \geq \frac{s+2}{s+t+4}.
\eeq
\end{cor}

\begin{proof}
Simply note that in the notation of Section 
\ref{sec:lowBdSIS}, $\sist=e_{\frac s2,\frac t2}$ and use 
Proposition \ref{prop:preLowBdSIS}.
\end{proof}

Combining this lower bound for $\sist$ with Theorem \ref{simmons+}, we have
for $s,t\in\N$ with $s\geq t$,
\beq\label{eq:chain1}
\isp(s,t):=\frac{s+2}{s+t+4} \leq\sist\leq \frac s{s+t} =: \psi (s,t).
\eeq

\begin{lem}\label{lem:trip0}
For $s,t\in\N$ with $s\geq t$ we have
\beq\label{eq:chain2}
g_{s,t}(\isp(s,t)) \leq
g_{s,t}(\sist) = f_{s,t}(\sist) = g_{s,t}(\sist) \leq
g_{s,t}(\psi(s,t)).
\eeq
\end{lem}

\begin{proof}
This follows from 
the monotonicity of $g_{s,t}$ on 
$\ds \left[0,\frac s{s+t}\right]$ and by the coincidence of 
$f_{s,t}$ and $g_{s,t}$ in the equipoint $\sist$.
\end{proof}

\ssec{Two step monotonicity of $f_{s,t}(\sis)$}

In this subsection, in Proposition \ref{prop:f2step},
we show
for $s,t\in\N$ with $s\geq t$ that
$
f_{s,t}(\sist) \leq f_{s+2,t-2}(\sigma_{s+2,t-2}).
$.

\begin{lem}\label{lem:trip1}
If $s,t\in\R_{>0}$ with $s\geq t$ then
\beq\label{eq:trip1}
g_{s+2,t-2}(\isp(s+2,t-2)) \geq g_{s,t}(\psi(s,t)).
\eeq
\end{lem}

\begin{proof}
With $d=s+t$, \eqref{eq:trip1} is equivalent to
\beq\label{eq:ms1}
(4 + s)^{s + 2}   d^d \geq s^s (4 + d)^d (2 + s)^2.
\eeq
This follows from \eqref{eq:gst} and the identities
$\ds B(\al+1,\be)=\frac{\al}{\al+\be}B(\al,\be)$
and
$\ds B(\al,\be+1)=\frac{\be}{\al+\be}B(\al,\be)$.

Rewrite \eqref{eq:ms1} into
\beq\label{eq:ms2}
\left(1+\frac4d\right)^d \leq 
\left(1+\frac4s\right)^s \left(\frac{s+4}{s+2}\right)^2 =:\xi(s).
\eeq

We claim the right-hand side $\xi(s)$ is an increasing function of $s$ on $\R_{\geq0}$.
Indeed, using $s=2S$, 
\[
\begin{split}
\Xi(S)&=\xi\left(\frac s2\right) = 
\left(1+\frac2S\right)^{2S} \left(\frac{S+2}{S+1}\right)^2 \\
&= 
\left(1+\frac2S\right)^{S} \cdot \left(1+\frac2S\right)^{S} \left(\frac{S+2}{S+1}\right)^2.
\end{split}
\]
The first factor in the last expression is well-known to be 
an increasing function with limit $e^2$. Let
\[
\zeta(S):=\left(1+\frac2S\right)^{S} \left(\frac{S+2}{S+1}\right)^2.
\]
Then 
\[
\zeta'(S)=\frac{(S+2)^2 \left(\frac{S+2}{S}\right)^S \left((S+1) \log
   \left(\frac{S+2}{S}\right)-2\right)}{(S+1)^3}.
\]
Observe that
$\ds
(S+1) \log
   \left(\frac{S+2}{S}\right)-2 \geq0$
   iff
   \[
   \left(1+\frac2S\right)^{S+1}\geq e^2.
   \]
   But it is well-known and easy to see that this left-hand side
   is decreasing with limit $e^2$.
   This shows that $\zeta'(S)\geq0$ and hence
the right-hand side  of \eqref{eq:ms2} is an increasing function of $s$.

It thus suffices to show
\beq\label{eq:ms3}
\left(\frac{d+4}d\right)^d \leq 
\xi\left(\frac d2\right) = \left(\frac{d+8}{d+4}\right)^2
\left(\frac{d+8}{d}\right)^{\frac d2},
\eeq
or equivalently,
\beq\label{eq:ms4}
\left(\frac{d+4}d\right)^{\frac d2} \leq 
\left(\frac{d+8}{d+4}\right)^{2+\frac d2}.
\eeq
Again, we show this hold for $d\in\R_{>0}$.  Writing
$d=4D$, \eqref{eq:ms4} becomes
\beq\label{eq:ms5}
\left(\frac{D+1}D\right)^{ 2D} \leq 
\left(\frac{D+2}{D+1}\right)^{2+2D}.
\eeq
So it suffices to establish
\[
\left(1+\frac{1}D\right)^{ D} \leq 
\left(1+\frac{1}{D+1}\right)^{D+1}.
\]
But this is well-known or easy to establish using calculus.
\end{proof}

\begin{prop}\label{prop:f2step}
For $s,t\in\N$ with $s\geq t$ we have
\beq\label{eq:f2step}
f_{s,t}(\sist) \leq f_{s+2,t-2}(\sigma_{s+2,t-2}).
\eeq
\end{prop}

\begin{proof}
Observe that
\[
\begin{split}
f_{s+2,t-2}(\sigma_{s+2,t-2}) & \stackrel{\eqref{eq:chain2}}{\geq}
g_{s+2,t-2}(\isp(s+2,t-2))  \\
& \stackrel{\eqref{eq:trip1}}{\geq} 
g_{s,t}(\psi(s,t))  \\
&
\stackrel{\eqref{eq:chain2}}{\geq}
f_{s,t}(\sist). \qedhere
\end{split} 
\]
\end{proof}

\ssec{Boundary cases}

Having established for each $d$ monotonicity in $s$ of
$f_{s,d-s}(\sigma_{s,d-s})$, 
where $\frac d 2 \le s \leq d-2$,
we now turn to the boundary cases.

\begin{lem}\label{lem:trip2}
For $s\in\N$ we have
\beq\label{eq:trip2}
g_{s+1,s-1}(\isp(s+1,s-1)) \geq g_{s,s}(\psi(s,s)).
\eeq
\end{lem}

\begin{rem}\rm
 \label{rem:chu}
  The proof uses Chu's inequality (see e.g.~\cite[p. 288]{MV70}) on the quotient of gamma functions,
  which says that for $s\in\mathbb N$,
\[
 \sqrt{\frac{2s+1}{4}} \le \frac{\Gamma(\frac{s}{2}+1)}{\Gamma(\frac{s+1}{2})} \le \frac{s+1}{\sqrt{2s+1}}.
\]
 Thus, it is at this point the assumption that $s$ is an integer is used. 
\end{rem}

\begin{proof}
Inequality \eqref{eq:trip2} is equivalent to 
\beq\label{eq:ms7}
\frac{(s+2)^{-s} ((s+1) (s+3))^{\frac{s+1}{2}} \Gamma
   \left(\frac{s}{2}+1\right)^2}{2 \Gamma \left(\frac{s+3}{2}\right)^2}\geq1.
\eeq
Further, using Chu's inequality 
it suffices to establish
\beq\label{eq:ms8}
\left(1-\frac{1}{(s+2)^2}\right)^{\frac{s+1}{2}}
\geq
1- \frac{1}{2 s+5}.
\eeq
Equivalently,
\beq\label{eq:ms9}
\left(1-\frac{1}{(s+2)^2}\right)^{\frac{s+1}{2}(2s+5)}
\geq
\left(1- \frac{1}{2 s+5}\right)^{2s+5}.
\eeq
The right-hand side of \eqref{eq:ms9} is increasing with 
limit as $s\to\infty$ being $\ds e^{-1}$. 
Further,
as
\[
\frac{s+1}{2}(2s+5)\leq 
(s+2)^2-1 
\qquad \text{ for }s\geq-1,
\]
we have
\beq\label{eq:ms10}
\left(1-\frac{1}{(s+2)^2}\right)^{\frac{s+1}{2}(2s+5)}
\geq
\left(1-\frac{1}{(s+2)^2}\right)^{(s+2)^2-1}
\eeq

Now consider 
\[
\zeta(x):=\left(1-\frac1{x^2}\right)^{x^2-1}.
\]
We claim it is (for $x>1$) decreasing. Indeed,
\[
\zeta'(x)=
\frac{2 \left(1-\frac{1}{x^2}\right)^{x^2} x \left(x^2 \log
   \left(1-\frac{1}{x^2}\right)+1\right)}{x^2-1}
   \]
   and
   \[
   x^2 \log
   \left(1-\frac{1}{x^2}\right)+1<0
   \]
   since
   \[
   \left(1-\frac{1}{x^2}\right)^{x^2}
   \]
  is increasing with limit $e^{-1}$.

Now the left-hand side of \eqref{eq:ms9} is greater
than the right-hand side of \eqref{eq:ms10} which
is decreasing with $s$ towards $\ds e^{-1}$
which is an upper bound on the right-hand side of \eqref{eq:ms9}.
   \end{proof}

\begin{lem}\label{lem:trip3}
For $s\in\R$ with $s\geq1$ we have
\beq\label{eq:trip3}
g_{s+2,s-1}(\isp(s+2,s-1)) \geq g_{s+1,s}(\psi(s+1,s)).
\eeq
\end{lem}   
   
   \begin{proof}
 Expanding $g$'s 
as was done to obtain \eqref{eq:ms1}, 
 we see \eqref{eq:trip3} is equivalent to
\beq\label{eq:ms12}
\xi(s):=\frac{s+4}{s+2} \left(1+\frac4s\right)^{\frac s2}
 \left(\frac{2 s+1}{2
   s+5}\right)^{s+\frac{1}{2}}\geq1
\eeq
Letting $s=2S$, consider
\[
\Xi(S)=\xi\left(\frac s2\right) =
\frac{S+2}{S+1} \left(1+\frac2S\right)^{S}
 \left(1-\frac{4}{4S+5}\right)^{2S+\frac{1}{2}}.
\]

We have to show that $\Xi(S)\ge1$ for $S\in[\frac12,\infty)$. To this end, we will show that
$\Xi'(S)\le0$ for $S\in[\frac12,\infty)$ and $\lim_{S\to\infty}\Xi(S)=1$. For the latter, note that
\[\lim_{S\to\infty}\left(1-\frac{4}{4S+5}\right)^{2S+\frac{1}{2}}=\lim_{S\to\infty}\left(1-\frac{4}{4S+5}\right)^{-2}
\sqrt{\lim_{S\to\infty}\left(1-\frac{4}{4S+5}\right)^{5+4S}}=e^{-4}\]
and therefore
\[\lim_{S\to\infty}\Xi(S)=e^{2}\sqrt{e^{-4}}=1.
\]
A straightforward computation shows
\[\Xi'(S)=\frac{\left(1+\frac2S\right)^S \left(1-\frac{4}{4 S+5}\right)^{2 S+\frac{1}{2}}}{(S+1)^2 (4 S+5)}\Xi_1(S)\]
where
\begin{multline*}
\Xi_1(S):=1+2 S+(S+1)
   (S+2) (4 S+5)\\
   \log\left(1-\frac{8}{4 S+5}+\frac{16}{(4 S+5)^2}-\frac{16}{(4 S+5) S}+\frac{32}{(4 S+5)^2 S}+\frac{2}{S}\right).
\end{multline*}
It is enough to show that $\Xi_1(S)\le0$ for $S\in[\frac12,\infty)$. Using $\log x\le x-1$ for $x>0$, we have
\begin{multline*}
\Xi_1(S)\le\\
1+2 S+(S+1)(S+2) (4 S+5)
\left(-\frac{8}{4 S+5}+\frac{16}{(4 S+5)^2}-\frac{16}{(4 S+5) S}+\frac{32}{(4 S+5)^2 S}+\frac{2}{S}\right)\\
=\frac{9}{5 (4 S+5)}+\frac{4}{5 S}-2
\end{multline*}
which evaluates to $-\frac17$ for $S=\frac12$ and therefore is clearly negative for $S\ge\frac12$.
\end{proof}

\begin{proof}[Proof of Proposition {\rm\ref{prop:thetafirst}}]
Suppose $d=s +t$ is an even integer. If $s$ and $t$ are even  integers with $s \geq t$, then
two step monotonicity in
Proposition \ref{prop:f2step}
tells us that the minimizer over $r$ even with $s-r \geq t+r$,
 of
$f_{s-r,t+r}(\sigma_{s-r,t+r})$ occurs where $s-r=t+r$;
that is, $s=\frac d 2 =t$.
 If $s$ and $t$ are odd  integers with $s \geq t$, then
  stepping $r$ by $2$ preserves odd integers, so
 two step monotonicity  gives that the minimizer in
this case  is at  $ s= \frac d 2  + 1$
 and $ t= \frac d 2  - 1$.
 Compare the two minimizers using Lemmas \ref{lem:trip1} 
 and
\ref{lem:trip2} which give:
 \beq 
  f_{s,s}(\sigma_{s,s})  \leq
  g_{s,s}(\psi(s,s)) \leq
g_{s+1,s-1}(\isp(s+1,s-1)) \leq f_{s+1,s-1}(\sigma_{s+1,s-1}).
\eeq
Apply this inequality to $s= \frac d 2$ to get the minimizer  is $\frac d 2$.

Suppose  $d$ is an odd integer.
As before, Proposition \ref{prop:f2step} gives that the minimizer 
of $  f_{s,t}(\sigma_{s,t})$
over $s$  odd and $t$  even is $s= \frac d 2 + \frac 1 2$ and
$t= \frac d 2 - \frac 1 2$.
Likewise  minimizing over $t$  odd and $s$  even
yields   a minimizer which compares to the previous one unfavorably 
(using Lemmas \ref{lem:trip1} 
 and \ref{lem:trip3}).
 \end{proof}

\section{Estimating  $\th(d)$ for Odd $d$.}
\label{sec:odd}
Recall that $\th(d)$ ($d\in\N$) has been introduced in \eqref{eq:BTbd} and was simplified in Proposition \ref{prop:starastarb}.
It was explicitly determined in \eqref{eq:mc} for even $d$ by the expression which we repeat in part (b) of Theorem \ref{thm:thetaEvenAndOdd}. For
odd $d$, we have only the implicit characterization of Theorem \ref{thm:thetaExplicit}. We do not know a way of making this more explicit. In this section,
we exhibit however, for odd $d$, a compact interval containing $\th(d)$ whose end-points are given by nice analytic expressions in $d$. For the upper
end point of the interval, we provide two versions: one involving
the gamma function only and another which seems to be even tighter 
but involves the regularized beta function.
The main result of this section is

\begin{thm}
\label{thm:thetaEvenAndOdd}
Let $d\in\N$.
\begin{enumerate}[\rm(a)]
\item $\th(1)=1$
\item
Suppose $d$ is even. Then
\[
 \th(d) = \sqrt\pi\frac{\Ga\left(1+\frac d4\right)}{\Gamma\left(\frac12+\frac d4\right)}.
\]
\item
Suppose $d\ge3$ is odd. Then there is a unique $p\in[0,1]$ satisfying
\[I_p\left(\frac{d+1}4,\frac{d+3}4\right)=I_{1-p}\left(\frac{d-1}4,\frac{d+5}4\right).\]
For this $p$, we have $p\in[\frac12,\frac{d+1}{2d}]$,
\beq\label{eq:lowUpOdd}
\th_-(d)\le
\th(d)=\frac{\Ga\left(\frac{d+3}4\right)\Ga\left(\frac{d+5}4\right)}{p^{\frac{d-1}4}(1-p)^{\frac{d+1}4}\Ga\left(\frac d2+1\right)}\le\min\{\th_+(d),\th_{++}(d)\}
\eeq
where $\th_-(d)$, $\th_+(d)$ and $\th_{++}(d)$ are given by
\begin{align*}
\th_-(d)&=
\sqrt[4]{\frac{d^{2d}}{(d+1)^{d+1}(d-1)^{d-1}}}\;
\th_{++}(d),\\
\frac1{\th_+(d)}&=
\frac{d-1}dI_{\frac{d+1}{2 d}}\left(\frac{d+1}{4},\frac{d+3}{4}\right)+\frac{d+1}dI_{\frac{d-1}{2
   d}}\left(\frac{d-1}{4},\frac{d+5}{4}\right)-1\text{ and}\\
\th_{++}(d)&=\sqrt{\frac\pi2}\frac{\Ga\left(\frac{d+3}2\right)}{\Ga\left(\frac d2+1\right)}.
\end{align*}
\end{enumerate}
\end{thm}

\subsection{Proof of Theorem \ref{thm:thetaEvenAndOdd}}

Our starting point is the optimization Problem \ref{prob:opt}
(or its equivalent Problem \ref{prob:opt'}) from Section \ref{sec:opt}.
Any good approximation $p$ of $\sigma_{s,t}$ will now give upper bound $f_{s,t}(p)$ on the minimum $f_{s,t}(\sigma_{s,t})$ of $f_{s,t}$.
This will be our strategy to get a good upper bound of $\frac1{\th(d)}$, i.e., a good lower bound of $\th(d)$. Getting good upper bounds of $\th(d)$ will be harder.
To do this, we introduce sort of artificially simplified versions $g_{s,t},h_{s,t}\colon[0,1]\to\R$ of $f_{s,t}$ having the property
\[f(\sigma_{s,t})=g(\sigma_{s,t})=h(\sigma_{s,t})\]
but decreasing
at least in one direction while moving away from $\sigma_{s,t}$. The upper bound of $\th(d)$ arising from $g$ seems to be tighter while the one arising from $h$ will be given
by a simpler expression. Let these functions be given by 
\begin{align*}
g_{s,t}(p)&=2 \left(\frac{s I_{1-p}\left(\frac{t}{2},\frac{s}{2}+1\right)}{s+t}+\frac{tI_p\left(\frac{s}{2},\frac{t}{2}+1\right)}{s+t}\right)-1\qquad\text{and}\\
h_{s,t}(p)&=I_{1-p}\left(\frac{t}{2},\frac{s}{2}+1\right)+I_p\left(\frac{s}{2},\frac{t}{2}+1\right)-1
\end{align*}
for $p\in[0,1]$. Using the standard identities with beta functions, it is easy to compute
\begin{align*}
g_{s,t}'(p)&=\frac{2 p^{\frac{s}{2}-1}(1-p)^{\frac{t}{2}-1}}{B\left(\frac{s}{2},\frac{t}{2}\right)}(1-2p)\qquad\text{and}\\
h_{s,t}'(p)&=\frac{p^{\frac{s}{2}-1} (1-p)^{\frac{t}{2}-1} (s+t) ((1-p)s-pt)}{s t B\left(\frac{s}{2},\frac{t}{2}\right)}
\end{align*}
for $p\in[0,1]$. Therefore $g_{s,t}$ is strictly increasing on $[0,\frac12]$ and strictly decreasing on $[\frac12,1]$, and $h_{s,t}$ is strictly increasing on
$[0,\frac s{s+t}]$ and strictly decreasing on $[\frac s{s+t},1]$. Another useful identity which we will use is
\begin{equation}\label{eq:fs}
f_{s,t}\left(\frac s{s+t}\right)=h_{s,t}\left(\frac s{s+t}\right).
\end{equation}

\begin{lemma}\label{hsimple}
Let $s,t\in\N$ and $p\in[0,1]$. Then
\begin{align*}
h_{s,t}(p)&=\frac{2(s+t)}{stB(\frac s2,\frac t2)}p^{\frac s2}(1-p)^{\frac t2}\\
&=\frac{p^{\frac s2}(1-p)^{\frac t2}}{\left(\frac{s}{2}+\frac{t}{2}+1\right)B\left(\frac{s}{2}+1,\frac{t}{2}+1\right)}\\
&=\frac{\Gamma \left(\frac{s}{2}+\frac{t}{2}+1\right)}{\Gamma\left(\frac{s}{2}+1\right) \Gamma \left(\frac{t}{2}+1\right)}p^{\frac s2}(1-p)^{\frac t2}.
\end{align*}
\end{lemma}

\begin{proof} Using the identities $\ds I_p(x,y)=I_p(x-1,y+1)-\frac{p^{x-1}(1-p)^y}{yB(x,y)}$\footnote{see (8.17.19) in  \url{http://dlmf.nist.gov/8.17}}
and
$\ds B\left(\frac{s}{2}+1,\frac{t}{2}\right)=\frac s{s+t}B\left(\frac{s}{2},\frac{t}{2}\right)$,
we get
\begin{multline*}
1-I_{1-p}\left(\frac t2,\frac s2+1\right)=
I_p\left(\frac s2+1,\frac t2\right)=I_p\left(\frac s2,\frac t2+1\right)-\frac{p^{\frac s2}(1-p)^{\frac t2}}{\frac t2B(\frac s2+1,\frac t2)}\\
=I_p\left(\frac s2,\frac t2+1\right)-\frac{2(s+t)p^{\frac s2}(1-p)^{\frac t2}}{stB(\frac s2,\frac t2)}
\end{multline*}
and therefore $h_{s,t}(p)$ equals the first of the three expressions. Using
$\ds B\left(\frac{s}{2}+1,\frac{t}{2}+1\right)=\frac s{s+t+2}B\left(\frac{s}{2},\frac{t}{2}+1\right)=\frac{st}{(s+t)(s+t+2)}B\left(\frac{s}{2},\frac{t}{2}\right)$,
we get from this that $h_{s,t}(p)$ also equals the second expression. Finally,
\[B\left(\frac{s}{2}+1,\frac{t}{2}+1\right)=\frac{\Ga\left(\frac{s}{2}+1\right)\Ga\left(\frac{t}{2}+1\right)}{\Ga\left(\frac{s}{2}+\frac{t}{2}+2\right)}
=\frac{\Ga\left(\frac{s}{2}+1\right)\Ga\left(\frac{t}{2}+1\right)}{\left(\frac s2+\frac t2+1\right)\Ga\left(\frac{s}{2}+\frac{t}{2}+1\right)},
\]
yielding that $h_{s,t}(p)$ equals the third expression.
\end{proof}

\begin{proof}[Proof of Theorem {\rm\ref{thm:thetaEvenAndOdd}}]
(a) is clear. Part (b) has already been proven in \eqref{eq:mc} but we shortly give again an argument: If $d$ is even, we know that
\[\frac1{\th(d)}=f_{\frac d2,\frac d2}(\sigma_{\frac d2,\frac d2})=h_{\frac d2,\frac d2}(\sigma_{\frac d2,\frac d2})\]
but obviously $\sigma_{\frac d2,\frac d2}=\frac12$ and so by Lemma \ref{hsimple}
\[\frac1{\th(d)}=h_{\frac d2,\frac d2}\left(\frac12\right)=
\frac{\Ga\left(\frac d2+1\right)}{\Ga\left(\frac{d}{4}+1\right)^22^{\frac d2}}.\]
By the Lagrange duplication formula\footnote{see 5.5.5 in  \url{http://dlmf.nist.gov/5.5}}
we have
\[
\Ga\left(\frac d2+1\right)=\Ga\left(2\left(\frac d4+\frac 12\right)\right)=\frac1{\sqrt\pi2^{\frac d2}}\Ga\left(\frac d4+\frac 12\right)\Ga\left(\frac d4+1\right)
\]
which proves part (b).

It remains to prove (c) and we suppose from now on that $d\ge3$ is odd. Then the defining Equation \eqref{eq:defSIS}
for $p:=\sigma_{\frac{d+1}2,\frac{d-1}2}$ is the one stated in (c) and
we know from Theorem \ref{thm:thetaExplicit} that
\[\frac1{\th(d)}=f_{\frac{d+1}2,\frac{d-1}2}(p)=g_{\frac{d+1}2,\frac{d-1}2}(p)=h_{\frac{d+1}2,\frac{d-1}2}(p).\]
Using Lemma \ref{hsimple}, we get
\[\frac1{\th(d)}=h_{\frac{d+1}2,\frac{d-1}2}(p)=
\frac{\Gamma \left(\frac{d+1}{4}+\frac{d-1}{4}+1\right)}{\Gamma\left(\frac{d+1}{4}+1\right) \Gamma \left(\frac{d-1}{4}+1\right)}p^{\frac{d+1}4}(1-p)^{\frac{d-1}4},\]
showing the equality we claim for $\th(d)$. From 
Section \ref{sec:lowBdSIS} we know that
\[\frac12\le p\le\frac{d+1}{2d}.\]
By the monotonicity properties observed earlier, this implies
\[g_{\frac{d+1}2,\frac{d-1}2}\left(\frac{d+1}{2d}\right)\le g_{\frac{d+1}2,\frac{d-1}2}(p)=\frac1{\th(d)}\]
and
\[h_{\frac{d+1}2,\frac{d-1}2}\left(\frac12\right)\le h_{\frac{d+1}2,\frac{d-1}2}(p)=\frac1{\th(d)}.\]
The first inequality shows by a simple calculation that $\th(d)\le\th_+(d)$.

To show that $\th(d)\le\th_{++}(d)$, it is enough to verify that
\begin{equation}\label{eq:pluplu}
h_{\frac{d+1}2,\frac{d-1}2}\left(\frac12\right)=\frac1{\th_{++}(d)}.
\end{equation}
To this end, we use the third expression in Lemma \ref{hsimple} to obtain
\[
h_{\frac{d+1}2,\frac{d-1}2}\left(\frac12\right)=\frac{\Gamma \left(\frac d2+1\right)}{\Gamma\left(\frac{d+1}{4}+1\right) \Gamma \left(\frac{d-1}{4}+1\right)2^{\frac d2}}.
\]
By the Lagrange duplication formula, we have
\[
\Ga\left(\frac{d+3}2\right)=\Ga\left(2\left(\frac{d+3}4\right)\right)=\frac1{\sqrt\pi}2^{\frac{d+1}2}\Ga\left(\frac{d+3}4\right)\Ga\left(\frac{d+5}4\right)
\]
and therefore
\[h_{\frac{d+1}2,\frac{d-1}2}\left(\frac12\right)=\frac{\Gamma \left(\frac d2+1\right)}{\Gamma \left(\frac{d+3}2\right)\frac{\sqrt\pi}{\sqrt2}}=\sqrt{\frac2\pi}\frac{\Gamma \left(\frac d2+1\right)}{\Gamma \left(\frac{d+3}2\right)}=\frac1{\th_{++}(d)}.\]

Finally, we show that $\th_-(d)\le\th(d)$. This will follow from
\[\frac1{\th(d)}=f_{\frac{d+1}2,\frac{d-1}2}(p)\le f_{\frac{d+1}2,\frac{d-1}2}\left(\frac{d+1}{2d}\right)\overset{\eqref{eq:fs}}=h_{\frac{d+1}2,\frac{d-1}2}\left(\frac{d+1}{2d}\right)\]
if we can show that
\[h_{\frac{d+1}2,\frac{d-1}2}\left(\frac{d+1}{2d}\right)=\frac1{\th_-(d)}.\]
But this follows from
\begin{multline*}
\frac1{\th_-(d)}=\left(\frac{d+1}d\right)^{\frac{d+1}4}\left(\frac{d-1}d\right)^{\frac{d-1}4}\frac1{\th_{++}(d)}=\\
\left(\frac{d+1}d\right)^{\frac{d+1}4}\left(\frac{d-1}d\right)^{\frac{d-1}4}h_{\frac{d+1}2,\frac{d-1}2}\left(\frac12\right)=h_{\frac{d+1}2,\frac{d-1}2}\left(\frac{d+1}{2d}\right)
\end{multline*}
where the second equality stems from \ref{eq:pluplu} and the third from Lemma \ref{hsimple}.
\end{proof}

The following table shows the approximate values of $\th(d)$ and its lower and upper bounds from Theorem
\ref{thm:thetaEvenAndOdd} for $t\le20$

\[
\begin{array}{r|cccc}
 d & \th_-(d) & \th(d) & \th_+(d) & \th_{++}(d)\\
 \hline
 1 & - & 1 & - & - \\
 2 & - & 1.5708 & - & - \\
 3 & 1.73205 & 1.73482 & 1.77064 & 1.88562 \\
 4 & - & 2 & - & - \\
 5 & 2.15166 & 2.1527 & 2.17266 & 2.26274 \\
 6 & - & 2.35619 & - & - \\
 7 & 2.49496 & 2.49548 & 2.50851 & 2.58599 \\
 8 & - & 2.66667 & - & - \\
 9 & 2.79445 & 2.79475 & 2.80409 & 2.87332 \\
 10 & - & 2.94524 & - & - \\
 11 & 3.064 & 3.06419 & 3.07131 & 3.13453 \\
 12 & - & 3.2 & - & - \\
 13 & 3.31129 & 3.31142 & 3.31707 & 3.37565 \\
 14 & - & 3.43612 & - & - \\
 15 & 3.54114 & 3.54123 & 3.54585 & 3.6007 \\
 16 & - & 3.65714 & - & - \\
 17 & 3.75681 & 3.75688 & 3.76076 & 3.8125 \\
 18 & - & 3.86563 & - & - \\
 19 & 3.96068 & 3.96073 & 3.96404 & 4.01316 \\
 20 & - & 4.06349 & - & - \\
\end{array}
\]

\subsection{Explicit bounds on $\th(d)$}

In this subsection we present explicit 
bounds on $\th(d)$ for $d\in\N$.
Theorem \ref{thm:thetaEvenAndOdd} gives
an explicit analytic expression for $\th(d)$ when
$d\in\N$ is even, and gives analytic bounds
for $\th(d)$ with $d$ odd.

\begin{prop}
Let $d\in\N$. 
\ben[\rm(1)]
\item
If $d$ is even, then
\[
\frac {\sqrt{\pi}}{2}\sqrt{d+1} \leq
\vartheta(d) \leq\frac{\sqrt{\pi}}{2}\cdot
\frac d{\sqrt{d-1}}.
\]
\item
If $d$ is odd, then
\[
\sqrt[4]{ \left(1-\frac1{d+1}\right)^{d+1}
\left(1+\frac1{d-1}\right)^{d-1}
}
\cdot \frac {\sqrt{\pi}}{2}\sqrt{d+\frac32}
\leq
\vartheta(d) \leq\frac{\sqrt{\pi}}{2}\cdot
\frac {d+2}{\sqrt{d+\frac52}}.
\]
\item
We have $\ds\lim_{d\to\infty} \frac{\th(d)}{\sqrt d}=\frac{\sqrt\pi}2$.
\een
\end{prop}

\begin{proof}
Suppose that $d$ is even. By Theorem \ref{thm:thetaExplicit}, 
\beq\label{eq:thetaExplicit}
\th(d)= \sqrt\pi \frac{ \Gamma\left(\frac d4+1\right)}{\Gamma\left(\frac d4+\frac12\right)}.
\eeq
Since $d$ is even, we may apply  Chu's inequality  (see Remark \ref{rem:chu}),
to the the right-hand side of
\eqref{eq:thetaExplicit} and obtain (1).

Similarly, (2) is obtained by applying Chu's inequality 
to \eqref{eq:lowUpOdd}. Finally, (3) is an easy consequence of
(1) and (2).
\end{proof}

Observe that these upper bounds are tighter than the
bound $\ds\th(d)\leq \frac{\pi}2\sqrt d$ given in \cite{BtN}.

\section{Dilations and Inclusions of Balls}
\label{sec:balls}
 In this section free relaxations of the problem of including 
 the unit ball of $\R^g$ into a
 spectrahedron  are considered.  Here the focus is the dependence of the inclusion scale as a function of $g$ (rather than $d$).  Among the results we identify the worst case inclusion constant as $g$. This inclusion constant can be viewed as a symmetric variable matrix version  of the quantitative measure $\alpha(\mathbb C^g)$ of the difference between the maximal and minimal operator space structures associated with the unit ball in $\mathbb C^g$ introduced by Paulsen \cite{Pau} for which the best results give only upper and lower bounds \cite{Pisier-book}.

\subsection{The general dilation result}

 Let $A\in\smatdg$ be a given $g$-tuple and assume $\cDA$ is bounded. 

\begin{prop} \label{prop:small g}
  Suppose $\cDA$ has the property that if $C\in\cDA$ and $1\le j\le g$, then
\[
 \hat{C}_j = (0,\dots,0,C_j,0,\dots,0)\in \cDA.
\]
  If $C\in \cDA(n)$, then there exists a commuting tuple $T\in\cDA(gn)$ such that $C$ dilates to $gT$.
The estimate is sharp in that $g$ is the smallest number such that for every $g$-tuple $A$ (satisfying the assumption above) and $g$-tuple $X\in \cDA$ there exists a commuting $g$-tuple $T$ of symmetric matrices of size $gn$ such that $X$ dilates to $gT$ and  the joint spectrum of $T$ lies in $\cDA(1)$. 
\end{prop}

\begin{proof}
  Given $C$, let $T_j = \oplus_{k=1}^g  T_{jk}$, where $T_{jk}=0$ (the $n\times n$ zero matrix) if $j\ne k$ and $T_{jj} = C_j$. It is automatic that the $T_j$ commute. Further,
\[
 L_A(T) = 
   \bigoplus_{j=1}^g \big( I- A_j \otimes C_j \big) \succeq 0,
\]
 so that $T\in\cDA$.  Finally, let $V:\mathbb R^n\to\mathbb R^n\otimes \mathbb R^g$ denote the mapping
\[
  Vh = \frac{1}{\sqrt{g}} \bigoplus_1^g h.
\]
 It is routine to verify that $V^* TV= \frac{1}{g} C$.
 
 The proof of sharpness is more difficult and is completed in Corollary \ref{cor:g-sharp}. 
\end{proof}

\begin{rem}\rm
  The hypothesis of Proposition \ref{prop:small g} applies to the matrix cube.
 When $d$ is large and $g$ is small, the proposition gives a better estimate 
 for the matrix cube relaxation than does Theorem \ref{thm:BtN} of Ben-Tal and Nemirovski.
More generally, if $\cDA$ is \df{real Reinhardt}, meaning if $X=(X_1,\dots,X_g)\in \cDA$, then all the tuples $(\pm X_1,\dots,\pm X_g)\in\cDA$, then $\cDA$ satisfies the hypothesis of Proposition \ref{prop:small g}. Of course any $\cD_{L_B}$ can be embedded in a $\cD_{L_A}$ which satisfies the hypotheses of Proposition \ref{prop:small g}.
\end{rem}

 \subsection{Four types of balls}

 Let \df{$\mathbb B_g$} denote the unit ball in $\mathbb R^g$.  
 Here we consider four matrix convex sets each of which, at level $1$, 
 equal  $\mathbb B_g$. Two of these we know to be free spectrahedra. A third is for  $g=2$, but likely not for $g\ge 3$.  The remaining one we will prove is not a free spectrahedron as part of a forthcoming paper.

\subsubsection{The OH ball}
 \label{sec:oh ball}
 The \df{OH ball}   (for operator Hilbert space  \cite{Pisier-book})   \df{$\mfBoh_g$} is the set of 
  tuples $X=(X_1,\dots,X_g)$ of symmetric  matrices such that 
\[
  \sum_{j=1}^g X_j^2 \preceq  I.
\]
 Equivalently, the row matrix $\begin{pmatrix} X_1 & X_2 & \dots & X_g\end{pmatrix}$ 
 (or its transpose) has norm at most $1$. 
The ball $\mfBoh_g $ is symmetric about the origin and also satisfies the conditions of Proposition \ref{prop:small g}. 

\begin{example}\rm
 \label{ex:g=d=2}
 For two variables, $g=2$, the  commutability index 
 of $\mfBoh_2(2)$ is at least  $\frac 1{\sqrt2}$.

 Let 
\[
 C_1 = \frac{1}{\sqrt{2}} \begin{pmatrix} 1 & 0\\0 & -1 \end{pmatrix}
\qquad and
\qquad
 C_2 = \frac{1}{\sqrt{2}} \begin{pmatrix} 0 & 1\\ 1 & 0\end{pmatrix}.
\]
 Evidently $C=(C_1,C_2) \in\mfBoh_2$. 
 Suppose $T=(T_1,T_2)$ is a commuting tuple of size $2+k$ which dilates $C$. Thus,
\[
 T_j =\begin{pmatrix} C_j & a_j \\ a_j^* & d_j\end{pmatrix},
\]
 where $a_j$ is $2\times k$ and $d_j$ is $k\times k$. 
  Commutativity of the tuple $T$ implies, 
\[
  \begin{pmatrix} 0 & 1\\ -1 & 0\end{pmatrix} 
     = C_1 C_2 -C_2 C_1  = a_1 a_2^* - a_2 a_1^* = 
   \begin{pmatrix} a_1 & -a_2 \end{pmatrix} \, \begin{pmatrix} a_1^* \\ a_2^*\end{pmatrix}.
\]
 It follows that either the norm of $(a_1 \ \ -a_2)$ or $(a_1 \ \ a_2)^*$ is at least one. In either case,
\[
 a_1 a_1^* + a_2 a_2^* \not\preceq I.
\]
 On the other hand, the $(1,1)$ block entry of $T_1^2+T_2^2$ is 
\[
 I_2 + a_1 a_1^* + a_2 a_2^* \not\preceq 2 I_2. 
\]

 Thus, for the tuple $C$ the smallest $\rho$ for which there exists a tuple $T$ of commuting operators with spectrum in $\mathbb B_g$ such that $C$ dilates to $\rho T$ is at least $\sqrt{2}$. In other words, the commutability index of $\mfBoh_2(2)$ is at most $\frac{1}{\sqrt{2}}$. 
 \end{example}

\subsubsection{The min and max balls}
 \label{sec:min ball}
  Let \df{$\mfBm_g$} denote the \df{min  ball} (the unit ball $\mathbb B_g$   with the minimum operator system structure).  
   Namely, $X=(X_1,\dots,X_{g})\in \mfBm_g$ if 
\[
 \sum_{j=1}^{g} x_j X_j \preceq  I
\]
 for all unit vectors $x\in\mathbb R^{g}$. 

\begin{lemma}
   For a tuple $X$ of $n\times n$ symmetric matrices, the following are equivalent.
 \begin{enumerate}[\rm (i)]
  \label{lem:min alts}
   \item $X$ is in the min ball;
   \item $\mathbb B_g \subset \cD_{L_X}(1)$;
   \item for each unit vector $v\in\mathbb R^n$ the vector $v^*Xv = (v^* X_1v,\dots,v^*X_gv)\in\mathbb B_g$;
   \item $X\in \cDA(n)$ for every $g$-tuple $A$ of symmetric $1\times 1$ matrices for which  $\cDA(1)\supseteq \mathbb B_g$.
 \end{enumerate}
\end{lemma}

\begin{proof}
  The equivalence of (i) and (ii) is immediate.  The tuple $X$ is in the min ball if and only if for each pair of unit vectors $x,v\in\mathbb B_g$,
\[
  \sum_{j=1}^g  x_j (v^* X_j v) \le v^*v = 1
\]
 and the equivalence of (i) and (iii) follows.    Now suppose $X$ is not in the min ball. In this case there exists $a\in\mathbb B_g$ such that $X\notin \cD_{L_a}(1)$ but of course $\cD_{L_a}(1)\supseteq \mathbb B_g$. Thus (iv) implies (i).  If (iv) doesn't hold, then there is a $a\in\mathbb B_g$ such that $X\notin \cD_{L_a}(1)$, but $\cD_{L_a}(1)\supseteq \mathbb B_g$. This latter inclusion implies $a\in\mathbb B_g$ and it follows that $X$ is not in the min ball and (i) implies (iv).
\end{proof}


\begin{rem}
 Note that  $\mfBm_g$ is not  exactly a free spectrahedron since it is defined by infinitely many linear matrix inequalities (LMIs). In a forthcoming paper we show using the theory of matrix extreme points that in fact $\mfBm_g$ is not a free spectrahedron.
\end{rem}

By comparison, the \df{max ball}, denoted \df{$\mfBmax_g$}, is the set of $g$-tuples of symmetric matrices  $X=(X_1,\dots,X_g)$ such that $X\in\cDA$ for every  $d$ and $g$-tuple $A$ of symmetric $d\times d$ matrices for which  $\cDA(1)\supseteq \mathbb B_g$.  Like the min ball, the max ball is not presented as a free spectrahedron since it is defined in terms of infinitely many LMIs.  It is described by an LMI when  $g=2$. See Subsection \ref{sec:spin ball} below.

\begin{prop}
 \label{prop:min max duality}
   The min and max balls are polar dual\index{polar dual}\index{dual!polar} in the following sense.  A tuple $X\in \mfBmax_g$ (resp. $\mfBm_g$) if and only if 
\begin{equation}
 \label{eq:min max duality}
  \sum_{j=1}^g X_j\otimes Y_j \preceq I
\end{equation}
 for every $Y\in\mfBm_g$ (resp. $\mfBmax_g$).  Moreover, if $A$ is a $g$-tuple of symmetric $d\times d$ matrices and if $\cDA(1)=\mathbb B_g$, then, for each $n$,
\[
 \mfBmax_g(n) \subset \cDA(n) \subset \mfBm_g(n).
\]
\end{prop}

\begin{proof}
  Recall that $Y\in \mfBm_g$ if and only if $\cD_{L_Y}(1) \supseteq \mathbb B_g$.  Thus $X\in \mfBmax_g$ if and only if $X\in \cD_{L_Y}$ for every $Y\in\mfBm_g$.  Conversely, $X\in \mfBmax_g$ if and only if $X\in \cD_{L_Y}$ for every $Y$ such that $\mathbb B_g \subset \cD_{L_Y}(1)$ (equivalently $Y\in\mfBm_g$). 

 Now suppose $\mathbb B_g =\cDA(1)$. By definition, if $X\in\mfBmax_g$, then $X\in \cDA$ since $\mathbb B_g \subset \cDA(1)$.  On the other hand, if $X\in\cDA(n)$, then,  for each unit vector $v\in\mathbb R^n$, 
\[
  I = v^*v I\succeq  (v^* \otimes I)(\sum X_j\otimes A_j) (v^*\otimes I)  = \sum_j (v^*X_jv) A_j.
\]
 Hence $v^* Xv \in \cDA(1)\subset \mathbb B_g.$  By Lemma \ref{lem:min alts}(iii), $X$ is in the min ball. 
\end{proof}

\begin{rem}\rm
  The use of the term {\it minimal} in {\it min ball} refers not to the size of the spectrahedron itself, but rather by analogy to its use in the theory of operator spaces \cite{Pau}.  In particular, the min ball is defined, in a sense, in terms of a minimal number of positivity conditions; whereas the max ball is defined in terms of the maximal number of positivity conditions (essentially that it should be contained in every free spectrahedron which, at level $1$, is the unit ball). Proposition \ref{prop:min max duality} is a version of the duality between the minimum and maximum operator space structures on a normed vector space \cite{Pau}.
\end{rem}

\subsubsection{The spin ball and the canonical anticommutation relations}
 \label{sec:spin ball}
 Fix a positive integer $g$.  The description of our fourth ball uses the \df{canonical anticommutation relations (CAR)} \cite{GW54, Der}.
 A $g$-tuple $p=(p_1,\dots,p_g)$ of symmetric matrices  satisfies the CARs or is a \df{spin system} \cite{Pisier-book} if 
\[
 p_j p_k + p_k p_j = 2\delta_{jk} I.
\]

 One construction of such a system $P=(P_1,\dots,P_g)$, and the one adopted here, starts with the spin matrices,
\[
 \sigma_1 = \begin{pmatrix} 1 & 0\\ 0 & -1 \end{pmatrix},  \ \ \sigma_2 =\begin{pmatrix} 0 & 1\\1 & 0\end{pmatrix}.
\] 
 For  convenience let $\sigma_0=I_2$. 
 Given $g\ge 2$, each $P_j$ is a $(g-1)$-fold tensor product of combinations of the $2\times 2$ matrices $\sigma_j$ for $0\le j\le 2$. In particular, each $P_j$ is a symmetric matrix of size $2^{g-1}$. Define $P_1= \sigma_1 \otimes \sigma_0 \otimes \cdots \otimes \sigma_0$ and, for $2\le j\le g-1$,
\[
 P_j = \sigma_2\otimes \cdots \otimes \sigma_2 \otimes \sigma_1\otimes \sigma_0 \otimes \dots \otimes \sigma_0,
\]
 where $\sigma_1$ appears in the $j$-th (reading from the left) tensor; thus $\sigma_2$ appears $j-1$ times and $\sigma_0$ appears $g-j-1$ times. Finally, let $P_g$ denote the $(g-1)$-fold tensor product  $\sigma_2\otimes \cdots \otimes \sigma_2$.

The \df{spin ball}, denoted \df{$\mfBc_g$}, is the free spectrahedron determined by the tuple $P$. Thus, $X\in\mfBc_g$ 
if and only if $L_P(X)=I-\sum P_j\otimes X_j\succeq0$ if and only if 
 $\sum_{j=1}^g X_j\otimes P_j$ is a contraction
(cf.~Lemma \ref{lem:P vs -P}\eqref{it:P vs -P}). Further relations between the three balls are explored in the following subsection.

Here are some further observations on the spin ball.  Let
\[
 \sigma_3= \begin{pmatrix} 0  & 1 \\ -1 & 0 \end{pmatrix}
\]
and let
 $$\Sigma = \Sigma_g =\big\{\otimes_{j=1}^{g-1} \sigma_{j_k}: j_k \in \{0,1,2,3\} \big\}.$$
   In particular, the cardinality of $\Sigma_g$ is $4^{g-1}$.

\begin{lemma}
 \label{lem:P vs -P}
If $X$ is a $g$-tuple of $d\times d$ symmetric matrices, then the matrix $ 
\sum_{j=1}^g X_j\otimes P_j$ has the following properties. 
 \begin{enumerate}[\rm (i)]
     \item  
     \label{it:uxpu}
     $$(I\otimes u)^* (\sum_{j=1}^g X_j\otimes P_j) (I\otimes u) = \sum_{j=1}^g \pm X_j\otimes P_j$$
      for each $u\in \Sigma$ with each sequence of $\pm$ assumed $2^{g-2}$ times;
     \item \label{it:skew} 
the sets $\Sigma_g^{\pm}$ consisting of those elements $u$ of $\Sigma_g$ such that 
   $(I\otimes u)^*  (\sum_{j=1}^g X_j\otimes P_j) (I\otimes u) = \pm \sum_{j=1}^g  X_j\otimes P_j$ have $2^{g-2}$ elements; 
   each $u\in \Sigma_g^\pm,$ other than the identity,  is skew symmetric;
        \item \label{it:P vs -P}  $\sum_j X_j\otimes P_j$ is unitarily equivalent to $-\sum_j X_j\otimes P_j$;
     \item \label{it:eigs} $\sum_j X_j\otimes P_j$ has $2d$ eigenvalues (coming in $d$ pairs of $\pm \lambda$ by item \eqref{it:P vs -P}) each with multiplicity $2^{g-2}$;
     \item \label{it:2deigs}
     If $(\sum_j X_j\otimes P_j) \Gamma =0$, then $(\sum_j X_j \otimes P_j) (I\otimes u)\Gamma =0$ for all $u\in \Sigma_g^+\cup \Sigma_g^-$.
 \end{enumerate}
\end{lemma}

\begin{proof}
 We first prove   
\begin{equation}
 \label{eq:sksjsk}
   \sigma_k \sigma_j \sigma_k = \pm \sigma_j
\end{equation} for $0\le j,k\le 3$. Observe  it may be assumed that $1\le j,k\le
 3$. For such $j,k$, with $s\in \{1,2,3\}\setminus \{j,k\}$, evidently
 $\sigma_j\sigma_k = \sigma_s$. Hence, $\sigma_k\sigma_s =\sigma_j$
 (since $j\in \{1,2,3\}\setminus \{k,s\}$) and Equation \eqref{eq:sksjsk} follows.  

 Equation \eqref{eq:sksjsk} immediately implies $u^* P_k u =\pm P_k$
 for $u\in \Sigma_g$ and $1\le k\le g$. Thus $(I\otimes u)^* [\sum_j X_j\otimes P_j](I\otimes u) = \sum_{j=1}^g  \pm X_j\otimes P_j$ 
 for some choice of signs. This proves the first part of item \eqref{it:uxpu}.  The rest of item \eqref{it:uxpu} is established after the
 proof of item \eqref{it:skew}.
 
 Turning to the proof of item \eqref{it:skew}, observe   $u\in \Sigma_g^+$  if and only if $u^* P_k  u =P_k$ for each $1\le k\le g$.
When $g=2$, $P_1=\sigma_1$ and $P_2=\sigma_2$. Evidently, $\Sigma_2^+ =\{\sigma_0\}$ and $\Sigma_2^-=\{\sigma_3\}$.  Now suppose item \eqref{it:skew} holds for $g$.  In this case, letting $\{P_1,\dots,P_g\}$ denote the CAR matrices for $g$, the CAR matrices for $g+1$ are $\{q_1,\dots,q_{g+1}\}$ where $q_1= \sigma_1 \otimes 1$ and $q_j =\sigma_2 \otimes P_{j-1}$ for $j>1$.  If  $u$ is in $\Sigma_g^+$, then $I\otimes u \in \Sigma_{g+1}^+$ and $\sigma_3\otimes u \in \Sigma_{g+1}^-$. Similarly, if $u\in\Sigma_g^-$, then $\sigma_1\otimes u\in\Sigma_{g+1}^+$ and $\sigma_2\otimes u\in\Sigma_{g+1}^-$.   It follows that $\Sigma_{g+1}^-$ has at least $2^{g-1}$ elements and all of these, except for the identity are skew symmetric.  Since $v=\sigma_3\otimes I \in \Sigma_{g+1}^-$, it follows that $vu\in\Sigma_{g+1}^-$ for each $u\in \Sigma_{g+1}^+$ and therefore $\Sigma_{g+1}^-$ has at least $2^{g-1}$ elements too. By induction, item \eqref{it:skew} holds.  Moreover, this argument shows there is a positive integer $N_g$ so that each sign arrangement is taken either $0$ of $N_g$ times. 

We now use induction to show that every sign arrangement is assumed and thus complete the proof of item \eqref{it:uxpu}.  The result  is evident for $g=2$. Now assume it is true for $g$. Let $r=\sigma_3\otimes \sigma_3\otimes I$.  Compute $r^* q_1r= \sigma_3^* \sigma_1 \sigma_3 \otimes (\sigma_3^* I \sigma_3) =-q_1$, but  
\[
 \begin{split}
   r'q_{j+1}r = & \sigma_3^* \sigma_2 \sigma_3 \otimes (\sigma_3\otimes I) P_j (\sigma_3\otimes I)  \\
       = & - \sigma_2 \otimes (-P_j) = q_{j+1}.
 \end{split}
\]
  Thus the combinations of $I\otimes u$ and $r^*(I\otimes u)$ for $u\in \Sigma_g$ produce all sign combinations for $g+1$. 

 Item \eqref{it:P vs -P} is an immediate consequence of item \eqref{it:skew}.

  To prove item \eqref{it:eigs} - namely to see that each eigenvalue has
  multiplicity a multiple of $2^{g-2}$ - observe if $(\sum X_j\otimes  P_j)\Gamma =\lambda\Gamma$, then the set 
  $\{u\Gamma: u  \in\Sigma_g^+\}$ is linearly independent. To verify this last assertion, first note that if $u,v\in\Sigma_g^+$, then $uv\in \pm \Sigma_g^+$.  Further, each $u\in\Sigma_g^+$ is skew symmetric, except for the identity $1.$ In particular, 
\[
 \langle (1\otimes u)\Gamma,\Gamma \rangle =0
\]
for $u\ne 1.$ 
 Hence, if $\sum_{u\in \Sigma_g^+} c_u (I\otimes u)\Gamma =0$, then by multiplying by a $v$ for which $c_v\ne 0$, we can assume $c_1$ (the constant corresponding to the identity $1$) is not zero. Thus,
\[
 0 =\langle \sum c_u (1\otimes u)\Gamma,\Gamma\rangle = c_1 \|\Gamma\|^2
\]
and a contradiction is obtained.   To complete the proof, let  $m$ is the largest integer such that there exists
  $\Gamma^1,\dots\Gamma^m$ such that $\{u\Gamma^k : u \in \Sigma_g^+,  \ \ 1\le k \le m\}$ spans the eigenspace
  corresponding to eigenvalue $\lambda$, then the dimension of this
  space is $m 2^{g-2}$.  We prove this assertion using induction. Suppose $1\le k\le m$ and $S_k:=\{u\Gamma^j : u\in \Sigma_g^+, \, 1\le j\le k\}$ is linearly independent (and thus has dimension $k2^{g-2}$).  Arguing by contradiction, suppose the $\Delta$ is in the intersection of $S_k$ and $\{u\Gamma^{k+1}: u\in\Sigma_g^+\}$.  From what is already been proved, the dimension of the span of $\{u\Delta: u\in\Sigma_g^+\}$ is $2^{g-2}$ but this subspace is a subspace of both the span of $S_k$ and the span of $\{u\Gamma^{k+1}: u\in\Sigma_g^+\}$, contradicting the minimality of $m$.  Hence the dimension of the span of $S_{k+1}$ is $(k+1)2^{g-2}$ and the proof is complete.

 Item \eqref{it:2deigs} is evident.
\end{proof}

\begin{lemma}
  \label{lem:minp}
     The mapping from $\mathbb R^{g}$ (in the Euclidean norm) to $M_{2^{g-1}}(\mathbb R)$ (in the operator norm),
\[
 x\mapsto \sum_{j=1}^g x_j P_j
\]
 is an isometry.  In particular, $\mfBc_g(1)=\mathbb B_g$.
 \end{lemma}

 \begin{proof}
 The result follows from $(\sum_{j=1} x_j P_j)^2 = (\sum_j x_j^2) I  + \sum_{j<k} (x_j x_k - x_k x_j) P_jP_k= (\sum x_j^2) I$.   \end{proof}

 A consequence of Lemma \ref{lem:minp} is that the tuple $(P_1,\dots,P_{g})$ is in $\mfB^{\text{min}}_{g}$.

\subsection{Inclusions and dilations}

In this subsection we investigate inclusions between the different types
of balls introduced above.

\begin{lemma}
 \label{lem:normp}
   The norm of $\sum_{j=1}^{g} P_j\otimes P_j$  is $g$.  Hence $P$ is in the (topological)  boundary of $g \mfBc_g$. 
   The norm of the block row matrix $\begin{pmatrix} P_1 & \cdots & P_{g}\end{pmatrix}$ is $\sqrt{g}$. Thus $P$ is in the boundary of $\sqrt{g}\mfBoh_g$.
\end{lemma}

\begin{proof}
 We prove a bit more. Let $\{e_0,e_1\}$ denote the standard basis of $\mathbb R^2$.  In particular, $\sigma_0 e_j = e_j$, $\sigma_1 e_j = (-1)^j e_j$ and $\sigma_2 e_j = e_{j+1}$ (modulo $2$).  For convenience, let $h=g-1$. Given $\alpha =(\alpha_1,\dots,\alpha_{h}) \in \mathbb Z_2^{h}$, let
\[
 e_\alpha = e_{\alpha_1}\otimes \dots \otimes e_{\alpha_{h}} \in \mathbb R^{2^h}
\]
and 
\[
 \gamma = \sum_{\alpha\in\mathbb Z_2^h} e_{\alpha}\otimes e_{\alpha}.
\]
 Now we verify that, for $1\le j \le g$, 
\begin{equation}
 \label{eq:pjpj}
  P_j\otimes P_j \gamma = \gamma.
\end{equation}
 Indeed, $P_1 e_\alpha = (-1)^{\alpha_1}e_\alpha$ and hence $P_1\otimes P_1 e_\alpha \otimes e_\alpha  = e_\alpha \otimes e_\alpha$.  For $2\le j\le h$, 
\[
 \begin{split}
 P_j e_\alpha & =  e_{\alpha_1}\otimes \cdots \otimes e_{\alpha_h}\\
    & =   e_{\alpha_1+1} \otimes \cdots \otimes e_{\alpha_{j-1}+1} \otimes (-1)^{\alpha_j} e_{\alpha_j} \otimes e_{\alpha_{j+1}} \otimes \cdots \otimes e_{\alpha_h}\\
   &= (-1)^{\alpha_j} e_{\beta},
  \end{split}
\]
 where $\beta=(\alpha_1+1,\dots,\alpha_{j-1}+1,\alpha_j,\dots,\alpha_h)$.
Thus,
\[
 (P_j \otimes P_j )(e_\alpha\otimes e_\alpha) = e_\beta\otimes e_\beta
\]
 and the conclusion $P_j\otimes P_j \gamma =\gamma$ for $1\le j \le h$ follows. The argument that $P_g\otimes P_g \gamma=\gamma$ is similar.

It follows that $(\sum_j P_j\otimes P_j) \gamma = g \, \gamma$ and hence the norm of $\sum P_j\otimes P_j$ is at least $g$. Since the norm of each $P_j$ is one, the norm of $\sum P_j\otimes P_j$ is at most $g$.   The remainder of the lemma is evident.
\end{proof}

\begin{cor}
 \label{cor:g-sharp}
    The smallest $\rho$ such that $\mfBm_g\subset \rho \mfBc_g$ is $\rho=g$.  In particular, the estimate $g$ in Proposition \ref{prop:small g} is sharp.
\end{cor}

\begin{proof}
 Suppose $\mfBm_g\subset \rho\mfBc_g$. By Lemma \ref{lem:minp}, the tuple $P$ coming from the CAR is in $\mfBm_g$ and by Lemma \ref{lem:normp}, $P$ is in the boundary of $g\mfBc_g$. Thus $\rho\ge g$. On the other hand, if $X\in\mfBm_g$, then, since $\mfBm_g$ satisfies the hypotheses of Proposition \ref{prop:small g}, $X=  g V^* TV$, where $V$ is an isometry and $T$ is a commuting tuple of self adjoint matrices with spectrum in $\mathbb B_g$. In particular, $T\in \mfBc_g$ and thus $\frac{1}{g} X \in \mfBc_g$ too.  Hence $\mfBm_g\subset g\, \mfBc_g$.  Further, if the estimate in Proposition \ref{prop:small g} were not sharp, the argument just given would produce a $\rho<g$ such that $\mfBm_g \subset \rho \mfBc_g$, a contradiction. 
\end{proof}

\begin{thm}\rm
 \label{conj:wild versus spin}
 The smallest $\rho$ such that $\mfBoh_g$ embeds into $\rho \mfBc_g$ is   $\sqrt{g}$.
\end{thm}

\begin{proof}
 By Lemma \ref{lem:normp}, 
$P$ is in the topological boundaries of $\sqrt g\mfBoh_g$ and  $g\mfBc_g$. Hence $\rho\geq\sqrt g$.

 To prove the converse inequality, we use complete positivity to show
 $\mfBoh_g\subseteq\sqrt g\, \mfBc_g$. We follow the solution \cite{HKM13}  to the
 free spectrahedral inclusion problem as described in Subsection \ref{subsub:spIncl}. Let $e_{i,j}$ denote the $(g+1)\times(g+1)$ matrix units. Letting $A_i=e_{1,i+1}+e_{i+1,1}$ for $i=1,
 \ldots,g$, and $A=(A_1,\ldots,A_g)$, we have $\mfBoh_g=\cD_{A}$. Similarly,
 $\mfBc_g=\cD_{P}$, where $P=(P_1,\ldots, P_g)$. It thus suffices to show
 there is a unital completely positive map
 \[
\psi: e_{1,i+1}+e_{i+1,1} \mapsto \frac1{\sqrt g} P_i, \qquad i=1,\ldots,g.
 \]
Consider the following ansatz for the Choi matrix for $\psi$:
\beq\label{eq:Cpsi}
C_{\psi}= 
\begin{pmatrix} 
\frac 12 I & \frac1{2\sqrt g} P_1 & \cdots & \frac1{2\sqrt g} P_g \\
\frac1{2\sqrt g} P_1 \\ 
$\vdots$ & &  S \\
\frac1{2\sqrt g} P_g
\end{pmatrix}
\eeq
Set 
\[
S =  \frac1{2 g} \begin{pmatrix} P_1 \\ \vdots \\ P_g \end{pmatrix}
\begin{pmatrix} P_1 \\ \vdots \\ P_g \end{pmatrix}^* 
=
\frac1{2g} \begin{pmatrix} I & P_1 P_2 & \cdots & P_1 P_g \\
P_2P_1 & I & \cdots  & P_2P_g \\
\vdots & \ddots & \ddots & \vdots \\
P_g P_1  & \cdots & \cdots & I
\end{pmatrix}
\]
It is clear that $S\succeq0$. Furthermore, the Schur complement of the top left block  of $C_\psi$ from \eqref{eq:Cpsi} is $0$. Thus $C_\psi$ is positive semidefinite. Furthermore, $\frac12I+g \frac1{2g}I=I$, whence $\psi$ is unital.
\end{proof}

\begin{prop}
 \label{prop:min v wild}
 The smallest $\rho$ such that $\mfBm_g$ embeds into $\rho \mfBoh_g$ is $\rho=\sqrt{g}$.
\end{prop}

\begin{proof}
 The tuple $P$ is in $\mfBm_g$, but 
 is, by Lemma \ref{lem:normp}, in the topological boundary of $\sqrt{g} \mfBoh_g$.  Thus $\rho\ge \sqrt{g}$.  On the other hand, if $X=(X_1,\dots,X_g)$ is in $\mfBm_g$, then each $X_j$ is a contraction. Hence the norm of the row matrix $X=\begin{pmatrix} X_1 & \dots & X_g \end{pmatrix}$ is at most $\sqrt{g}$; i.e., $X\in\mfBoh_g$.
\end{proof}

\begin{prop}
 \label{prop:2spin}
 A $2$-tuple $X$ is in the spin ball $\mfBc_2$ if and only if it dilates to a commuting $2$-tuple $T$ of symmetric matrices (an upper bound on the size of the matrices in $T$ in terms of $g$ and $d$ can be deduced from the proof)  with joint spectrum in $\mathbb B_2$.
\end{prop}

Before giving the proof of Proposition \ref{prop:2spin} let us note a few consequences.

\begin{cor}\label{cor:spinBallembed}
  $\mfBc_2 = \mfBmax_2$. In particular, for a given $2$-tuple $A$ of $d\times d$ matrices, $\mathbb B_2\subset \cDA(1)$ if and only if $\mfBc_2 \subset \cDA$.
  Finally, $Y\in \mfBm_2$ if and only if there exists a positive integer $\mu$ such that $Y$ dilates to $I_\mu \otimes P$. 
\end{cor}

\begin{proof}
  By Proposition \ref{prop:min max duality}, $\mfBmax_2 \subset \mfBc_2$.  Thus it remains to show if $X\in\mfBc_2$, then $X\in \mfBmax_2$. By Proposition \ref{prop:2spin}, there exists a commuting pair $T$ of symmetric matrices and an isometry $V$ such that $X_j = V^* T_j V$ and the joint spectrum of $T$ is in $\mathbb B_2$.  It follows that $T\in \mfBmax_2$ and thus $X\in \mfBmax_2$ too. 

 To prove second statement, note that, by definition of the max ball, if $\mathbb B_2 \subset \cDA(1)$, then $\mfBmax_2 \subset \cDA$. The converse is automatic since $\mfBc_2(1)=\mathbb B_2$.  Thus the second part of the corollary follows immediately from the first.

 The final part of the corollary follows easily from \cite{HKMpd}. Here
 is a sketch.
 Let $\mathcal P$ denote the span of $\{I,P_1,\ldots,P_k\}$.  Suppose $Y\in \mfBm_2$ and let $\mathcal Y$ denote the span of $\{I,Y_1,Y_2\}$. By Proposition \ref{prop:min max duality} and what has already been proved, if $X\in\mfBc_2$, then $I\succeq \sum_{j=1}^2 X_j\otimes Y_j$. Hence the unital mapping $\varphi:\mathcal P\to \mathcal Y$ defined by $\varphi(P_j)=Y_j$ is completely positive.  The dilation conclusion follows.  
 Conversely,  if $Y$ dilates to $I_\mu\otimes P$, then evidently $Y$ is in $\mfBm_2$ and the proof is complete.
\end{proof}

\begin{rem}[Matrix Ball Problem]\index{matrix ball problem}
Given a $d\times d$ monic linear pencil $L$ consider the problem of embedding
the unit ball $\mathbb B_g$ into the spectrahedron $\mathscr S_L=\cD_L(1)$.
Equivalently (Corollary \ref{cor:spinBallembed}),
consider embedding $\mfBc_g$ into $\cD_L$.
Both objects are free spectrahedra, so the complete positivity
machinery on embeddings of free spectrahedra applies 
(see Subsubsection \ref{subsub:spIncl} or \cite{HKM13} for details). That is,
$\mathbb B_g\subseteq\mathscr S_L=\cD_L(1)$ is equivalent to an explicit LMI of  size  $2^{g-1}gd$.
\end{rem}

 The proof of Proposition \ref{prop:2spin} uses the following proposition.

\begin{prop}
 \label{prop:enough}
   A tuple $X\in\mfBc_2(d)$ is an extreme point\index{extreme point} of $\mfBc_2(d)$ if and only if it is a commuting tuple and   $\sum_{j=1}^g X_j\otimes P_j$ is unitary. 
\end{prop}

\begin{proof}
 For notational ease, let $\Lambda(x) = \sum_{j=1}^g P_jx_j$, with the dependence on $g$ supressed.
Observe that, by Lemma \ref{lem:P vs -P}, a tuple  $X\in \mfBc_g$ if and only if $\Lambda(X)^2 \preceq I$.  Equivalently, $X$ is in the spin ball if and only if
$L(X):= I - \Lambda_* (X) \succeq 0,$  where
\[
\Lambda_*(X) = \Lambda(X) \oplus -\Lambda(X) =\begin{pmatrix} \Lambda(X) & 0 \\ 0 & -\Lambda(X)\end{pmatrix}.
\]
 In the case of $g=2$ and the tuple $X=(X_1,X_2)$,
\[
 \Lambda(X) = \begin{pmatrix} X_1 & X_2 \\ X_2 & -X_1 \end{pmatrix}
\]
 and $X$  is in the spin ball if and only if 
\[
 I\succeq  \Lambda(X)^2  = \begin{pmatrix} X_1^2 + X_2^2 & X_1 X_2 - X_2 X_1 \\ -(X_1 X_2 - X_2 X_1)  & X_1^2 + X_2^2 \end{pmatrix}.
\]

 Suppose this inequality holds.  If $\gamma \in\mathbb R^d$ and $(X_1^2+X_2^2) \gamma = \gamma$, then evidently $(X_1 X_2 - X_2 X_1)\gamma =0$.  Let 
\[
  \cH =\{\gamma \in\mathbb R^d :  (X_1^2+X_2^2) \gamma = \gamma \}.
\]
 If $\cH=\mathbb R^d$, then $X_1$ and $X_2$ commute. Hence, we may assume $\cH$ is a proper subspace of $\mathbb R^d$. Let $\cK=\mathbb R^d \ominus \cH$.

 Let $\mathscr E_{\pm}$ denote the nullspace  of  $I\mp \Lambda(X)$. In particular, $\mathscr E$, the nullspace of $I-\Lambda(X)^2$ is the direct sum $\mathcal E_+\oplus\mathcal E_-$.  Let $\{e_0,e_1\}$ denote the standard basis for $\mathbb R^2$.  If  $\gamma_0 \in \cK$ and $\gamma_1\in\mathbb R^d$ and $\Gamma = \sum_{j=0}^1 \gamma_j \otimes e_j \in \mathscr E_\pm$, then, since $(1\otimes \sigma_3)\Lambda(x) (1\otimes \sigma_3) = -\Lambda(x)$ (see Lemma \ref{lem:P vs -P}), 
\[
 \Lambda(X)  (I\otimes \sigma_3) \Gamma  = \mp (I\otimes \sigma_3)\Gamma,
\]
 and hence  $(I\otimes \sigma_3)\Gamma$ lies in $\mathscr E_{\mp}$.   It follows that $I\otimes \sigma_3$ leaves $\mathscr E$ invariant. Let $Q$ denote the projection of $\mathbb R^d$ onto $\cK$.  If $\Gamma \in \mathscr E$, then $((I-Q)\otimes I)\Gamma \in \cH\otimes \mathbb R^2 \subset \mathscr E$. Hence $(Q\otimes I)\Gamma\in\mathscr E$. Consequently, we can assume there is a nonzero  $\Gamma\in\mathscr E \cap (\cK\otimes \cK)$.  Write, as before $\Gamma = \sum_{j=0}^1 \gamma_j \otimes e_j$.   To see that $\{\gamma_0,\gamma_1\}$ spans a two dimensional space $\mathcal S$ (is a linearly independent set), suppose not. In that case, $\Gamma = \gamma \otimes e$ for vectors $\gamma \in \cK$ and $e\in\mathbb R^2$.  Since $\sigma_3$ is skew symmetric (as is $X_1X_2-X_2X_1$), 
\[
 \Big\langle \big[(X_1X_2-X_2X_1)\otimes \sigma_3\big] \Gamma,\Gamma \Big\rangle = \big\langle (X_1X_2-X_2X_1)\gamma,\gamma\rangle \, \langle \sigma_3 e,e\big\rangle =0.
\]
 Thus, as $\Gamma =\Lambda(X)^2\Gamma$, 
\[
 \begin{split}
 \|\gamma\|^2 \, \|e\|^2 & =  \|\Gamma\|^2\\
  &=   \langle \Gamma,\Gamma \rangle\\
  &=  \langle \Lambda(X)^2\Gamma,\Gamma\rangle \\
  &=  \langle (X_1^2+X_2^2)\gamma \otimes e + (X_1X_2-X_2X_1)\gamma \otimes e, \gamma \otimes e\rangle \\
  &=  \langle (X_1^2+X_2^2)\gamma,\gamma \rangle \, \|e\|^2.
 \end{split}
\]
 Since $I\succeq X_1^2+X_2^2$ it follows that $(X_1^2+X_2^2)\gamma =\gamma \in H$. Hence $\gamma =0$. 

 Both $\Gamma$ and $(I\otimes \sigma_3)\Gamma \in \mathscr E\cap (\mathcal S\oplus \mathcal S)$. On the other hand, if $\{\Gamma,(I\otimes \sigma_3)\Gamma\}$ don't span $\mathscr E\cap (\mathcal S\oplus \mathcal S)$, then it is easily shown that  this intersection contains an element of the form $\gamma\otimes e$ and we obtain a contradiction as before. Thus $\mathscr E\cap (\mathcal S \oplus \mathcal S)$ is spanned by $\{\Gamma,(I\otimes \sigma_3)\Gamma\}$. 

 Define $Y=(Y_1,Y_2)$ on $\mathcal S$ as follows. Let $Z$ be a generic  $4\times 4$ matrix.  We will choose $Z$ so that it has the  form $\Lambda(Y)$  and such that $\Lambda(Z)$ is zero on $\mathscr E \cap (\mathcal S\oplus \mathcal S)$.  There are $16$ free variables in $Z$.   Insuring $Z$ has the proper form with $Y_j=Y_j^*$ consists of $2+2\cdot4=10$ homogeneous equations. The condition $Z\Gamma =0$ is another $4$ homogeneous equations. The form of $Z$ and $Z\Gamma =0$ implies $Z(I\otimes \sigma_3)\Gamma =0$ too. Hence we are left with $2$ free variables and a choice of $Y\ne 0$ exists.  And $\Lambda(Y) \mathcal E\cap (\mathcal S \oplus \mathcal S)=0$ for such a choice of $Y$.

Extend the definition of $Y$ to all of $\mathbb R^d$ by declaring $Y_j=0$ on $\mathcal S^\perp$. It follows that $Y\ne 0$ and at the same time $\Lambda(Y)$ is $0$ on $\mathscr E$ and hence on each of $\mathscr E_{\pm}$. With respect to the decomposition of the space that $\Lambda(X)$ acts upon as $\mathscr E_+ \oplus \mathscr E_-\oplus H$,
\[
 \Lambda(X) = \begin{pmatrix} I & 0 & 0 \\0 & -I & 0\\ 0 & 0 & X^\prime \end{pmatrix}, \ \ 
 \Lambda(Y) = \begin{pmatrix} 0 & 0 & 0 \\ 0& 0 & 0 \\ 0& 0& Y^\prime \end{pmatrix},
\]
 where $X^\prime,Y^\prime$ are self-adjoint and $X^\prime$ is a strict contraction. It follows, by choosing $t$ small enough, that $I\pm \Lambda(X+tY) \succeq 0$
 and thus $X$ is not an extreme point of $\mfBc_2$. 
\end{proof}

\begin{proof}[Proof of Proposition \ref{prop:2spin}] If $X$ in the spin ball $\mfBc_2(d)$, then by Caratheodory's Theorem there exists an $N$, extreme points $X^1,\dots X^N$ of $\mfBoh_2(d)$ and scalars $0\le t_1,\dots, t_N$ such that $\sum t_j =1$ and $X=\sum t_j X^j$. Let  $T=\oplus X^j$ act on $H=\oplus_1^N \mathbb R^d$ and define $V:\mathbb R^d\to H$ by $Vh = \oplus \sqrt{t_j}h$. Thus $V$ is an isometry and it is evident that $V^*TV =\sum t_j X^j =X$. By Proposition \ref{prop:enough}, $T$ is a commuting tuple of symmetric matrices.
\end{proof}

\subsubsection{An alternate proof of Proposition \ref{prop:2spin}}
 An ad hoc proof of Proposition \ref{prop:2spin} is based upon the Halmos dilation\index{dilation!Halmos} of a contraction matrix to a unitary matrix.

\begin{lemma}
 \label{lem:shape of S}
  Suppose $u,v,a,b,d$ are $n\times n$ matrices and let
\[
  R=\begin{pmatrix} a & b\\ b^* & d\end{pmatrix}
\]
  If 
 $R$ is positive semidefinite, and 
\[
  R^2 = Z:=\begin{pmatrix} u & v \\ -v & u \end{pmatrix},
\]
 then $a=d$ and $b^*=-b$.
\end{lemma}

\begin{proof}
Note that $U^* R^2 U =R^2$, where $U$ is the unitary matrix
\[
 U =\frac{1}{2} \begin{pmatrix} I & I \\ -I & I \end{pmatrix}.
\]
 Using the functional calculus  it follows that $U^*RU=R$ too.  From this relation, direct computation reveals
\[
 \begin{split}
   a - (b+b^*)+d& = 2a\\
   a + (b-b^*)-d &= 2b
 \end{split}
\]
 from which it follows that $b+b^*=0$ and $a=d$.
\end{proof}

\begin{proof}[Proof of Proposition \ref{prop:2spin}]
 Let 
\[
  S=  \begin{pmatrix} X & Y \\ Y & -X \end{pmatrix}.
\]
 Almost by definition, $X\in\mfBc_2$ means $S$ is a a contraction. Let $D=(I-S^2)^{\frac 12}$, the defect of $S$. By Lemma \ref{lem:shape of S}, 
\[
 D= \begin{pmatrix} d & e \\ -e & d \end{pmatrix}.
\]
 In particular, $e$ is skew symmetric,  $e^*=-e$.
 The operator 
\[
 U = \begin{pmatrix} S & D \\ D & - S\end{pmatrix}
\]
 is unitary and consequently
\beq
 \begin{split}\label{eq:julia}
   X^2+Y^2 + d^2 -e^2 & =  I \\
   XY-YX + de+ed &=  0\\
   Yd+Xe+eX-dY &=  0\\
   Xd-Ye-dX-eY &=  0.
 \end{split}
\eeq
 
  Let
\[
 T_1 = \begin{pmatrix} X & e \\ -e & X \end{pmatrix}, \qquad
 T_2 =  \begin{pmatrix} Y & d \\ d & -Y \end{pmatrix}.
\]
 Compute,  using \eqref{eq:julia},
\[
 T_1 T_2 - T_2 T_1 = \begin{pmatrix} XY+ed -YX + de & Xd-eY - Ye-dX \\ -eY+Xd - dX-Ye & -ed-XY -de +YX  \end{pmatrix} = 0.
\]
 Likewise,
\[
 T_1^2 +T_2^2 = \begin{pmatrix} X^2-e^2 +Y^2 +d^2 & Xe+eX+Yd-dY \\ -eX-Xe+dY-Yd & X^2-e^2+Y^2 +d^2 \end{pmatrix} = I. \qedhere
\]
\end{proof}

\section{Probabilistic Theorems and Interpretations continued}
\label{sec:shortprob}

This section follows up on Section \ref{sec:introprob},
 adding a few more probabilistic facts and summarizing
 properties involving equipoints.
We follow the conventions of Section 
\ref{sec:introprob}.
In particular,  for $\head,\tail \in \RR$ with  $\dd=\head + \tail >0$
 the equipoint  $\sih$ is defined by 
 \beq
\label{eq:probEqui}
P^{b(\head +1 , \tail)} (\BRV\leq \sih)  \ = \
  \Pbe (\BRV\geq \sih). 
  \eeq
 
\subsection{The nature of equipoints}
 \label{sec:equinatureAGAIN}
 Here are basic properties of
equipoints versus  medians.\index{equipoint}\index{median}\index{beta distribution}\index{binomial distribution}

\begin{prop}
\label{prop:binbeta}
Various properties of 
the distributions
${\rm Bin}(\dd,p)$  and ${\rm Beta}(\head, \tail)$
are:
\ben[\rm(1)]
\item
\label{it:equis}
Bin and Beta: \ The equipoint exists and is unique.
\item
\label{it:eqimed}
Bin: \ Given $ \head,$ if  $\sih$ is an equipoint, then
$ \head $ is a median
for ${\rm Bin}(\dd,\sih)$.

\item
\label{it:eqiSym} Bin:\ 
For even $\dd$ and  any integer  $ 0 \leq k \leq \frac\dd2 $,
$$
P_{ \sigma_{ ( \frac \dd 2 + k )} }  \left (\SRV= \frac \dd 2 + k\right)  
= 
P_{ \sigma_{( \frac \dd 2  - k) }  }   \left(\SRV=  \frac \dd 2 - k\right).
    $$
Also we have the symmetry    
   $$P_{\frac \dd 2 + k }\left(\SRV= \frac \dd 2 + k\right)  
   = P_{\frac  \dd 2 - k} \left(\SRV=  \frac \dd 2 - k\right).
    $$
\een
\end{prop}

\begin{proof}
\eqref{it:equis}
Note that for fixed integer $ \head $, the function
$P_p(\SRV\geq  \head )$ is increasing from $p=0$ to $p=1$.
Likewise 
$P_p(\SRV\leq  \head )$ is decreasing from $p=0$ to $p=1$.
The graphs  are continuous, so must cross at a unique point,
namely at $\sih $.
Likewise, $P^{b(\head+1, \tail)}(\BRV \leq p)$  increases from $0$ up to 1
while  $\Pbe(\BRV \geq p)$ decreases 
from 1 down to 0.

\eqref{it:eqimed} \
 Fix $\dd, s$, hence $ \sih $. 
 Then by the definition of $\sih $, we have
\begin{equation}
\label{eq:isMed}
 1= 2 P_{\sih } (\SRV< \head) +P_{\sih } (\SRV= \head ).
 \end{equation}
If  $P_{\sih } (\SRV< \head ) +  P_{\sih } (\SRV= \head ) < \frac1 2 $
then 
 $P_{\sih } (\SRV> \head ) +P_{\sih } (\SRV =\head ) <\frac 1 2$
 which contradicts \eqref{eq:isMed}.
Thus 
$\head$ is a median.

\eqref{it:eqiSym}
The symmetry is 
seen by switching the roles of heads and tails:
$$P_{p}(\SRV=  \head )  
= P_{1-p}(\SRV=  \dd-  \head ).
    $$
Then note    
$\dd - ({\frac \dd 2 + \frac k\dd}) = 
{\frac \dd 2 -\frac k\dd} $.
\end{proof}

\def\hF{\hat F}

\ssec{Monotonicity}
\label{sec:monot}

 For $\dd\in \mathbb R_{>0}$ fixed  recall the functions
\beq
\Phi(\head):=   P^{b(\head,\dd-\head+1)} (\BRV\leq  e_{\head, \dd - \head   +1})
\qquad \text{and} \qquad
\hat \Phi(\head):=   
P^{b(\head,\dd -\head  +1)}
\left(\BRV\leq \frac \head {\dd} \right) 
\eeq
based on the CDF of the Beta Distribution.
 The proof of {\it one step monotonicity} 
 of these functions claimed in Theorem \ref{thm:incr}
 from  Section \ref{sec:introprob} is proved below
  in Subsubsection \ref{sec:CDFmonot}. A similar result
 with the CDF replaced by the PDF is established in
  Subsubsection \ref{sec:PDFmonot}.

\sssec{Monotonicity of the CDF}
\label{sec:CDFmonot}

\begin{proof}[Proof of Theorem {\rm\ref{thm:incr}}]
(1) The claim is that $ \Phi(\head ) \leq    \Phi(\head + 1) $ 
for $\head, \dd \in \NNhalf$ and $ \frac \dd 2 \leq \head < \dd -1 $. Recall Lemma \ref{lem:fisg}
which says that
$f_{s,t}(\sigma)$ defined in \eqref{prob:opt'}
when evaluated at the equipoint is
$$
f_{s,t}(\sist) =
2 \;
I_{\sist}\left(\frac s2,1+\frac t2\right) 
 -1 
$$
Using the conversion $\head = \frac s 2$, we get
$\ds \Phi(\head) =  \frac{ f_{s,t}(\sist)  +1 } {2}.
$
Proposition \ref{prop:f2step} gives two step monotonicity
of $f_{s,t}(\sist)$ when $s\geq t$
which implies $\Phi$ is one step monotone for 
$\head \geq \tail$.

 (2) We claim that $\hat \Phi(\head ) \leq   \hat \Phi(\head + 1) $ 
for $\head,\dd \in \RR$ with $ \frac \dd 2 \leq \head  <\dd -1 $.
\\
Define
$\hF$ by 
\[ 
\hat{F} (\dd,\head) = 
\frac{P^{b(\head, \tail +1)}(\BRV\leq \frac {\head}{\dd} )}
{\Gamma\left( \dd+1\right) }
=
\frac{I_{\frac {\head}{\dd}}\left(\head, \tail +1\right)}{\Gamma\left( \dd+1\right) }
=
 \frac{\ds \int_0^{\frac {\head}{\dd}} x^{\head-1}(1-x)^{\tail}\,dx }
{\Gamma\left(\tail+1\right) \Gamma\left( \head\right) }.
\]
for $\head+\tail=\dd$.
Now we show that for
$\frac {\dd} 2 \leq \head \leq \dd-1$
we have
$\hat{F}(\dd,\head+1)\ge \hat{F}(\dd,\head)$,
equivalently
$\ds\frac {\hat{F}(\dd,\head+1)} { \hat{F}(\dd,\head)} \geq 1$.

We start by simplifying this quotient:
\[
\begin{split}
\frac {\hat{F}(\dd,\head+1)} { \hat{F}(\dd,\head)} &= 
\frac{\ds \int_0^{\frac {\head +1}{\dd}} x^{\head}(1-x)^{\tail-1}\,dx \;
\Gamma\left(\tail+1\right) \Gamma\left( \head\right)}
{\ds \Gamma\left(\tail\right) \Gamma\left( \head+1\right) \int_0^{\frac {\head}{\dd}} x^{\head-1}(1-x)^{\tail}\,dx} \\
&= 
\frac {\tail}{\head}
\frac{\ds \int_0^{\frac {\head +1}{\dd}} x^{\head}(1-x)^{\tail-1}\,dx }
{\ds \int_0^{\frac {\head}{\dd}} x^{\head-1}(1-x)^{\tail}\,dx} .
\end{split}
\]
Thus $\ds\frac {\hat{F}(\dd,\head+1)} { \hat{F}(\dd,\head)} \geq 1$
is equivalent to 
\beq\label{eq:Fs+3}
\tail
\int_0^{\frac {\head +1}{\dd}} x^{\head}(1-x)^{\tail-1}\,dx 
\geq
\head
 \int_0^{\frac {\head}{\dd}} x^{\head-1}(1-x)^{\tail}\,dx,
\eeq
so it  suffices to prove
\beq\label{eq:Fs+4p}
\tail
\int_{\frac {\head }{\dd}}^{\frac {\head +1}{\dd}}  x^{\head}(1-x)^{\tail-1}\,dx 
\geq
\head
 \int_0^{\frac {\head}{\dd}} x^{\head-1}(1-x)^{\tail}\,dx
- \tail 
\int_0^{\frac {\head }{\dd}}  x^{\head}(1-x)^{\tail-1}\,dx .
\eeq

As
\[
\left( x^{\head} (1-x)^{\tail}\right)' = 
\head
x^{\head-1}(1-x)^{\tail}\,dx
- \tail 
x^{\head}(1-x)^{\tail-1},
\]
the right-hand side of \eqref{eq:Fs+4p} equals
\[
\frac{\head^{\head} \tail^{\tail}}{\dd^\dd}.
\]

Letting $\eta(x):=x^{\head}(1-x)^{\tail-1}$, we see
\[
\eta'(x)= x^{\head-1} (1-x)^{\tail-2} (-x \dd +x+\head),
\]
so $\eta(x)$ is increasing on 
$\ds\left[0,\frac{\head}{\dd-1}\right]$
and decreasing on 
$\ds\left[\frac{\head}{\dd-1},1\right]$.
Since $\head\leq\dd-1$, we have $\ds\frac{\head}{\dd-1}\in
\left[\frac{\head}{\dd},\frac{\head+1}{\dd}\right]$.
We claim that
\beq\label{eq:Fs+4pp}
\eta\left(\frac{\head}{\dd}\right) \leq
\eta\left(\frac{\head+1}{\dd}\right) .
\eeq
Indeed, \eqref{eq:Fs+4pp} is easily
seen to be equivalent to
\[
\left(1+\frac1\head\right)^\head\geq\left(1+\frac1{\tail-1}\right)^{\tail-1},
\]
which holds since $\head\geq\tail$.

We can now apply a box inequality on the left-hand side of
\eqref{eq:Fs+4p}:
\[
\tail
\int_{\frac {\head }{\dd}}^{\frac {\head +1}{\dd}}  x^{\head}(1-x)^{\tail-1}\,dx 
\geq 
\tail \ \frac1\dd \ \eta\left(\frac{\head}{\dd}\right)
=
\frac{\head^{\head} \tail^{\tail}}{\dd^\dd},
\]
establishing \eqref{eq:Fs+4p}.
\end{proof}

Ideas in the paper \cite{PR07} 
were very helpful in the proof above.

\subsubsection{Monotonicity of the PDF}
\label{sec:PDFmonot}
So far we have studied the  CDF of the Beta Distribution.
However, the functions
\begin{equation}
 \label{eq:reform}
P_{ \eih }   (\SRV=  \head )
\qquad \text{and}\qquad 
 P_\sdd   (\SRV=  \head )
\end{equation}
based on  PDF's of the Binomial distribution also have monotonicity properties
for integer  $ \dd/2 \leq \head \leq \dd$.

\begin{prop}
Let $\dd\in\N$. 
For integer $ \head  \geq \frac \dd 2$, we have that 
\ben[\rm(1)]
\item
$P_\sdd( \SRV=  \head )$
is increasing;
its minimum is
$
P_{ \sdd }   (\SRV=  \lceil \dd/2 \rceil)$;
\item
$
P_{\eih }   (\SRV=  \head )$
is increasing; its minimum is
$
P_{ \sigma_{  \head } }   (\SRV=  \lceil \dd/2 \rceil)$.
\een
\end{prop}

\def\hPhi{{\hat \Phi}}

\begin{proof}
(2) By the definition of $\sih $, we have \
  $
  P_{\sih } (\SRV= \head ) = \ 2 P_{\sih } (\SRV\leq \head) -1.
$
Theorem \ref{thm:incr} implies the required monotonicity. 

(1)
Recall
\[
P_{ \sdd }   (\SRV=  \head) = {\dd\choose\head}
\frac{\head^\head\tail^\tail}{\dd^\dd}.
\]
Thus
\[
\frac{P_ {\frac{\head+1}\dd }   (\SRV=  \head+1) }{P_{ \sdd }   (\SRV=  \head) }
=
\left(\frac{\head+1}\head\right)^\head\ 
\left(\frac{\tail-1}\tail\right)^{\tail-1}
\]
is $\geq1$ iff
\beq\label{eq:imdone}
 \left(1+\frac1\head\right)^\head\geq\left(1+\frac1{\tail-1}\right)^{\tail-1}.
\eeq
Since $\head>\tail-1$, \eqref{eq:imdone} holds,
establishing the monotonicity of $P_{ \sdd }   (\SRV=  \head) $.
\end{proof}

\newpage

{\small
\linespread{1.1}

}

\end{document}